\pdfoutput=1

\newcommand{\showdate}{false}

\documentclass[reqno,a4paper]{amsart}
\usepackage{microtype}
\usepackage{amssymb}
\usepackage[utf8]{inputenc}
\usepackage{a4wide}
\usepackage{enumitem}
\usepackage{pstricks}
\usepackage{pst-pdf}
\usepackage{multido}
\usepackage{ifthen}
\usepackage{comment}
\usepackage[all]{xy}
\usepackage{booktabs}
\usepackage{url}
\usepackage[pdftitle={Extra-twisted connected sum G2-manifolds},hidelinks]{hyperref}
\RequirePackage{xspace}

\title{Extra-twisted connected sum $G_2$-manifolds}
\author[J. Nordström]{Johannes Nordström}
\address{Department of Mathematical Sciences,
University of Bath,
Bath BA2 7AY, UK}
\email{j.nordstrom@bath.ac.uk}

\ifthenelse{\boolean{\showdate}}{\date{\today}}{}

\newcommand{\jncomm}[1]{\textbf{\color{red}[}#1\textbf{\color{red}]}}

\newcommand{\ignore}[1]{}

\DeclareMathOperator{\re}{Re}

\DeclareMathOperator{\im}{Im}

\DeclareMathOperator{\coker}{coker}
\DeclareMathOperator{\rk}{rk}

\newcommand{\ie}{\emph{i.e.} }
\newcommand{\eg}{\emph{e.g.} }
\newcommand{\cf}{\emph{cf.} }

\newcommand{\kd}{\Sigma}
\newcommand{\rs}{Y}
\newcommand{\bm}{H}
\newcommand{\hdg}{h}

\newcommand{\sm}[1]{\left(\begin{smallmatrix} #1 \end{smallmatrix} \right)}
\newcommand{\cvec}[2]{\sm{#1 \\ #2}}
\newcommand{\rvec}[2]{\sm{\!#1 & #2\!}}
\newcommand{\fbb}{\mathcal{Z}}
\newcommand{\sff}{\mathcal{Y}}

\newcommand{\calo}{\mathcal{O}}
\DeclareMathOperator{\Amp}{Amp}
\newcommand{\trivc}{\underline{\bbc}}
\newcommand{\trivr}{\underline{\bbr}}

\newcommand{\cyl}{\infty}
\newcommand{\thet}{\vartheta}
\newcommand{\hk}{hyper-Kähler\xspace}
\DeclareMathAlphabet{\matheur}{U}{eur}{m}{n}
\newcommand{\hkr}{\matheur{r}}
\newcommand{\tormat}{\matheur{t}}
\newcommand{\kclass}{\matheur{k}}
\newcommand{\oM}{\hspace*{0pt}\mkern 4mu\overline{\mkern-4mu M\mkern-1mu}\mkern 1mu}
\DeclareMathOperator{\Pic}{Pic}

\newcommand{\PP}{\mathbb{P}}

\newcommand{\Sph}{\mathbb{S}}

\newcommand{\isom}{\stackrel{\sim}{\to}}
\DeclareFontFamily{U} {MnSymbolA}{}
\DeclareFontShape{U}{MnSymbolA}{m}{n}{
  <-6> MnSymbolA5
  <6-7> MnSymbolA6
  <7-8> MnSymbolA7
  <8-9> MnSymbolA8
  <9-10> MnSymbolA9
  <10-12> MnSymbolA10
  <12-> MnSymbolA12}{}
\DeclareFontShape{U}{MnSymbolA}{b}{n}{
  <-6> MnSymbolA-Bold5
  <6-7> MnSymbolA-Bold6
  <7-8> MnSymbolA-Bold7
  <8-9> MnSymbolA-Bold8
  <9-10> MnSymbolA-Bold9
  <10-12> MnSymbolA-Bold10
  <12-> MnSymbolA-Bold12}{}
\DeclareSymbolFont{MnSyA} {U} {MnSymbolA}{m}{n}
\DeclareMathSymbol{\dashedleftarrow}{\mathrel}{MnSyA}{98}
\DeclareMathSymbol{\dashedrightarrow}{\mathrel}{MnSyA}{96}

\newcommand{\ZZ}{\mathbb{Z}}
\DeclareMathOperator{\Tor}{Tor}
\DeclareMathOperator{\Hom}{Hom}

\newcommand{\krel}[2]{[#1,\,#2]}
\newcommand{\ctrel}[2]{c_2[#1,\,#2]}
\newcommand{\corel}[2]{c_1[#1,\,#2]}

\newcommand{\mmod}{\!\!\mod}

\newcommand{\fpif}{\texorpdfstring{$\tfrac{\pi}{4}$}{pi/4}}
\newcommand{\fpis}{\texorpdfstring{$\tfrac{\pi}{6}$}{pi/6}}
\newcommand{\half}{{\textstyle\frac{1}{2}}}
\newcommand{\threehalf}{{\textstyle\frac{3}{2}}}
\newcommand{\third}{{\textstyle\frac{1}{3}}}
\newcommand{\twothird}{{\textstyle\frac{2}{3}}}

\newcommand{\quart}{\textstyle\frac{1}{4}}
\newcommand{\threequart}{{\textstyle\frac{3}{4}}}

\newcommand{\bbz}{\mathbb{Z}}
\newcommand{\Z}{\mathbb{Z}}
\newcommand{\cg}[1]{\bbz_{#1}}
\newcommand{\Q}{\mathbb{Q}}
\newcommand{\bbq}{\mathbb{Q}}
\newcommand{\bbr}{\mathbb{R}}
\newcommand{\R}{\mathbb{R}}
\newcommand{\bbc}{\mathbb{C}}

\newcommand{\bbp}{\mathbb{P}}
\newcommand{\bbrp}{\mathbb{R}^{+}}

\newcommand{\into}{\hookrightarrow}

\newcommand{\Id}{\textup{Id}}

\newcommand{\gtstr}{$G_{2}$\nobreakdash-\hspace{0pt}structure}

\newcommand{\gtmfd}{\texorpdfstring{$G_{2}$\nobreakdash-\hspace{0pt}}{G2-}manifold}

\newcommand{\norm}[1]{\Vert #1 \Vert}

\newcommand{\anglen}{u}
\newcommand{\anglex}{v}
\newcommand{\drn}{\mathbf{u}}
\newcommand{\drx}{\mathbf{v}}
\newcommand{\lnx}{\xi}
\newcommand{\lnn}{\zeta}
\newcommand{\arx}{\psset{linecolor=red}}
\newcommand{\arn}{\psset{linecolor=blue}}

\newcommand{\antip}{a}
\newcommand{\mtor}[2]{S^1_\lnx \mathbin{\wt\times} #1} 
\newcommand{\mtt}[1]{\mtor{#1}{\tau}}

\newcommand{\mtlift}[1]{\wt{{#1}}}
\newcommand{\clplift}[1]{\wh{{#1}}}

\newcommand{\wt}[1]{\widetilde #1}
\newcommand{\wh}[1]{\widehat #1}

\newcommand{\gen}[1]{\langle#1\rangle}
\newcommand{\inner}[1]{\langle#1\rangle}

\newcommand{\contra}[1]{\textstyle\frac{\partial}{\partial #1}}

\numberwithin{equation}{section} 
\newtheorem{thm}[equation]{Theorem}
\newtheorem{prop}[equation]{Proposition}
\newtheorem{lem}[equation]{Lemma}

\newtheorem{cor}[equation]{Corollary}

\theoremstyle{definition}
\newtheorem{defn}[equation]{Definition}
\newtheorem{notn}[equation]{Notation}

\newtheorem{constr}[equation]{Construction}
\theoremstyle{remark}
\newtheorem{rmk}[equation]{Remark}
\newtheorem*{rmk*}{Remark}
\newtheorem{ex}[equation]{Example}

\setlist{leftmargin=*}
\newcommand{\halfedge}{}

\newcommand{\gridx}{}
\newcommand{\gridy}{}

\newcommand{\labdarrow}[4]{
\psline[arrowscale=2]{->}(0,0)(!#1 5 mul \gridx\space mul #2 5 mul \gridy\space mul)
\uput[#4](!#1 5 mul \gridx\space mul \halfedge\space add #2 5 mul \gridy\space mul \halfedge\space add){#3}}

\newcommand{\grid}[2]{
\multido{\ix= \numexpr -#1 * 2 \relax +2}{\numexpr #1 * 2 + 1}{
\multido{\iy= \numexpr -#2 * 2 \relax +2}{\numexpr #2 * 2 + 1}{
\psdot[dotsize=1](!\ix\space 5 mul \gridx\space mul \iy\space 5 mul \gridy\space mul)}}
}

\newcommand{\halfgrid}[2]{
\multido{\ix= \numexpr -#1 * 2 - -1 \relax +2}{\numexpr #1 * 2}{
\multido{\iy= \numexpr -#2 * 2 - -1 \relax +2}{\numexpr #2 * 2}{
\psdot[dotsize=1](!\ix\space 5 mul \gridx\space mul \iy\space 5 mul \gridy\space mul)}}}

\newcommand{\insertgrid}[3]{
\renewcommand{\gridx}{#1}
\renewcommand{\gridy}{#2}
\psset{unit=1.2mm}
\begin{pspicture}(\numexpr \halfedge*2,\numexpr \halfedge*2)
\psset{origin={\halfedge,\halfedge}} 
\newgray{intergray}{0.75}
\newgray{vlight}{0.85}
\newgray{darkgray}{0.5}
\SpecialCoor
#3
\psdot[linecolor=black, dotsize=1.5](0,0)
\end{pspicture}}

\newcommand{\grida}[1]{
\insertgrid{1}{1}{
\grid{1}{1}
#1}}

\newcommand{\gridb}[1]{
\insertgrid{1.412}{1.412}{
\grid{1}{1}
\halfgrid{1}{1}
#1}}

\newcommand{\gridc}[1]{
\insertgrid{1.732}{1}{
\grid{1}{1}
\halfgrid{1}{2}
#1}}

\newcommand{\gridd}[1]{
\insertgrid{1}{1.732}{
\grid{1}{1}
\halfgrid{2}{1}
#1}}

\newcommand{\doublefig}[3]{
\renewcommand{\halfedge}{#1}
\protect \begin{figure}
\protect \begin{minipage}{0.48\textwidth}
\centering #2
\protect \end{minipage}\hfill
\protect \begin{minipage}{0.48\textwidth}
\centering #3
\protect \end{minipage}
\protect \end{figure}
}

\begin{document}

\begin{abstract}
We present a construction of closed 7-manifolds of holonomy $G_2$, which
generalises Kovalev's twisted connected sums by taking quotients of the
pieces in the construction before gluing.
This makes it possible to realise a wider range of topological types,
and Crowley, Goette and the author use this to exhibit examples
of closed 7-manifolds with disconnected moduli space of holonomy $G_2$ metrics.
\end{abstract}

\maketitle

The twisted connected sum construction pioneered by Kovalev \cite{kovalev03}
is a way to construct closed 7-dimensional Riemannian 7-manifolds with holonomy
$G_2$ from algebraic geometric data. Corti, Haskins, Pacini and the author
\cite{g2m} employed the construction to exhibit many examples of
$G_2$-manifolds whose topology can be understood in great detail.
The aim of this paper is to present a variation of the twisted connected sum
construction that removes some restrictions on the topology of the resulting
7-manifolds and $G_2$-structures. In particular, it is proved by Crowley,
Goette and the author in \cite{eta} that this construction can be used to
produce examples of 7-manifolds such that the moduli space of $G_2$ metrics is
disconnected.

7-dimensional manifolds with holonomy $G_2$ appear as an exceptional
case in Berger's classification of possible holonomy groups of Riemannian
manifolds \cite{berger55}. The first complete examples of manifolds with
holonomy $G_2$ were found by Bryant and Salamon \cite{bryant89}, and have
large symmetry group. In contrast, closed $G_2$-manifolds can never have
continuous symmetries, because $G_2$-metrics are always Ricci-flat.
The first examples of holonomy $G_2$ metrics on closed manifolds were
found by Joyce~\cite{joyce96-I}, gluing together reducible pieces to resolve
quotients of flat orbifolds.

The twisted connected sum construction developed later by
Kovalev \cite{kovalev03} works by gluing together two
pieces, each of which is a product of a circle $S^1$ and a complex 3-fold
with an asymptotically cylindrical Calabi-Yau metric. Each piece thus has
holonomy $SU(3)$, a proper subgroup of $G_2$. The asymptotically cylindrical
Calabi-Yau 3-folds can be obtained from algebraic geometry data, \eg starting from
Fano 3-folds. The cross-section of the asymptotic cylinder is of the form
$S^1 \times \kd$ for a K3 surface $\kd$. In the gluing,
the asymptotic cylinders of the pieces---each with cross-section
$S^1 \times S^1 \times \kd$---are identified by an isomorphism that swaps the
$S^1$ factors in order to produce a simply-connected 7-manifold $M$, admitting
metrics with holonomy exactly $G_2$. This relies on finding a so-called
\emph{\hk rotation} between the K3 factors in the cross-sections,
see Definition~\ref{def:hkr}.

Corti, Haskins, Pacini and the author \cite{cym, g2m} extended the supply of
algebraic geometric building blocks to which the twisted connected sum
construction can be applied, and analysed the topology of millions of the
resulting $G_2$-manifolds. While the $G_2$-manifolds constructed by Joyce
typically have non-zero second Betti number $b_2$, many twisted connected
sums---indeed, the ones that can be constructed with the least effort---are
2-connected, making it possible to apply the classification
theory of Wilkens \cite{wilkens72, wilkens74}, Crowley \cite{crowley02} and
Crowley and the author \cite{7class} (see Theorem \ref{thm:7class})
to completely determine the diffeomorphism type of the underlying 7-manifold.
 
Twisted connected sum $G_2$-manifolds $M$ always have the following topological
properties.
\begin{enumerate}
\item $b_2(M) + b_3(M)$ is odd \cite[(8.56)]{kovalev03}.
\item 
The torsion subgroup $\Tor H^4(M)$ equipped with the linking form 
splits as ${G \times \Hom(G, \Q/\Z)}$ for some finite group $G$
\cite[Proposition 3.8]{exotic}.
In particular, the size of $\Tor H^4(M)$ is a square integer.
\item The invariant $\nu \in \Z/48$ takes the value 24 \cite[Theorem 1.7]{nu},
and the refinement $\bar \nu \in \Z$ vanishes \cite[Corollary 3]{eta}.
\end{enumerate}
Here $\nu$ and $\bar\nu$ are invariants not of the 7-manifold, but of the
$G_2$-metric. A metric with holonomy exactly $G_2$ is equivalent to a
torsion-free \gtstr. A \gtstr{} means a reduction of the structure group of
the frame bundle from $GL(7, \bbr)$ to $G_2$, but is simplest described in
terms of a smooth pointwise stable 3-form $\varphi \in \Omega^3(M)$.
The torsion-free condition corresponds to a
first-order partial differential equation for the 3-form $\varphi$. 

Now, given a $G_2$-structure $\varphi$ on any closed 7-manifold,
we may define $\nu(\varphi) \in \Z/48$ in terms of a spin coboundary
\cite[Definition 3.1]{nu}. This is invariant under both diffeomorphisms and
homotopies (continuous deformations of the \gtstr, ignoring the torsion-free
condition). Further, \mbox{\cite[Definition 1.4]{eta}} introduces a refinement
$\bar\nu(\varphi) \in \Z$ in terms of eta invariants. It is a refinement in the
sense that for \gtstr s of holonomy $G_2$ metrics, $\bar\nu$ determines $\nu$
by the relation $\nu(\varphi) \equiv \bar \nu(\varphi) +24 \mmod 48$.
While $\bar\nu(\varphi)$ too is invariant under diffeomorphisms, it is
not invariant under arbitrary homotopies of \gtstr s. However, $\bar\nu$ is
invariant under deformations through torsion-free \gtstr s.

\begin{rmk*}
There is a parity constraint
\begin{equation}
\label{eq:nu_parity}
\nu(\varphi) = \chi_2(M) \mmod 2,
\end{equation}
where $\chi_2(M)$ is the semi-characteristic $\sum_{i=0}^3 b_i(M) \in \Z/2$,
reducing to $1 + b_2(M) + b_3(M)$ for a simply-connected 7-manifold. Thus (iii)
formally entails (i).
\end{rmk*}

These invariants give a potential method to distinguish connected components
of the $G_2$ moduli space on a closed 7-manifold. However, even though
there are many pairs of twisted connected sums whose underlying 7-manifolds
can be shown to be diffeomorphic by the classification theory, (iii) means that
$\nu$ and $\bar\nu$ fail to distinguish their components in the moduli space
in this case.

In this paper we modify the twisted connected sum construction by dividing
either or both of the two pieces in the construction by an involution before
gluing.
This maintains many of the attractive features of the twisted connected sum
construction: examples can be generated starting from algebraic geometry data,
topological invariants can be computed from the algebraic inputs, and the
resulting 7-manifolds are often 2-connected and simple enough to apply
diffeomorphism classification theory. On the other hand, the topology of the result is less restrictive.

\begin{enumerate}[label=(\roman*')]
\item There is no constraint on the parity of $b_2(M) + b_3(M)$.
\item The size of $\Tor H^4(M)$ need not be a square integer, and in particular
the linking form need not split.
\item The values of $\nu$ and $\bar\nu$ can vary.
\end{enumerate}

The drawback compared with the ordinary twisted connected sum construction is
that requiring an involution limits the range of algebraic building blocks to
which the construction can be applied. Also, the topological computations are
more involved. 

We exhibit a selection of 50 explicit examples of 7-manifolds with
holonomy $G_2$ obtained from the new construction.
All except Example \ref{ex:intersection} are 2-connected. 7 of those
have odd $b_3$ and torsion-free $H^4(M)$, and 5 of those are
diffeomorphic to some ordinary twisted connected sum.
The $\bar\nu$-invariant of extra-twisted connected sums is computed in
\cite[Corollary 2]{eta} (see Theorem~\ref{thm:nubar}), and used there to prove
that these lead to examples of closed 7-manifolds with disconnected moduli
space of holonomy $G_2$ metrics.

Among the examples in this paper, we also find
\begin{itemize}
\item
A 7-manifold whose $G_2$ moduli space has at least 3 components
(see Examples \ref{ex:77pi4} and \ref{ex:pi6deg5},
using the formula for $\bar\nu$ from~\cite{eta}).
\item A pair of \gtmfd s whose diffeomorphism types
are distinguished only by the type of the torsion linking form
(Examples \ref{ex:diag} and \ref{ex:hyperb}).
\item A pair of \gtmfd s with equal $\bar\nu$-invariant, such that the
underlying manifolds are diffeo\-morphic, but (due to order 3 torsion in $H^4$)
only by an orientation-reversing diffeomorphism; thus the fact that $\bar\nu$
changes sign under reversing orientation can be used to distinguish connected
components of the $G_2$ moduli space on this 7-manifold (Examples \ref{ex:t3a} and \ref{ex:t3b}).
\item A \gtmfd{} that illustrates a subtlety in the calculation of the number
of smooth structures on 2-connected 7-manifolds with 8-torsion in
$H^4$: Wilkens \cite[Conjecture p.\,548]{wilkens74} predicts that Example
\ref{ex:inertia} has a unique
smooth structure, but according
to \cite[Theorem 1.10]{7class} it has two.
\end{itemize}

\subsection*{Organisation}

The paper consists of two strands. The first is to set up the general
machinery of the extra-twisted connected sum construction. The procedure for
gluing ACyl Calabi-Yau manifolds (possibly with involution) is made precise
in \S\ref{sec:glue}, while \S\ref{sec:blocks} describes
the closed Kähler 3-fold ``building blocks'' from which we obtain ACyl
Calabi-Yau 3-folds, and what data of these blocks is important.
The matching problem, \ie how to find \hk rotations between pairs of
ACyl Calabi-Yau 3-folds, is addressed in \S\ref{sec:match}, and \S\ref{sec:top} explains how to compute key invariants of the resulting \gtmfd s.

The second strand is producing examples. Two methods of producing building
blocks are provided in \S\ref{sec:sfblocks} and \S\ref{sec:k3blocks}, starting
from semi-Fano 3-folds and K3s with non-symplectic involution, respectively.
In \S\ref{sec:examples} we exhibit a number of examples of matchings of those
blocks and compute the topology of the extra-twisted connected sums.
In some cases, the matchings rely on understanding of which K3 surfaces appear
in certain families of building blocks, which is studied in detail in
\S\ref{sec:genericity}.

Some of the machinery we set up---in particular the discussion of the matching
problem in~\S\ref{sec:match}---works in the same way in a more general
setting where one allows to divide by automorphisms of order greater than 2.
This is studied further by Goette and the author in \cite{nuxx}. However,
the topological calculations are less tractable there.

\subsection*{Acknowledgements}

The author thanks
Olivier Debarre,
Alessio Corti,
Diarmuid Crowley,
Sebastian Goette,
Mark Haskins,
Jesus Martinez Garcia,
and Dominic Wallis for valuable discussions,
the referee for constructive comments,
and the Simons Foundation for its support under the Simons Collaboration
on Special Holonomy in Geometry, Analysis and Physics
(grant \#488631, Johannes Nordström).

\section{The basics of the construction}
\label{sec:glue}

\subsection{Reducible \gtmfd s}
\label{subsec:reducible}

For $\lnn > 0$, let $S^1_\lnn$ denote $\bbr/\lnn\bbz$, and $\anglen$ its
coordinate (with period~$\lnn$); the parameter $\lnn$ 
affects the geometric meaning of the coordinate expressions for metrics
below.

\begin{thm}[{\cite[Theorem D]{hhn}}]
\label{thm:acyl}
Let $Z$ be a compact Kähler 3-fold
containing a smooth anti-canonical K3 surface $\kd$ with trivial normal bundle.
Let $V := Z \setminus \kd$, and consider it as a manifold with a cylindrical
end of cross-section $S^1_\lnn \times \kd$.
Let $I$ be the complex structure on $\kd$ induced by~$Z$, and let
$(\omega^I, \omega^J, \omega^K)$ be a \hk K3 structure on $\kd$ such that
$\omega^J + i \omega^K$ is (2,0) with respect to $I$ while $[\omega^I]$ is the
restriction of some Kähler class $\kclass \in H^2(Z;\R)$.
For any $\lnn > 0$ there is a unique ACyl Calabi-Yau structure
$(\Omega, \omega)$ on $V$, with $\omega \in \kclass_{|V}$ and asymptotic limit
\begin{align*} 
\omega_\cyl &:= dt \wedge d\anglen + \omega^I, \\
\Omega_\cyl &:= (d\anglen - i dt) \wedge (\omega^J + i \omega^K) .
\end{align*}
(In this metric, the $S^1_\lnn$ factor in the cross-section has circumference
$\lnn$.)
\end{thm}

Given $\lnx > 0$, define a product \gtstr{} $\varphi$ on $S^1_\lnx \times V$ by
\[ \varphi := d\anglex\wedge\omega + \re \Omega , \]
where $\anglex$ denotes the coordinate on the \emph{external} circle factor
$S^1_\lnx$ (whose circumference with respect to the induced metric is $\lnx$).
The asymptotic limit of $\varphi$ is
\[
\varphi_\cyl 
= d\anglex\wedge dt \wedge d\anglen + d\anglex\wedge\omega^I +
d\anglen\wedge\omega^J + dt \wedge \omega^K .
\]
Letting
\begin{equation}
\label{eq:zdef}
z = \anglex + i\anglen,
\end{equation}
we can rewrite the limit as
\begin{equation}
\label{eq:limit}
\varphi_\cyl = 
\re \left(dz \wedge (\omega^I - i \omega^J)\right) +
dt \wedge \left(\omega^K - {\textstyle \frac{i}{2}} dz \wedge d\bar z \right) .
\end{equation}
Note that $\lnn$ and $\lnx$ are the side lengths of the rectangular $T^2$
factor in the cross-section of $S^1_\lnx \times V$.
If $\partial_\anglen, \partial_\anglex \in \bbr^2$ is the orthonormal
frame dual to $d\anglen, d\anglex$, then we can think of
$\lnn\partial_\anglen$ and $\lnx\partial_\anglex$ as the generators of the
lattice defining the $T^2$.
Let $\varphi^{s0}$ be the \gtstr{} obtained by setting $\lnn = \lnx =1$, as we
do in the ordinary twisted connected sum construction; then the $T^2$ factor
is simply the quotient of $\bbc$ by the unit square lattice as illustrated in
Figure \ref{fig:varphi0}.
(Note that real axis (in red)
$\leftrightarrow \; \anglen = 0\; \leftrightarrow$
external circle factor.) 

Suppose now that there is a holomorphic involution $\tau$ on $Z$ such that
$\kd$ is a component of the fixed set; \cf Definition \ref{def:inv_block}.
Then the restriction of $\tau$ to $V$ is asymptotic to the involution
$a \times \Id$ on $S^1_\lnn \times \kd$, where
$a : S^1_\lnn \to S^1_\lnn$ denotes the antipodal map
$\anglex \mapsto \anglex + \half \lnn$. If we choose the Kähler class $\kclass$
in Theorem \ref{thm:acyl} to be $\tau$-invariant, then so is the resulting
Calabi-Yau structure $(\Omega, \omega)$.
The product \gtstr s above then descend to ones on the quotient
\[ \mtt{V} := S^1_\lnx {\times} V \, / \, a {\times} \tau . \]
The cross-section is $T^2 \times \kd$ for
$T^2 := S^1_\lnx \times S^1_\lnn/a \times a$. Note that this $T^2$ is still a flat
2-torus, but \emph{not} a metric product of circles unless $\lnx = \lnn$.
Let $\varphi^{s1}$, $\varphi^{h0}$ and $\varphi^{h1}$ be the \gtstr s on
$\mtt{V}$ corresponding to $(\lnn,\lnx) = (\sqrt{2}, \sqrt{2})$,
$(\sqrt{3}, 1)$ and $(1, \sqrt{3})$ respectively.
As illustrated in Figures \ref{fig:varphi1}--\ref{fig:varphi3}, the $T^2$
factor in the cross-section is a unit square torus with respect to
$\varphi^{s1}$, and a hexagonal torus with side length 1 with respect to
$\varphi^{h0}$ and $\varphi^{h1}$.

\doublefig{17}{
\insertgrid{0.707}{0.707}{
\psframe[linestyle=none,fillstyle=solid,fillcolor=vlight](0,0)(7.3,7.3)
\grid{2}{2}
\small
\arn
\labdarrow{0}{2}{$\partial_\anglen$}{l}
\arx
\labdarrow{2}{0}{$\partial_\anglex$}{d}}

\caption{$\varphi^{s0}$}
\label{fig:varphi0}
}{
\insertgrid{1}{1}{
\pspolygon[linestyle=none,fillstyle=solid,fillcolor=vlight](0,0)(5,5)(10,0)(5,-5)
\small
\grid{2}{2}
\halfgrid{2}{2}
\arn
\labdarrow{0}{2}{$\sqrt{2}\partial_\anglen$}{l}
\arx
\labdarrow{2}{0}{$\sqrt{2}\partial_\anglex$}{d}
}

\caption{$\varphi^{s1}$}
\label{fig:varphi1}
}

\doublefig{21}{
\insertgrid{0.707}{1.225}{
\pspolygon[linestyle=none,fillstyle=solid,fillcolor=vlight](0,0)(3.535,6.125)(10.605,6.125)(7.07,0)
\grid{2}{3}
\halfgrid{2}{2}
\small
\arn
\labdarrow{0}{2}{$\sqrt{3}\partial_\anglen\!\!$}{l}
\arx
\labdarrow{2}{0}{$\partial_\anglex$}{d}}

\caption{$\varphi^{h0}$}
\label{fig:varphi2}
}{
\insertgrid{1.225}{0.707}{
\pspolygon[linestyle=none,fillstyle=solid,fillcolor=vlight](0,0)(6.125,3.535)(6.125,10.605)(0,7.07)
\small
\grid{2}{2}
\halfgrid{2}{2}
\arn
\labdarrow{0}{2}{$\partial_\anglen$}{l}
\arx
\labdarrow{2}{0}{$\sqrt{3}\partial_\anglex$}{d}}

\caption{$\varphi^{h1}$}
\label{fig:varphi3}
}

\subsection{Gluing}

Let $(M_+, \varphi_+)$ and $(M_-, \varphi_-)$ be a pair of reducible
ACyl \gtmfd s, such that either each is of the form
$(S^1_\lnx \times V, \varphi^{s0})$ or $(\mtt{V}, \varphi^{s1})$, or each is of
the form $(\mtt{V}, \varphi^{h0})$ or $(\mtt{V}, \varphi^{h1})$ above.
We strive to treat the cases as uniformly as possible, and may use the
shorthand $\varphi^{ab}$ for symbols $a \in \{s, h\}$ and $b \in \{0, 1\}$.
Let $(\omega^I_\pm, \omega^J_\pm, \omega^K_\pm)$
be the corresponding \hk structures, and define $z_\pm$ by \eqref{eq:zdef}.

Let $\thet \in \bbr$ such that the isometry
$\bbc \to \bbc, \; z_+ \mapsto z_- := e^{i\vartheta} \bar z_+$ descends
to an %
isometry
\begin{equation}
\label{eq:tormat}
\tormat : T^2_+ \to T^2_-
\end{equation}
of the torus factors in the cross-sections of $M_+$ and $M_-$.
The condition that $\tormat$ is well-defined on the quotient is equivalent to
\begin{equation}
\vartheta \; = \; \left\{ \begin{array}{c}
\displaystyle\frac{k\pi}{2} \quad \textrm{if } a = s , \\
\displaystyle\frac{k\pi}{3} \quad \textrm{if } a = h ,
\end{array}
\right.
\end{equation}
for some~$k\in\frac12\bbz$ with~ $k\equiv\frac{b_+ + b_-}{2} \mod \bbz$.
We call $\thet$ the \emph{gluing angle} of $\tormat$.

Let $\hkr : \kd_+ \to \kd_-$ be a diffeomorphism, and
\begin{equation}
\label{eq:gluemap}
F := (-\Id_\R) \times \tormat \times \hkr
\; : \;\R \times T^2_+ \times \kd_+ \;  \longrightarrow \;
\R \times T^2_- \times \kd_- . %
\end{equation}
From \eqref{eq:limit}, we see that \eqref{eq:gluemap} is an isomorphism of the
asymptotic limits of $\varphi_\pm$ if and only if
\begin{equation}
\label{eq:hkr}
\begin{aligned}
\hkr^*\omega^K_- &= - \omega^K_+ \\
\hkr^*(\omega^I_- + i \omega^J_-) &= e^{i\vartheta} (\omega^I_+ - i\omega^J_+) .
\end{aligned}
\end{equation}

\begin{defn}
\label{def:hkr}
Call $\hkr : \kd_+ \to \kd_-$ a \emph{$\vartheta$-\hk rotation} if \eqref{eq:hkr}
holds.
\end{defn}

We consider the problem of finding such \hk rotations in \S\ref{sec:match}.
The special case of a $\frac{\pi}{2}$-\hk rotation coincides with the
notion of a \hk rotation from previous work on twisted connected sums,
\eg \cite[Definition 3.10]{g2m}. 

In these terms, suppose we can find a pair of reducible ACyl \gtmfd s
$(M_\pm, \varphi_\pm)$ of the above form, with asymptotic cross-sections
$T^2_\pm \times \kd_\pm$. Suppose further we can find an isometry
$\tormat : T^2_+ \to T^2_-$ as in \eqref{eq:tormat},
and a $\thet$-\hk rotation $\hkr : \kd_+ \to \kd_-$ for $\thet$ the
gluing angle of~$\tormat$.

\begin{thm}
\label{thm:glue}
For $\ell \gg 0$, let $M_\pm[\ell]$ be the truncation of $M_\pm$ at $t = \ell$,
and form a closed 7-manifold $M$ by gluing $M_+[\ell]$ to $M_-[\ell]$ along
their boundaries by the diffeomorphism
$\tormat \times \hkr : T^2_+ \times \kd_+ \to T^2_- \times \kd_-$.
Use a cut-off function to patch $\varphi_+$ and $\varphi_-$ to a closed
\gtstr{} $\tilde \varphi_\ell$ on $M$ such that
$\norm{\tilde \varphi_{|M_\pm[\ell]} - \varphi_{\pm|M_\pm[\ell]}} = O(e^{-\delta \ell})$.
Then there exists a unique torsion-free
\gtstr{} $\varphi$ in the cohomology class of $\tilde \varphi_\ell$
such that $\norm{\varphi - \tilde \varphi} = O(e^{-\delta \ell})$.
\end{thm}

\begin{proof}
Analogous to \cite[Theorem 5.34]{kovalev03}.
\end{proof}

\begin{constr}
\label{constr:xtcs}
We call the 7-manifold $M$ from Theorem \ref{thm:glue} a
\emph{$\thet$-twisted connected sum}.
\end{constr}

\noindent
When $a = s$ and $b_+ = b_- = 0$,
setting $\thet = \frac{\pi}{2}$ recovers the usual
notion of a twisted connected sum (and $\thet\in\pi\bbz$ gives an
``untwisted'' connected sum, with $b_1(M) = 1$ and holonomy not all of~$G_2$).

\subsection{Angles}
\label{subsec:angles}

Before we enumerate the possible combinations of $(a, b_+, b_-, \vartheta)$
that make it possible to match $\varphi^{ab_+}$ to $\varphi^{ab_-}$ with
a torus matching $\tormat$ with gluing angle $\vartheta$,
let us discuss briefly the geometric meaning of $\vartheta$.
We can think of $\vartheta$
as the angle in $T^2$ between the external circle factors in $M_+$ and $M_-$,
but that leaves an ambiguity of sign and complementary angles. However,
because the definition of the \gtstr s involves an orientation of the external
circle factors the direction of the tangent vectors $\partial_{\anglex_+}$ and
$\partial_{\anglex_-}$ have some meaning, and the angle between them is
$|\vartheta| \in (0, \pi)$. The sign can be described in terms of the complex
structure on the cross-section induced by the \gtstr{} on $M_+$ (vector
multiplication by $\partial_t$); because the $T^2$ factor is a complex curve,
it makes sense to consider the oriented angle from $\partial_{\anglex_+}$ to
$\partial_{\anglex_-}$.

If we swap the roles of $M_+$ and $M_-$ then the complex structure on the
cross-section is conjugated, so even though $\partial_{\anglex_+}$ and
$\partial_{\anglex_-}$ are swapped the oriented angle $\vartheta$ is unchanged.
More formally, note that if $\hkr : \kd_+ \to \kd_-$ is a
$\vartheta$-\hk rotation, then so is $\hkr^{-1}$. Let $(M', \varphi')$ be the
corresponding $\vartheta$-twisted connected sum of $M_-$ and $M_+$.
Then there is a tautological (oriented) diffeomorphism
$M \to M'$, and that pulls back $\varphi'$ to~$\varphi$.

\enlargethispage{0.9\baselineskip}

Here is another symmetry to bear in mind. We obtained the product \gtstr s
$\varphi_\pm$ on $M_\pm$ from ACyl Calabi-Yau structures
$(\Omega_\pm, \omega_\pm)$ on $V_\pm$. Phase rotation by $\pi$ gives an equally
good Calabi-Yau structure $(-\Omega_\pm, \omega_\pm)$, and another product
\gtstr{} $\varphi'_\pm$.
The asymptotic limit of $\varphi'_\pm$ is encoded by the \hk
structure $(\omega^I_\pm, -\omega^J_\pm, -\omega^K_\pm)$. Inspecting
\eqref{eq:hkr} we see that a $\vartheta$-\hk rotation for $\varphi_+$ and
$\varphi_-$ is the same thing as a $(-\vartheta)$-\hk rotation for
$\varphi'_+$ and~$\varphi'_-$. Let $(M', \varphi')$ be the resulting
$(-\vartheta)$-twisted connected sum.
Now $(\anglex_\pm, x) \mapsto (-\anglex_\pm, x)$ defines an
orientation-reversing diffeomorphism of $M_\pm$, pulling back $\varphi'_\pm$ to
$-\varphi_\pm$. These match up to define an orientation-reversing
diffeomorphism $M \to M'$ that pulls back $\varphi'$ to~$-\varphi$.
 
Taking these symmetries into account, any extra-twisted connected sum will be
isomorphic to one that has $b_+ \geq b_-$ and $\vartheta \in (0,\pi)$,
and uses exactly the same (unordered) pair of building blocks.

In listing the possibilities, we therefore restrict our attention to
such cases.
We find below that there is essentially a single interesting type of
$\vartheta$-twisted connected sum for each
\begin{equation}
\label{eq:psilist}
\vartheta \in \left\{ \frac{\pi}{6}, \frac{\pi}{4}, \frac{\pi}{3},
\frac{\pi}{2}, \frac{2\pi}{3}, \frac{3\pi}{4}, \frac{5\pi}{6} \right\} .
\end{equation}

\begin{rmk}
\label{rmk:comp}
Finally, one can also argue that every $\vartheta$-twisted connected sum is
diffeomorphic to \emph{some} $\vartheta{+} \pi$-twisted connected sum.
Let $V_+'$ be $V_+$ with the orientation reversed, equipped with the
ACyl Calabi-Yau structure $(\bar \Omega_+, -\omega_+)$. Then a
$\vartheta$-\hk rotation for $M_+$ and $M_-$ is also a $\vartheta {+} \pi$-\hk
rotation for $M'_+$ and $M_-$.
The orientation-preserving diffeomorphism
$S^1_{\lnx_+} \times V_+ \to S^1_{\lnx_+} \times V'_+,
\; (\anglex_+, x) \mapsto (-\anglex_+, x)$
descends to $M_+ \to M'_+$, and pulls back $\varphi'_+$ to $\varphi_+$.
It patches up with the identity map on $M_-$ to define an isomorphism
from $M$ to the $\vartheta{+} \pi$-twisted connected sum of $M'_+$ and $M_-$.

Combined with the symmetries discussed above, this means that any
extra-twisted connected sum is isometric to some extra-twisted connected sum
with $\vartheta \in (0, \frac{\pi}{2}]$, but not necessarily using the same
(in an oriented sense) ACyl Calabi-Yau manifolds.
\end{rmk}

\noindent
Now we list and describe the possible combinations of
$(a, b_+, b_-, \vartheta)$ (equivalently the different kinds of torus isometries $\tormat$).
In each case we illustrate the action on the $T^2$ factor with a figure that
shows the lattice corresponding to the two identified tori. The figure includes arrows indicating the ``external'' and ``internal'' circle factors on each
side, \eg the orthogonal arrows $\lnn_+\partial_{\anglen_+}$ and
$\lnx_+\partial_{\anglex_+}$ indicate the overlattice (of index 2 if it is not
the whole lattice) corresponding to the metric product
$S^1_{\lnn_+} \times S^1_{\lnx_+}$ that appears as the asymptotic cross-section
in $V_+ \times S^1_{\lnx_+}$. The gluing angle can be seen as the angle
between the red arrows $\lnx_+\partial_{\anglex_+}$ and  
$\lnx_-\partial_{\anglex_-}$ corresponding to the two external circle factors.

\doublefig{16}{
\grida{
\psarc{->}{3}{0}{90}
\psset{linecolor=purple}
\labdarrow{0}{2}{$\partial_{\anglen_+} \! = \partial_{\anglex_-}$}{u}
\labdarrow{2}{0}{$\partial_{\anglex_+} \! = \partial_{\anglen_-}$}{d}}
\caption{$\vartheta = \displaystyle \frac{\pi}{2}$}
\label{fig:1/2}
}{
\grida{
\psarcn{->}{3}{0}{-90}
\arx
\labdarrow{2}{0}{$\partial_{\anglex_+}$}{d}
\labdarrow{0}{-2}{$\partial_{\anglex_-}$}{l}
\arn
\labdarrow{0}{2}{$\partial_{\anglen_+}$}{l}
\labdarrow{-2}{0}{$\partial_{\anglen_-}$}{d}}
\caption{$\vartheta = - \displaystyle \frac{\pi}{2}$}
\label{fig:neg1/2}
}

\doublefig{18}{
\gridb{
\psarc{->}{3}{0}{45}
\arx
\labdarrow{2}{0}{$\sqrt{2}\partial_{\anglex_+}$}{d}
\labdarrow{1}{1}{$\partial_{\anglex_-}$}{u}
\arn
\labdarrow{0}{2}{$\sqrt{2}\partial_{\anglen_+}$}{l}
\labdarrow{1}{-1}{$\partial_{\anglen_-}$}{d}}
\caption{$\vartheta = \displaystyle \frac{\pi}{4}$}
\label{fig:1/4}
}{
\gridb{
\psarc{->}{3}{0}{135}
\arx
\labdarrow{2}{0}{$\sqrt{2}\partial_{\anglex_+}$}{d}
\labdarrow{-1}{1}{$\partial_{\anglex_-}$}{l}
\arn
\labdarrow{0}{2}{$\sqrt{2}\partial_{\anglen_+}$}{l}
\labdarrow{1}{1}{$\partial_{\anglen_-}$}{r}}
\caption{$\vartheta = \displaystyle \frac{3\pi}{4}$}
\label{fig:3/4}
}

\begin{itemize}[itemsep=6pt]

\item Square, $b_+ = b_- = 0$, $\vartheta = \displaystyle \frac{\pi}{2}$.\\
As %
already explained, this corresponds to the usual twisted connected
sums. $\vartheta = -\frac{\pi}{2}$ is the same up to orientation.
See Figures \ref{fig:1/2} and
\ref{fig:neg1/2}.

\item Square, $b_+ = 1$, $b_- = 0$,
$\vartheta = \displaystyle \frac{\pi}{4}$ or $\displaystyle \frac{3\pi}{4}$. \\
See Figures \ref{fig:1/4} and \ref{fig:3/4}.
The figures also help us understand the fundamental group. Note that
$\sqrt{2}\partial_{\anglen_+}$ and $\partial_{\anglen_-}$ generate $\pi_1 T^2$.
On the other hand, we can picture $\pi_1 M_\pm$ as the projection of the
lattice onto the line spanned by $\partial_{\anglex_\pm}$ (this uses that
$V_\pm$ is simply connected, which is a consequence of our definition of what
it means for $Z_\pm$ to be a building block, \cf Lemma \ref{lemg:Z&V}(i)).
Thus we see that $\sqrt{2}\partial_{\anglen_+}$ is in the kernel of the
push-forward to $\pi_1 M_+$, while its image in $\pi_1 M_-$ is a generator.
Similarly $\partial_{\anglen_-}$ is in the kernel of the push-forward to
$\pi_1 M_-$, while its image in $\pi_1 M_+$ is a generator.
Van Kampen implies that the resulting extra-twisted connected sums are
simply-connected.

\item Hexagonal, $b_+ = b_- = 1$,
$\vartheta = \displaystyle \frac{\pi}{3}$ or $\displaystyle \frac{2\pi}{3}$. \\
See Figures \ref{fig:1/3} and \ref{fig:2/3}. The resulting extra-twisted
connected sums are simply-connected by the same reasoning as in the previous
case.

\doublefig{21}{
\gridc{
\psarc{->}{3}{0}{60}
\arx
\labdarrow{2}{0}{$\sqrt{3}\partial_{\anglex_+}$}{d}
\labdarrow{1}{3}{$\sqrt{3}\partial_{\anglex_-}$}{r}
\arn
\labdarrow{0}{2}{$\partial_{\anglen_+}$}{l}
\labdarrow{1}{-1}{$\partial_{\anglen_-}$}{d}}
\caption{$\vartheta = \displaystyle \frac{\pi}{3}$}
\label{fig:1/3}
}{
\gridc{
\psarc{->}{3}{0}{120}
\arx
\labdarrow{2}{0}{$\sqrt{3}\partial_{\anglex_+}$}{d}
\labdarrow{-1}{3}{$\sqrt{3}\partial_{\anglex_-}$}{l}
\arn
\labdarrow{0}{2}{$\partial_{\anglen_+}$}{r}
\labdarrow{1}{1}{$\partial_{\anglen_-}$}{r}}
\caption{$\vartheta = \displaystyle \frac{2\pi}{3}$}
\label{fig:2/3}
}

\doublefig{21}{
\gridc{
\psarc[arrowscale=0.8]{->}{4}{0}{30}
\arx
\labdarrow{2}{0}{$\sqrt{3}\partial_{\anglex_+}$}{d}
\labdarrow{1}{1}{$\partial_{\anglex_-}$}{u}
\arn
\labdarrow{0}{2}{$\partial_{\anglen_+}$}{l}
\labdarrow{1}{-3}{$\sqrt{3}\partial_{\anglen_-}$}{r}}
\caption{$\vartheta = \displaystyle \frac{\pi}{6}$}
\label{fig:1/6}
}{
\gridc{
\psarc{->}{3}{0}{150}
\arx
\labdarrow{2}{0}{$\sqrt{3}\partial_{\anglex_+}$}{d}
\labdarrow{-1}{1}{$\partial_{\anglex_-}$}{l}
\arn
\labdarrow{0}{2}{$\partial_{\anglen_+}$}{l}
\labdarrow{1}{3}{$\sqrt{3}\partial_{\anglen_-}$}{r}}
\caption{$\vartheta = \displaystyle \frac{5\pi}{6}$}
\label{fig:5/6}
}

\item Hexagonal, $b_+ = 1$, $b_- = 0$,
$\vartheta = \displaystyle \frac{\pi}{6}$ or $\displaystyle \frac{5\pi}{6}$. \\
See Figures \ref{fig:1/6} and \ref{fig:5/6}.
Once more, the resulting extra-twisted connected sums are simply-connected.

\end{itemize}

\noindent
The remaining possibilities do not give simply-connected extra-twisted
connected sums, and are in fact quotients of twisted connected sums of the
types above.
By a ``$\vartheta$-twisted connected sum'' for $\vartheta$ as in
\eqref{eq:psilist} we will therefore usually mean one of the types above.

\begin{itemize}[itemsep=6pt]

\item Square, $b_+ = b_- = 1$, $\vartheta = \displaystyle \frac{\pi}{2}$.\\
The lattice in Figure \ref{fig:1/2'} has index 2 in the direct sum of the
projections onto the $\partial_{\anglex_\pm}$ axes, so the fundamental group of
the extra-twisted connected sum $M$ is $\bbz_2$. The universal cover is the
ordinary twisted connected sum $\oM$ of $S^1_{\!\sqrt{2}} \times V_+$ and
$S^1_{\!\sqrt{2}} \times V_-$ (where
$M_\pm = S^1_{\!\sqrt{2}} {\times} V_\pm / a {\times} \tau_\pm$):
the involutions $a \times \tau_\pm$ patch up to an involution on $\oM$ with
quotient $M$.

\doublefig{23}{
\gridb{
\psarc{->}{3}{0}{90}
\psset{linecolor=purple}
\labdarrow{0}{2}{$\sqrt{2}\partial_{\anglen_+} {=} \sqrt{2}\partial_{\anglex_-}$}{u}
\labdarrow{2}{0}{$\sqrt{2}\partial_{\anglex_+} {=} \sqrt{2}\partial_{\anglen_-}$\hspace*{3ex}}{d}}
\vspace{-8mm}
\caption{$\vartheta = \displaystyle \frac{\pi}{2}$, $\pi_1 M \cong \bbz_2$}
\label{fig:1/2'}
}{
\gridc{
\psarc{->}{3}{0}{90}
\psset{linecolor=purple}
\labdarrow{0}{2}{$\partial_{\anglen_+} = \partial_{\anglex_-}$}{u}
\labdarrow{2}{0}{$\begin{array}{c}\sqrt{3}\partial_{\anglex_+} \\ {=} \sqrt{3}\partial_{\anglen_-} \end{array}$}{d}}
\vspace{-8mm}
\caption{$\vartheta = \displaystyle \frac{\pi}{2}$, $\pi_1M \cong \bbz_2$}
\label{fig:1/3''}
}

\doublefig{19}{
\gridd{
\psarc{->}{3}{0}{60}
\arx
\labdarrow{2}{0}{$\partial_{\anglex_+}$}{d}
\labdarrow{1}{1}{$\partial_{\anglex_-}$}{u}
\arn
\labdarrow{0}{2}{$\sqrt{3}\partial_{\anglen_+}$}{l}
\labdarrow{3}{-1}{$\sqrt{3}\partial_{\anglen_-}$}{d}}
\caption{$\vartheta = \displaystyle \frac{\pi}{3}$, $\pi_1M \cong \bbz_3$}
\label{fig:1/3'}
}{
\gridd{
\psarc{->}{3}{0}{120}
\arx
\labdarrow{2}{0}{$\partial_{\anglex_+}$}{d}
\labdarrow{-1}{1}{$\partial_{\anglex_-}$}{u}
\arn
\labdarrow{0}{2}{$\sqrt{3}\partial_{\anglen_+}$}{r}
\labdarrow{3}{1}{$\sqrt{3}\partial_{\anglen_-}$}{u}}
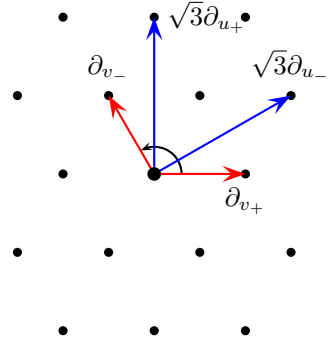
\caption{$\vartheta = \displaystyle \frac{2\pi}{3}$, $\pi_1M \cong \bbz_3$}
\label{fig:2/3'}
}

\item Hexagonal, $b_+ = 1$, $b_- = 0$, $\vartheta = \frac{\pi}{2}$.\\
See Figure \ref{fig:1/3''}. Clearly this configuration is essentially the
same as the previous one, up to some squashing of the $T^2$ factor.

\enlargethispage{0.1\baselineskip}

\item Hexagonal, $b_+ = b_- = 0$,
$\vartheta = \displaystyle \frac{\pi}{3}$ or $\displaystyle \frac{2\pi}{3}$. \\
See Figures \ref{fig:1/3'} and \ref{fig:2/3'}.
Using $\{ \partial_{\anglex_+}, \partial_{\anglex_-} \}$
as a basis for $\pi_1 T^2$, and $ \half \partial_{\anglex_\pm}$ as generators
for $\pi_1 M_\pm$, the push-forward $\pi_1 T^2 \to \pi_1 M_+ \times \pi_1 M_-$
is represented by $\sm{2 & \pm 1 \\ \pm 1 & 2}$.
Since the determinant is~3, we find $\pi_1 M \cong \bbz_3$.

Up to scale, the universal cover of $M$ is a $\vartheta$-twisted connected sum
$\oM$ of the form above, \ie with $b_+ = b_- = 1$.
Note that $M_\pm = S^1_{\!\sqrt{3}} {\times} V_\pm / a {\times} \tau_\pm$ has an
innocuous order 3 automorphism
$\rho_\pm : (\anglex_\pm, \, x)
\mapsto (\anglex_\pm {+} \frac{1}{\sqrt{3}}, \, x)$.
The quotient $M_\pm/\rho_\pm$ is diffeomorphic to $M_\pm$,
but the covering map pulls back product \gtstr s of the form $\varphi^{h1}_\pm$
to ones of the form $\varphi^{h0}_\pm$ (up to a scale factor~$\sqrt{3}$).
The automorphisms $\rho_\pm$ patch up to an order 3 automorphism of the
$\vartheta$-twisted connected sum $\oM$, whose quotient is $M$.

\end{itemize}

\section{Building blocks}
\label{sec:blocks}

In \S\ref{sec:glue} we started off by using Theorem \ref{thm:acyl} to produce
ACyl Calabi-Yau 3-folds $V$ from closed Kähler 3-folds $Z$.
We now discuss how the topology of the ACyl Calabi-Yau 3-folds is related to
the topology of these building blocks, especially in the presence of an involution.
Further we discuss the second Chern class of the blocks, and the moduli space
of K3s that appear as anticanonical divisors in the blocks, as these 
will also prove relevant for finding matchings and computing the topology
of the resulting extra-twisted connected sums.

\subsection{Ordinary building blocks}
\label{subsec:ordblocks}

We begin by reviewing the results from \cite[\S 5]{cym} in the absence of
an involution. Like there, we incorporate into our notion of building block
some conditions beyond those needed to apply Theorem \ref{thm:acyl}, in order
to simplify the topological calculations later.

\begin{defn}
  \label{def:block}
A \emph{building block} is a nonsingular algebraic \mbox{3-fold} $Z$ together
with a projective morphism $f\colon Z\to \PP^1$ satisfying the following
assumptions: 
\begin{enumerate}[leftmargin=*]
\item the anticanonical class $-K_Z\in H^2(Z)$ is
  primitive.
\item $\kd=f^\star (\infty)$ is a nonsingular K3 surface and $\kd\sim -K_Z$. 
\end{enumerate}
Identify $H^2(\kd)$ with the K3 lattice $L$ %
(\ie choose a marking for $\kd$), and let $N$ denote the image of
$H^2(Z) \to H^2(\kd)$.
\begin{enumerate}[resume]
\item The inclusion $N\hookrightarrow L$ is primitive, that is,
  $L/N$ is torsion-free.
\item The group $H^3(Z)$---and thus also $H^4(Z)$---is torsion-free.
\end{enumerate}
\end{defn}
\begin{lem}[{\cite[Lemma 5-2]{cym}, \cite[Lemma 3.6]{g2m}}]
\label{lem:basicz}
If $Z$ is a building block then
\hfill
\begin{enumerate}
\item \label{it:zpi1}
$\pi_1(Z) = (0)$. In particular, $H^*(Z)$ and $H_*(Z)$ are torsion-free.
\item \label{it:zh20}
$H^{2,0}(Z) = 0$, so $N \subseteq \Pic \kd$.
\end{enumerate}
\end{lem}

We regard $N$ as a lattice with the quadratic form inherited from $L$.
In examples, $N$ is almost never unimodular,
so the natural inclusion $N\hookrightarrow N^\ast$ is not an isomorphism.
We write
\[
T = N^\perp = \{ l \in L \mid \inner{l, n} = 0 \ \  \forall\; n\in N \} .
\]
($T$ stands for ``transcendental''; in examples, $N$ and $T$ are the Picard and
transcendental lattices of a lattice polarised K3 surface.) Using $N$ primitive
and $L$ unimodular we find \mbox{$L/T\simeq N^*$}.

Let $V=Z\setminus \kd$.
Since the normal bundle of $\kd$ in $Z$ is trivial, there is an inclusion
$\iota: \kd \into V$ whose homotopy class does not depend on any choices.
We let
\begin{equation}
\label{eq:k_def}
\rho = \iota^* \colon H^2(V) \to L \quad
\text{the natural restriction map, and}
\quad
K=\ker (\rho) .
\end{equation}
It follows from (ii) of the following lemma that the image of $\rho$
equals $N$.

\begin{lem}[{\cite[Lemma 5-3]{cym}}]
  \label{lemg:Z&V}
Let $f\colon Z\to \PP^1$ be a building block. Then:
\begin{enumerate}
\item $\pi_1(V)=(0)$ and $H^1(V)=(0)$;
\item \label{itg:h2}
the class $[\kd]\in H^2(Z)$ fits in a split exact sequence
\[(0)\to \ZZ\overset{[\kd]}{\longrightarrow} H^2(Z)\to H^2(V)\to (0),\]
hence $H^2(Z)\simeq\ZZ[\kd]\oplus H^2(V)$, and the restriction homomorphism $H^2(Z)\to L$ factors through
$\rho\colon H^2(V) \to L$;
\item 
there is a split exact sequence
\[
(0) \to H^3(Z) \to H^3(V) \to T \to (0),
\]
hence $H^3(V)\simeq H^3(Z)\oplus T$;
\item \label{itg:h4}
there is a split exact sequence
\[
(0) \to N^\ast \to H^4(Z)\to H^4(V)\to (0),
\]
hence $H^4(Z)\simeq H^4(V)\oplus N^\ast$;
\item
$H^5(V) = (0)$. 
\end{enumerate}
\end{lem}

We can also use the triviality of the normal bundle of $\kd$ in $Z$ to
get a natural
inclusion $\kd\times S^1_\lnn \subset V$ up to homotopy.
Since we have not introduced any metric yet the notation $S^1_\lnn$ does not
carry much meaning beyond serving to distinguish this ``internal'' circle
factor from the ``external'' one that will soon be introduced.
Let $\drn \in H^1(S^1_\lnn)$ denote the integral generator
($\drn = \lnn^{-1}[du]$ in terms of the coordinate $u$ on $S^1_\lnn$).

\begin{lem}[{\cite[Corollary 5-4]{cym}}]
\label{lem:vres}
Let $f\colon Z \to \PP^1$ be a building block.
The natural restriction homomorphisms:
\[
\beta^m\colon H^m(V) \to H^m(\kd\times S^1_\lnn)
= H^m(\kd)\oplus \drn H^{m-1}(\kd)
\] 
are computed as follows:
\begin{enumerate}
\item $\beta^1 =0$;
\item $\beta^2 \colon H^2(V) \to H^2(\kd\times S^1_\lnn)=H^2(\kd)$ is the
  homomorphism $\rho \colon H^2(V) \to L$;
\item $\beta^3\colon H^3(V)\to H^3(\kd\times S^1_\lnn)=\drn H^2(\kd)$ is
  the composition $H^3(V) \twoheadrightarrow T \subset L$;
\item \label{itg:4} the natural surjective restriction homomorphism
  $H^4(Z)\to H^4(\kd)=\ZZ$ factors through
  $\beta^4\colon H^4(V)\to H^4(\kd\times S^1_\lnn)=H^4(\kd)=\ZZ$, and
  there is a split exact sequence:
\[
(0) \to K^\ast \to H^4(V)\overset{\beta^4}{\longrightarrow} H^4(\kd)\to (0) .
\]
\end{enumerate}
\end{lem}

When we use $M := S^1_\lnx \times V$ in a gluing construction for a twisted
connected sum, computing the cohomology of the result by Mayer-Vietoris
requires understanding of the boundary maps from cohomology of $M$ to its
cross-section $W := S^1_\lnx \times S^1_\lnn \times \kd$.
These are trivial to write down in terms of the maps in Lemma \ref{lem:vres}.
Letting $\drx \in H^1(S^1_\lnx)$ denote the generator $\xi^{-1}[d\anglex]$
of the ``external'' circle factor, we can write
  \begin{align*}
    H^m(M) &=  H^m(V) \oplus \drx H^{m-1}(V)\\
H^m(W) &=  H^m(\kd) \oplus \drn   H^{m-1}(\kd)
\oplus  \drx  H^{m-1}(\kd) \oplus \drn\drx H^{m-2}(\kd).
  \end{align*}

\begin{cor}
\label{cor:mres}
 The homomorphisms $\gamma^m \colon H^m(M) \to H^m(W)$ are computed as follows:
 \begin{enumerate}
 \item $H^1(M) = \drx  H^0(V)$, \\
	$H^1(W) =  \drx  H^0(\kd)\oplus \drn   H^0(\kd)$, and 
\[
\gamma^1=
\begin{pmatrix}
  \mathbf{1} \\0
\end{pmatrix}
\colon H^0(V) \to H^0(\kd) \oplus H^0(\kd)
\]
   is the natural isomorphism.
 \item $H^2(M) = H^2(V)$, \\
   $H^2(W)= H^2(\kd) \oplus
   \drn\drx H^0(\kd) = L\oplus \ZZ[\kd]$, and 
\[
\gamma^2=
\begin{pmatrix}
  \rho \\ 0 
\end{pmatrix}
\colon H^2(V) \to L\oplus \ZZ[\kd].
\] 
 \item \label{it:mv3} $H^3(M) =H^3(V)\oplus \drx  H^2(V)$, \\
   $H^3(W)=\drn H^2(\kd)\oplus \drx H^2(\kd)$, and
\[
\gamma^3=
\begin{pmatrix}
  \beta^3 & 0 \\ 0 & \rho 
\end{pmatrix}
\colon H^3(V) \oplus H^2(V) \to L \oplus L;
\]
\item $H^4(M) = H^4(V)\oplus \drx H^3(V)$, \\
  $H^4(W) = H^4(\kd)\oplus \drn\drx
  H^2(\kd)=H^4(\kd)\oplus L$, and
\[
\gamma^4=
\begin{pmatrix}
  \beta^4 & 0 \\ 0 & \beta^3
\end{pmatrix}
\colon H^4(V) \oplus H^3(V) \to H^4(\kd) \oplus L.
\]
 \end{enumerate}
\end{cor}

\subsection{Building blocks with involution}
\label{subsec:inv_blocks}

Next we consider involutions of the type required in~\S\ref{subsec:reducible}.
Suppose
$(Z, f, \kd)$ is a building block in the sense of Definition \ref{def:block},
and that $\tau : Z \to Z$ is a holomorphic involution such that $\kd$ is
a connected component of the fixed set of $\tau$.
Because $f \circ \tau : Z \to \PP^1$ is a fibration with
$f^\star(\infty) = \kd$,
it must be equal to $f$. Thus $\tau$ covers an involution of $\PP^1$, and
WLOG that is $(z:w) \mapsto (z:-w)$. Thus there is precisely one other fibre
$\kd' := f^\star(0)$ mapped to itself by $\tau$.

\begin{defn}
\label{def:inv_block}
Call $(Z, f, \kd, \tau)$ a \emph{building block with involution}, or
more briefly an \emph{involution block}, if $(Z,f,\kd)$ is a building block
and $\tau : Z \to Z$ is a holomorphic involution such that $\kd$ is
a connected component of the fixed set of $\tau$, and the other fixed fibre
$\kd'$ is smooth too.

As before, let $V := Z \setminus \kd$.
Let $b_3^\pm(Z)$ and $b_3^\pm(V)$ denote the rank of the
$\pm 1$-eigenlattice of the action of $\tau$ on $H^3(Z)$ and $H^3(V)$
respectively,
\[ b_3^\pm(Z) := \rk H^3(Z)^{\pm \tau}, \quad
b_3^\pm(V) := \rk H^3(V)^{\pm \tau}  \]
(which will not be confused with (anti-)self-dual parts since the degree is
odd). Further, since the quotient by the sum of the invariant and
anti-invariant subspaces is a 2-elementary group we can let
\[ s := \dim_{\Z_2}
\displaystyle\frac{H^3(V)}{H^3(V)^\tau \oplus H^3(V)^{-\tau}} . \]
(To see what $s$ represents, it may be helpful to think about two different reflections on $\Z^2$: $(x,y) \mapsto (-x, y)$ has $s = 0$, while $(x,y) \mapsto (y,x)$ has $s = 1$.)

We call the involution block \emph{pleasant} if $K = 0$, \ie the restriction
map
\begin{equation}
\label{eq:k0}
H^2(V) \into H^2(\kd)
\end{equation}
is injective, and
\begin{equation}
\label{eq:pleasant}
s = b_3^-(V) .
\end{equation}
\end{defn}

When we describe examples of blocks with involution, the data we specify that
relates to the involution is $b_3^+(Z)$ and whether the block is pleasant.
Since $H^3(V) \cong H^3(Z) \oplus T$, so that
$H^3(V)^\tau \cong H^3(Z)^\tau \oplus T$ over $\bbq$ we can then recover
\begin{equation}
b_3^+(V) = b_3^+(Z) + 22 - \rk N .
\end{equation}

We will see in \S\ref{sec:top} that the
conditions \eqref{eq:k0} and \eqref{eq:pleasant} make it much easier to grasp
the cohomology of the extra-twisted connected sums, and in \S\ref{sec:sfblocks}
and \S\ref{sec:k3blocks} that the involution blocks we can most readily write
down do in fact satisfy this pleasantness condition.

Clearly $H^3(V) \subseteq \half H^3(V)^\tau \oplus \half H^3(V)^{-\tau}$.
The projections onto the components induce injective maps
$H^3(V) / H^3(V)^\tau \oplus H^3(V)^{-\tau}
\into \left(\half H^3(V)^{\pm\tau}\right)/ H^3(V)^{\pm\tau}$, so
\begin{equation}
\label{eq:sbound}
s \leq \min(b_3^+(V), b_3^-(V)) . 
\end{equation}
Alternatively, $s$ can be described as the dimension of the image of
$\Id + \tau^* : H^3(V; \bbz_2) \to H^3(V; \bbz_2)$, and \eqref{eq:sbound}
as a consequence of the fact that $\Id + \tau^*$ is 0 on
$H^3(V)^{\pm\tau} \otimes \bbz_2$.

Note that it is not generally the case that
$H^3(V)^\tau \cong H^3(Z)^\tau \oplus T$ over $\bbz$. In particular, $s$ need
not equal the $\Z_2$ rank of $H^3(Z)/(H^3(Z)^\tau \oplus H^3(Z)^{-\tau})$.

\begin{rmk}
\label{rmk:nonsymp_fibre}
The condition that the second fixed fibre $\kd'$ is smooth is not crucial
to the construction, but simplifies topological calculations.
Since $Z$ has a unique (up to scale) holomorphic 3-form with pole along~$\kd$,
that must be preserved by $\tau$. The action of $\tau$ on
$\kd'$ must therefore be by a non-symplectic involution in the sense described
in \S\ref{subsec:k3nonsymp}.

Other fibres of $f$, in particular $\kd$, need not admit a non-symplectic
involution (see Example~\ref{ex:p3cover}).
\end{rmk}

The fixed set of $\tau$ in $\kd'$ is a smooth holomorphic curve $C$.
The quotients $Z^0 := Z/\tau$ and $V^0 := V/\tau = Z^0 \setminus \kd$
have orbifold singularities along the image of $C$.
On the other hand, according to the theory of non-symplectic involutions
summarised in \S\ref{subsec:k3nonsymp},
$Y := \kd'/\tau$ is a smooth (in fact rational) surface;
$\kd' \to Y$ is a double cover branched over $C$, and $C \in |{-}2K_Y|$.
In particular, because $C$ is even in $H_2(Y)$, the image of the
restriction map $H^2(Z^0; \Z_2) \to H^2(C; \Z_2)$
is contained in the kernel of the integration map $H^2(C; \Z_2) \to \Z_2$.
Thus if we let
\begin{align*}
m &:= \rk(H^2(Z^0; \Z_2) \to H^2(C; \Z_2)), \\
k &:= \# (\textrm{connected components of }C) - 1
\end{align*}
then $m \leq k$.

\begin{lem}
\label{lem:pleasant}
If $K = 0$ then
\[ b_3^-(V) - s = \dim_{\Z_2} T_2 H^3(Z^0) + k-m . \]
In particular, an involution block is pleasant if and only if $K = 0$,
$H^3(Z^0)$ is torsion-free and $m = k$.
\end{lem}

\begin{proof}
Note that $b_3^+(V) = b_3(V^0)$.
If $K=0$ then $\tau$ acts trivially on $H^4(V) \cong \Z$, so $b_4(V^0) = 1$.

By Lee--Weintraub \cite[Theorem 1]{lee95} there exists a long exact sequence
\begin{equation}
\label{eq:coverseq}
H^k(V^0; \bbz_2) \stackrel{\pi^*}{\to} H^k(V; \bbz_2)
\stackrel{I}{\to} H^k(V^0,C; \bbz_2) \stackrel{\cup w_1}{\to}
H^{k+1}(V^0;\bbz_2) ,
\end{equation}
where $I$ is fibre-wise integration, and the connecting map
$H^k(V^0,C; \bbz_2) \to H^{k+1}(V^0;\bbz_2)$ is the cup product with
$w_1 \in H^1(V^0 \setminus C; \Z_2)$ of the double cover. (If $C$ were empty,
this would just be the Gysin sequence of the double cover $\pi: V \to V^0$
regarded as the unit $S^0$-bundle in a real line bundle.)

First note $H^5(V^0; \Z_2) \cong H^5(V^0, C; \Z_2) \cong
H_1(V^0 \setminus C, S^1 \times \kd; \Z_2)$ is isomorphic to the cokernel of
the push-forward $H_1(S^1 \times \kd; \Z_2) \to H_1(V^0 \setminus C; \Z_2)$
of the inclusion of the $S^1 \times \kd$ as the boundary of $V^0$.
Since $\pi_1(S^1 \times \kd) \to \pi_1(V^0 \setminus C)$ is surjective,
we find that $H^5(V^0; \Z_2)$ is trivial.

Now, since $b_4(V^0) = 1$, universal coefficients implies that the rank of
$H^4(V^0;\Z_2) \cong H^4(V^0,C;\Z_2)$ is one more than that of $T_2H^4(V^0)$.
By the exactness of \eqref{eq:coverseq}, we must have that in fact
$T_2H^4(V^0) = 0$, and $I_4 : H^4(V;\Z_2) \to H^4(V^0,C;\Z_2)$ is an
isomorphism.

Next we argue that the composition of $I_3$
with the push-forward $p : H^3(V^0,C;\Z_2) \to H^3(V^0;\Z_2)$ is surjective.
Since $p$ is surjective, it suffices to prove that $\cup w_1$ maps
$\ker p$ %
onto $H^4(V^0;\Z_2)$.
Equivalently, we need the composition of the snake map
$H^2(C;\Z_2) \to H^3(V^0,C;\Z_2)$ with $\cup w_1$ to be non-trivial.
The further composition with the restriction $H^4(V^0;\Z_2) \to H^4(Y; \Z_2)$
must in fact be non-trivial because the snake map
$H^2(C; \Z_2) \to H^3(Y, C; \Z_2)$ and $w_1$ both are.
Hence $p \circ I_3$ is surjective as claimed.

Because $\pi^* \circ p \circ I= \Id + \tau^*$, it follows that
$H^3(V^0; \Z_2)$ has the same image under $\pi^*$ as under $\Id + \tau^*$.
Hence $s = \rk \pi^* = \dim \ker I_3$.
The dimension of $H^3(V^0,C;\Z_2)$ can be expressed as 
$b_3(V^0) + (k+1-m) + \dim_{\Z_2} T_2H^3(Z^0)$, so
\[ b_3^-(V) - s
= (b_3(V) - b_3(V^0)) - (b_3(V) - \dim_{\Z_2} T_2H^3(V^0,C;\Z_2)+1)
= k -m + \dim_{\Z_2} T_2 H^3(Z^0) \]
as desired.
In particular, $b_3^-(V) = s$ if and only if equality holds in $k \geq m$
and $T_2H^3(Z^0)$ is trivial. The latter condition is equivalent to
$H^3(Z^0)$ being torsion-free, since $H^3(Z)$ is torsion-free implies
that the only possible torsion in $H^3(Z^0)$ is 2-torsion.
\end{proof}

\begin{rmk}
In this paper, we will only apply Lemma \ref{lem:pleasant} in cases where
$C$ is connected, so the condition $m = k$ is automatically satisfied
(both are 0). As a consequence of this, the polarising lattice $N$ of the
resulting building blocks with involution will always be completely even, in
the sense that the product of any two elements is even; this is because
$N$ embeds into the sublattice of $H^2(\kd')$ that is fixed by the
non-symplectic involution, which is totally even when the fixed locus $C$ is
connected (see Lemma \ref{lem:pdc}).
\end{rmk}

From now on we assume \eqref{eq:k0}.
This implies in particular that $\tau$ acts trivially on $H^2(V) \cong N$ and
$H^4(V) \cong H^4(\kd) \cong \bbz$, so $V^0$ has the same 
Betti numbers as $V$ except in the middle degree.
Since $\pi : V \to V^0$ is a double cover branched over $C$, we find
\[ \chi(V) = 2\chi(V^0) - \chi(C), \]
from which we deduce
\[ b_3(V) = 2b_3(V^0) - 2 - \rho + \chi(C) . \]
Similarly
\[ \chi(Z) = 2\chi(Z^0) - \chi(C) - \chi(\kd) \]
implies (using $\chi(Z) = 4 + 2\rho - b_3(Z)$ etc) that
\begin{equation}
\label{eq:b3+z}
b_3(Z) = 2b_3^+(Z^0) + 20 - 2\rho - \chi(C) .
\end{equation}
Now let
\[ M := S^1_\lnx {\times} V / a {\times} \tau. \]
The rational cohomology
of $M$ is simply the $\tau$-invariant part of $H^*(S^1_\lnx \times V)$. We see
from Lemma \ref{lemg:Z&V} and our description of $\tau^*$ that
\begin{align*}
b_1(M) &= 0 & 
b_2(M) &= b_2(V) &
b_3(M) &= b_2(V) + b_3^+(V) \\
b_4(M) &= b_3^-(V) + 1 &
b_5(M) &= 1 &
b_6(M) &= 0
\end{align*}
We can also readily compute the integral cohomology of $M$ from
the Mayer-Vietoris sequence
\begin{equation}
\label{eq:mv_maptorus}
\cdots \to H^{k-1}(V) \to H^k(M) \to H^k(V)
\stackrel{\Id - \tau^*}{\longrightarrow} H^k(V) \to \cdots
\end{equation}

\begin{lem} \label{lem:H_mtv} \hfill
\begin{enumerate}
\item $\bbz \isom H^1(M)$
\item \label{it:h2mtv} $H^2(M) \isom H^2(V) = N$
\item $0 \to H^2(V) \to H^3(M) \to H^3(V)^\tau \to 0$
\item $0 \to H^3(V)/(\Id - \tau^*)H^3(V) \to H^4(M) \to \bbz \to 0$ 
\item $\bbz \isom H^5(M)$
\item $H^6(M) = 0$
\end{enumerate}
\end{lem}

\noindent
Note that the only torsion in $H^*(M)$ is
\[
\Tor H^4(M) \; \cong \;
H^3(V)^{-\tau}/(\Id - \tau^*)H^3(V) \; \cong \; \bbz_2^{b_3^-(V) {-} s} ;
\]
thus $H^*(M)$ is torsion-free when the involution block $Z$ is pleasant.

We also need to understand the restriction map to the cross-section of the
cylindrical end, $H^*(M) \to H^*(T^2 \times \kd)$,
where $T^2 := S^1_\lnx {\times} S^1_\lnn / a \times a$. In particular we need
to describe the image.
Over $\bbq$, the image is the same
as for the maps in Corollary \ref{cor:mres},
\eg $H^3(M;\bbq) \to H^3(T^2 \times \kd; \bbq)$ has image
$\drx N \oplus \drn T$, but working with integer coefficients is more complicated.

\begin{notn}
\label{notn:abuse}
Here we are abuse notation slightly and denote classes in $H^*(T^2)$ by their
pull-backs to $H^*(S^1_\lnx {\times} S^1_\lnn)$; thus $2\drx$
and $2\drn \in H^1(T^2)$ are primitive classes, but they generate a subgroup of
index 2, and $H^2(T^2)$ is generated by $2\drx\drn$.
\end{notn}

\begin{lem}\hfill
\label{lem:mtres}
\begin{enumerate}
\item $H^2(M) \to H^2(T^2 \times \kd)$ is an isomorphism onto $N$.
\item $H^3(M) \to H^3(T^2 \times \kd)$ has image contained in
\[ I^3 := \{\drx n + \drn t :
\; n \in N, \, t \in T, \, n+t = 0 \mod 2L\} . \]
If $s = b_3^-(V)$ then equality holds.
\item $H^4(M) \to H^4(T^2 \times \kd)$ has image
$2\drx\drn T \oplus H^4(\kd)$.
\end{enumerate}
\end{lem}

\begin{proof}
First part is obvious because $H^2(M) \to H^2(V)$ is an isomorphism.
Last part is obvious because the Mayer-Vietoris boundary map in the
computation of $H^*(T^2 \times \kd)$ maps
$H^k(S^1_\lnx \times \kd) \to H^{k+1}(T^2 \times \kd)$ by $x \mapsto 2\drx x$.

$I^3$ is precisely the set of integral classes in the rational image
$\drx N \oplus \drn T \subseteq H^3(T^2 \times \kd; \bbq)$,
so the image of $H^3(M)$ is a finite index subgroup of $I^3$.
The long exact sequence of cohomology of $M$ relative to $T^2 \times \kd$
gives $I^3/\im H^3(M) \into H^4_{cpt}(M) \cong H_3(M)$. Thus
$I^3/\im H^3(M) \into \Tor H_3(M) \cong \Tor H^4(M)$, which is trivial
if $s = b_3^-(V)$.
\end{proof}

\subsection{The second Chern class}

When we compute characteristic classes of extra-twisted connected sums in
\S\ref{subsec:p},
it will prove convenient to present the second Chern class of a building block
with $K = 0$ in the following form:
\begin{equation}
\label{eq:c2cbar}
c_2(Z) = g(\bar c_2(Z)) + 24 h ,
\end{equation}
for some $\bar c_2(Z) \in N^*$ and $h \in H^4(Z)$ such that the restriction of $h$
to $\kd$ is the positive generator of $H^4(Z)$,
where $g : N^* \to H^4(Z)$ is dual to the restriction
$H^2(Z) \to N \subset H^2(\kd)$.
Alternatively, we can describe $g$ as follows: for $\bar c \in N^*$ and
any preimage $x$ of $\bar c$ under the duality map $\flat : H^2(\kd) \to N^*$
(which is surjective since $H^2(\kd)$ is unimodular),
\[ g (\bar c) = i_* \partial (\drn x) , \]
where $\partial : H^3(S^1_\lnn \times \kd) \to H^4_{cpt}(V)$ is the snake
map in the long exact sequence of the cohomology of $V$ relative to its
boundary, and $i_* : H^4_{cpt}(V) \to H^4(Z)$ is the push-forward of the
inclusion $V \into Z$.

For a building block with $K = 0$, Lemma \ref{lemg:Z&V}\ref{itg:h4} and
\ref{lem:vres}\ref{itg:4} give exactness of
\[ 0 \to N^* \stackrel{g}{\to} H^4(Z) \to H^4(\kd) \to 0 . \]
Since the image of $c_2(Z)$ in $H^4(\kd)$ is $\chi(\kd) = 24$ times the
generator, $c_2(Z)$ can then always be written in the form \eqref{eq:c2cbar}.
This presentation is not unique, but we will make convenient choices for
$\bar c_2(Z)$ and $h$ for each class of building blocks.
(If $K \not=0$ then we cannot in general write $c_2(Z)$ in the form
\eqref{eq:c2cbar}, and would need to make some further arbitrary choices to
capture the components in a direct summand isomorphic to $K^*$.) 

In the case of a building block $Z$ with involution $\tau$, we describe the
second Chern class in the same way, but in addition require the class $h$ to
be $\tau^*$-invariant. In the examples we care about, we can in fact do more:
we can essentially pick $h$ to be represented by a $\tau$-invariant integral
cochain.

Let us discuss more generally how to measure the failure of a $\tau$-invariant
class $h \in H^4(Z)$ to be represented by a $\tau$-invariant cochain.
For any chain representative $\alpha$ we can write $\alpha - \tau^*\alpha = d\beta$ for some 3-cochain $\beta$. Then $\beta + \tau^*\beta$ is closed, and
the resulting class
\begin{equation}
B(h) := [\beta + \tau^*\beta] \in H^3(Z)
\end{equation}
depends on the choices only modulo the image of $\Id + \tau^*$ on $H^3(Z)$.

We can relate this to the cohomology of $H^4(\mtt{Z})$. By the Mayer-Vietoris
sequence analogous to \eqref{eq:mv_maptorus}, $h \in H^4(Z)$ has a pre-image
$\tilde h \in H^4(\mtt{Z})$, and such a pre-image
can be pulled back by $\pi : S^1 \times \kd \to \mtt{Z}$.
The $H^4(Z)$ component of
$\pi^*\tilde h \in H^4(S^1 \times \kd) \cong H^4(Z) \oplus H^3(Z)$ is just $h$
itself, while the $H^3(Z)$-component depends
on the choice of $\tilde h$. By the Mayer-Vietoris sequence, the kernel of
$H^4(\mtt{Z}) \to H^4(Z)$ is the image of the snake map
$\delta : H^3(Z) \to H^4(\mtt{Z})$, whose composition with $\pi^*$ equals 
$\Id + \tau^* : H^3(Z) \to H^3(Z) \subset H^4(S^1 \times Z)$.
Thus the $H^3(Z)$-component of $\pi^* \tilde h$ depends on the choice of
$\tilde h$ up to the image of $\Id + \tau^*$, and in fact it equals $B(h)$. 

\begin{rmk}
\label{rmk:tilde_map}
For any $h = [\alpha] \in H^4(Z)$, the $\tau$-invariant cochain
$\alpha + \tau^*\alpha$ defines a class in $\widetilde{2h} \in H^4(\mtt{Z})$.
This depends only on $h$ (and since we assume $H^4(Z)$ is torsion-free,
in fact only on $2h$), and is a pre-image of $2h$ such that
$\pi^* \mtlift{2h} \in H^4(S^1 \times \kd)$ has no $H^3(Z)$-component
(but if $H^4(\mtt{Z})$ has 2-torsion then it is not the unique such pre-image).
However, even if $h$ is $\tau$-invariant, $\mtlift{2h}$ need not be even
in $H^4(\mtt{Z})$. Its parity is given by $\partial B(h)$.
\end{rmk}

The fact that $\kd \subset Z$ is fixed by $\tau$ allows us to define
a refinement of $B(h)$ supported away from~$\kd$, which will also play a role
in \S\ref{subsec:p}.
We can always choose a cochain representative $\alpha$ of $h$ to be $\tau$-invariant in a neighbourhood of $\kd$. Thus $\alpha - \tau^*\alpha$, which is exact on $Z$, has compact support in $V$. Because $H^4_{cpt}(V) \into H^4(Z)$
(since $H^3(\kd) = 0$) we can write $\alpha - \tau^*\alpha = d\beta$ for
a compactly supported cochain $\beta$ on $V$, and consider
\begin{equation}
\label{eq:h_inv_fail}
\wh B(h) = [\beta + \tau^*\beta] \in H^3_{cpt}(V).
\end{equation}
This is again defined up to the image of $\Id + \tau^*$ on $H^3_{cpt}(V)$,
and we can again relate it to the mapping torus.
For a pre-image $\tilde h \in H^4(\mtt{Z})$ of $h$, we can pick a cochain
representative $\tilde \alpha$ that near  $S^1 \times \kd$ is a pull-back of a
cochain on $\kd$. If we pick a cochain representative $\alpha$ of $h$ that near
$\kd$ is a pull-back of that same cochain on $\kd$, then the difference of
the pull-backs of $\tilde \alpha$ and $\alpha$ to $S^1 \times \kd$ has compact
support in $S^1 \times V$.
The $H^3_{cpt}(V)$ component of the resulting class in
$H^4_{cpt}(S^1 \times V)$ corresponds to $\wh B(h)$.

If a $\tau$-invariant class $h$ has a $\tau$-invariant cochain representative,
then certainly $B(h) = \wh B(h) = 0$. For our examples of involution blocks,
we will not be able to argue that we can choose a pre-image $h \in H^4(Z)$
of the generator of $H^4(\kd)$ to have a $\tau$-invariant cochain representative, but we will be able to pick it to be the Poincar\'e dual of a submanifold
that is preserved by $\tau$.

\pagebreak

\begin{lem}
\label{lem:h_inv_pd}
Suppose $h = PD(C)$ for a $\tau$-invariant submanifold $C \subset \kd$.
Then $B(h) = 0$.

If $C$ is transverse to $\kd$ then also $\wh B(h) = 0$.
\end{lem}

\begin{proof}
The pre-image of $C$ in $\mtt{Z}$ is simply $S^1 \times C$.
As a pre-image of $h$ in $H^4(\mtt{Z})$, we can take
$\tilde h = PD(S^1 \times C)$. Then certainly the pull-back of $\tilde h$
to $S^1 \times Z$ has no $H^3(Z)$ component, so $B(h) = 0$.

For the last claim, take a $\tau$-invariant tubular neighbourhood $U \subset Z$
of $C$, pick a cochain representative $\tilde \alpha$ of the above $\tilde h$
with support in $\mtt{U}$, and a cochain representative $\alpha$ of $h$
supported in $U$. Because $C$ is transverse to $\kd$, we can in addition
take both cochains to be pull-backs of the same representative of
$PD(C \cap \kd)$ near $\kd$, so that the difference $\alpha'$ of the
pull-backs to $S^1 \times Z$ is supported in $S^1 \times (U \cap V)$.
Since the image of $[\alpha']$ in $H^4_{cpt}(S^1 \times U)$ is clearly zero
and $H^4_{cpt}(S^1 \times (U \cap V)) \to H^4_{cpt}(U)$ is injective,
it follows that $[\alpha'] = 0$, and in particular the $H^3_{cpt}(Z)$-component $\wh B(h)$ of its image in $H^4_{cpt}(S^1 \times Z)$ vanishes.
\end{proof}

\subsection{Moduli of lattice-polarised K3s}

The final property of building blocks that we will wish to study concerns the
relation to moduli spaces of K3s. Because a K3 surface $\kd$ is simply
connected, its Picard group $\Pic \kd$ is isomorphic to
$H^2(\kd;\Z) \cap H^{1,1}(\kd;\bbc)$. The Picard lattice is $\Pic \kd$ equipped
with the restriction of the intersection form of $H^2(\kd;\Z)$.

Fix a non-singular lattice $L$ of signature $(3,19)$. A marking of a K3 surface
$\kd$ is an isomorphism $\hdg : H^2(\kd;\Z) \to L$. The Picard lattice of a
marked K3 is thus identified with a (primitive) sublattice of~$L$.
Meanwhile, the \emph{period} of the marked K3 is the image in $\PP(L_\bbc)$
of the 1-dimensional subspace $H^{2,0}(\kd;\bbc) \subset H^2(\kd;\bbc)$.
It lies in the subset
$\{ \Pi \in \PP(L_\bbc) : \Pi^2 = 0, \Pi \overline\Pi > 0\}$.
By the Torelli theorem, the moduli space of marked K3s is (modulo some niceties
about the choice of polarisations that do not concern us) isomorphic to an open
subset of this period domain.
 
Crucially, the K3 surfaces $\kd$ that appear in a building block $Z$ always
belong to a more restricted moduli space. According to
Lemma \ref{lem:basicz}\ref{it:zh20}, the Picard lattice of $\kd$ must
contain the polarising lattice $N$ of $Z$. Therefore the period $\Pi$ of
the marked K3 must be orthogonal to $N$. We say that $\kd$ is ``$N$-polarised''.

Equivalently, we can think of the period as the positive definite
2-plane $\Pi \subset L_\R$ spanned by the images of real and imaginary parts
of $H^{2,0}(\kd;\bbc)$.
If $\kd$ is $N$-polarised, then $\Pi$ belongs to the \emph{Griffiths domain}
\begin{equation}
\label{eq:griffiths}
D_N := \{ \textrm{ positive-definite 2-planes } \Pi \in Gr(2, N^\perp)\}.
\end{equation}
A principle that is valid for all building blocks we consider in this paper
is that they come in families, such that a generic $N$-polarised K3 appears
as an anticanonical divisor in some element of the family, and moreover we have
some control on the size of the ample cone
(see Proposition~\ref{prop:generic_fano}). In \S\ref{sec:match} we find
on the one hand that this genericity property is often enough for producing
matchings between some elements of a pair of families. On the other hand, we
find also that in some cases one needs to know that even generic elements of
a more restricted moduli space of K3s (with a larger polarising lattice
$\Lambda \supset N$) appear as anticanonical divisors.
We capture these conditions in the following definition.

\begin{defn}
\label{def:generic}
Let $N \subset L$ be a primitive sublattice, $\Lambda \subset L$ a primitive overlattice of $N$, and $\Amp_\fbb$ an open subcone
of the positive cone in $N_\bbr$. We say that a family of building blocks
$\fbb$ with polarising lattice $N$ is \emph{$(\Lambda, \Amp_\fbb)$-generic} if
there is a subset $U_\fbb$ of the Griffiths domain $D_\Lambda$
with complement a countable union of complex analytic submanifolds of
positive codimension with the property that:
for any $\Pi \in U_\fbb$ and $\kclass \in \Amp_\fbb$ there is a building block
$(Z,\kd) \in \fbb$ and a marking $\hdg : L \to H^2(\kd; \ZZ)$ such that
$\hdg(\Pi) = H^{2,0}(\kd)$, and $\hdg(\kclass)$ is the image of the restriction
to $\kd$ of a Kähler class on $Z$.
\end{defn}

\subsection{Presentation of data}
\label{subsec:data}

To finish the section, let us summarise what we consider to be the key pieces
of data of a building block, which will be sufficient to compute the
topological invariants of the resulting extra-twisted connected sums that we
are interested in.
\begin{itemize}
\item The kernel $K$ of $H^2(V) \to H^2(\kd)$ and (for involution blocks)
whether the block is pleasant,
\item $b_3(Z)$ and---in the case of blocks with involution---$b_3^+(Z)$,
\item the form on the polarising lattice $N$,
\item an element $\bar c_2(Z) \in N^*$ encoding information about $c_2(Z)$
as in \eqref{eq:c2cbar}, and $\wh B(h) \in H^4_{cpt}(V)$,
\item an open cone $\Amp \subset N_\bbr$ such that the family of blocks
is $(N, \Amp)$-generic in the sense of Definition~\ref{def:generic}.
\end{itemize}
Tables \ref{table:species1}, \ref{table:ordblocks} and \ref{table:invblocks}
will include this and some auxiliary data.
In fact, all the ordinary blocks included in the tables will have $K = 0$,
and all the involution blocks will be pleasant, with $\wh B(h) = 0$.

We always use the same basis of $N$ for describing the form on $N$,
$\bar c_2(Z)$ and $\Amp$. For all blocks we consider, it turns out to be
possible to choose a basis for $N$ that consists of the edges of $\Amp$, and in
the tables we always use such a basis.

Note that this means that the sign of $\bar c_2(Z)$ is meaningful.
Multiplying all elements of the basis by $-1$ preserves the intersection form,
but reverses the signs of $\bar c_2(Z)$ and $\Amp$ together. For instance, if
$N$ has rank 1, choosing $\Amp$ amounts to designating one of the two
generators of $N$ to be positive. Whether $\bar c_2(Z)$ evaluates to, say, 2 or
$-2$ mod 24 on the positive generator then has an invariant meaning, and can
affect the homeomorphism class of the extra-twisted connected sums built from
the block.

\begin{rmk}
\label{rmk:conj}
If $Z$ is a building block then so is its complex conjugate $\overline Z$,
\ie the same smooth manifold, but with the complex structure $J$ replaced
by $-J$.
This reverses the orientation of $Z$, but preserves it on $\kd$, so the
sign of the dual map $g : N^* \to H^4(Z)$ is reversed. At the same time,
the Kähler cone of $\kd$ is multiplied by $-1$, so $Z$ and $\overline Z$
are indistinguishable by our topological data. This is quite reasonable, since
in many cases it is clearly possible to deform $Z$ to a building block with
a real structure, and hence to its complex conjugate.
\end{rmk}

\section{Building blocks from semi-Fano 3-folds}
\label{sec:sfblocks}

The main method we use in this paper for producing examples of building blocks
is to blow up Fano 3-folds or semi-Fano 3-folds. Let us briefly recall some
terminology. A projective 3-fold $Y$ is \emph{weak Fano} if the anticanonical
bundle $-K_Y$ is big and nef, \ie if the sections of a sufficiently high power
of $-K_Y$ define a morphism $\phi$ of $Y$ to projective space, whose image $X$
(\emph{the anticanonical model}) is 3-dimensional. If $\phi$ is an
embedding, then $Y$ is Fano, \ie $-K_Y$ is ample. In the terminology from
\cite[Definition 4.11]{cym}, for $Y$ to be semi-Fano means that the fibres of
$\phi$ have dimension at most~1.

\subsection{Ordinary building blocks from Fano 3-folds}
\label{subsec:ordfano}

Let us first summarise the results from \cite{cym} concerning how to
construct building blocks (without involution) from Fano or semi-Fano 3-folds,
along with some previously studied examples of applying this to Fano 3-folds
mainly of Picard rank 1 or 2.

\begin{prop}[{\cite[Prop 4.24]{cym}}]
Let $Y$ be a closed Kähler 3-fold with an anticanonical pencil $|\kd_0: \kd_1|$
with smooth base locus $C$.
Let $Z$ be the blow-up of $Y$ along $C$, and let $\kd \subset Z$ be the proper
transform of $\kd_0$.
Then the image $N$ of $H^2(Z) \to H^2(\kd)$ equals the image
of $H^2(Y) \to H^2(\kd_0)$,
while $\ker (H^2(Z) \to H^2(\kd)) \cong \Z \oplus \ker(H^2(Y) \to H^2(\kd))$.
Further $\Tor H^3(Z) \cong \Tor H^3(Y)$, and the image of the Kähler cone of 
$Z$ in $H^{1,1}(\kd; \R)$ contains the image of the Kähler cone of $Y$.
\end{prop}

\begin{constr}
\label{constr:block}
Let $Y$ be a closed Kähler 3-fold such that
\begin{enumerate}
\item $H^3(Y)$ torsion-free,
\item an anticanonical pencil $|\kd_0: \kd_1|$ with smooth base locus $C$, and
\item the image $N$ of $H^2(Y) \to H^2(\kd_0)$ is primitive.
\end{enumerate}
Let $Z$ be the blow-up of $Y$ along $C$, and let $\kd \subset Z$ be the proper
transform of $\kd_0$. Then $(Z, \kd)$ is a building block, with polarising
lattice $N$, and $K \cong \ker H^2(Y) \to H^2(\kd_0)$.
\end{constr}

\begin{prop}
\label{prop:block_from_fano}
If $Y$ is a semi-Fano 3-fold whose anti-canonical ring is generated in degree 1
then conditions \textup{(i)} and \textup{(ii)} in
Construction \ref{constr:block} are satisfied, and $K = 0$.
\end{prop}

\begin{proof}
See \cite[Remark 4.10 and Proposition 5.7]{cym}.
\end{proof}

\pagebreak[2]

For the anticanonical ring of $Y$ to be generated in degree 1 is equivalent to
the anticanonical model $X$ of $Y$ to have very ample $-K_X$.
The only two classes of Fano 3-folds $Y$ for which $-K_Y$ fails to be very
ample are number 1 in the Mori-Mukai list of rank 2 Fanos, and the product of
$\PP^1$ with a degree 1 del Pezzo surface.
The possible singular anticanonical models $X$ for which $-K_X$ fails to be
very ample are listed by Jahnke-Radloff \cite[Theorem 1.1]{jahnke_radloff}.

Meanwhile, all known examples of semi-Fano 3-folds $Y$ have torsion-free
$H^3(Y)$. Thus we can justifiably say that Construction \ref{constr:block}
can be applied to produce a building block from almost any semi-Fano 3-fold.

Now let us proceed to explain how to obtain the other data listed
in~\S\ref{subsec:data}.

\begin{lem}[{\cite[Lemma 5.6]{cym}}]
\label{lem:b3}
$b_3(Z) = b_3(Y) + b_1(C) = b_3(Y) - \chi(C) + 2 = b_3(Y) -K_Y^3 +2$.
\end{lem}

\begin{lem}[{\cite[Proposition 5.11]{cym}}]
\label{lem:c2blowup}
Let $Z$ be a building block obtained from a closed Kähler 3-fold $Y$ as in
Construction \ref{constr:block}, and let $\pi : Z \to Y$ denote the blow-up
map.
Let $h \in H^4(Z)$ be the Poincar\'e dual to a $\PP^1$ fibre of $\pi$,
and let $\pi_! : H^4(Z) \to H^4(Y)$, $g : N^* \to H^4(Z)$ and $g_Y : N^* \to H^4(Y)$ be the Poincar\'e dual to $\pi^* : H^2(Y) \to H^2(Z)$ and the restrictions
$H^2(Z) \to N$ and $H^2(Y) \to N$ respectively.
Then $c_2(Z) = g(\bar c_2(Z)) + 24h$, for
\[ \bar c_2(Z) = g_Y^{-1}\pi_!c_2(Z) . \]
\end{lem}

\newcommand{\firstblocks}{
\begin{table}[tb]
\renewcommand{\arraystretch}{1.1}
\[
\begin{array}[t]{lcccccc} \toprule
r & -K_Y^3 & b_3(Y) & b_3(Z) & N & \bar c_2(Z) \\ \midrule
4 & 4^3         &   0 &  66 &  (4) & 22 \\ 
3 & 3^3 \cdot 2 &   0 &  56 &  (6) & 26 \\ 
2 & 2^3         &  42 &  52 &  (2) & 16 \\ 
2 & 2^3 \cdot 2 &  20 &  38 &  (4) & 20 \\ 
2 & 2^3 \cdot 3 &  10 &  36 &  (6) & 24 \\ 
2 & 2^3 \cdot 4 &   4 &  38 &  (8) & 28 \\ 
2 & 2^3 \cdot 5 &   0 &  42 & (10) & 32 \\ 
1 &           2 & 104 & 108 &  (2) & 26 \\ 
1 &           4 &  60 &  66 &  (4) & 28 \\ 
1 &           6 &  40 &  48 &  (6) & 30 \\ 
1 &           8 &  28 &  38 &  (8) & 32 \\ 
1 &          10 &  20 &  32 & (10) & 34 \\ 
1 &          12 &  14 &  28 & (12) & 36 \\ 
1 &          14 &  10 &  26 & (14) & 38 \\ 
1 &          16 &   6 &  24 & (16) & 40 \\ 
1 &          18 &   4 &  24 & (18) & 42 \\ 
1 &          22 &   0 &  24 & (22) & 46 \\ 
\bottomrule
\end{array} 
\]
\smallskip
\caption{Rank 1 Fano blocks}
\label{table:species1}
\end{table}}

\firstblocks

This description of $c_2(Z)$ is convenient when coupled with the following claim.

\begin{lem}[{\cite[(5-13)]{cym}}]
\label{lem:c2recurse}
If $\pi : Z \to Y$
is the blow-up of some closed Kähler 3-fold $Y$ along a curve $C$ contained in
an anticanonical divisor $\kd$, then
\[ \pi_!(c_2(Z) + c_1(Z)^2) = c_2(Y) + c_1(Y)^2 . \]
\end{lem}

Finally, for the matching problem it is an important principle that our blocks
come in families, such that a generic $N$-polarised K3 surface appears as
an anticanonical divisor in some element of the family.

\pagebreak[2]

\begin{prop}[{\cite[Proposition 6-9]{cym}}]
\label{prop:generic_fano}
Let $Y$ be a semi-Fano 3-fold with Picard lattice $N$ (\ie $N$ is the image
of $H^2(Y) \to H^2(\kd)$ for an anticanonical $\kd \subset Y$), and let
$\sff$ be the set of semi-Fano 3-folds in the deformation type of $Y$.
Then there is an open cone $\Amp_\sff \subset N_\R$ such that
$\sff$ is $(N, \Amp_\sff)$-generic in the sense
of Definition \ref{def:generic}.

In particular, the set of building blocks produced from $\sff$ by Construction
\ref{constr:block} is also $(N, \Amp_\sff)$-generic.
\end{prop}

Note however that Proposition \ref{prop:generic_fano} is limited in that it
does not tell us what $\Amp_\sff$ is. In the examples we can work it out from
the explicit description of the semi-Fanos.

\begin{ex}
\label{ex:spec1}
Table \ref{table:species1} summarises the key data of Fano 3-folds of rank 1
and the resulting building blocks (\cf \cite[Table 1]{cym}).
Apart from the data highlighted in \S\ref{subsec:data}, we include in the 
table the index~$r$ (\ie the
largest integer such that $-K_Y = rH$ for some $H \in \Pic Y$),
the anticanonical degree~$-K_Y^3$, and $b_3(Y)$.

$b_3(Z)$ is simply obtained from the preceding data by Lemma \ref{lem:b3}. 
In the rank 1 case, $\bar c$ is also easily determined as follows:
For any Fano one has $c_2(Y)(-K_Y) = 24$, so if $-K_Y = rH$
then
\begin{equation}
\label{eq:c2onH}
(c_2(Y) + c_1(Y)^2) H = \frac{24-K_Y^3}{r} .
\end{equation}
So Lemma \ref{lem:c2blowup} implies that with respect to the basis of $H^4(Z)$
dual to $H$, $\bar c$ is represented by the coordinate $\frac{24-K_Y^3}{r}$. 
The self-intersection of the generator of $N$ (which is not mentioned in the
table) is simply $\frac{-K_Y^3}{r^2}$.

Will refer to these examples as \ref{ex:spec1}$^r_d$.
\end{ex}

\newcommand{\ordblocks}{
\begin{table}
\renewcommand{\arraystretch}{1.25}
\setlength{\arraycolsep}{3pt}
\[
\begin{array}[t]{llcccccc} \toprule
Ex            & r & -K_Y^3 & b_3(Y) & b_3(Z) & N & \bar c_2(Z) \\ \midrule
\ref{ex:mm2}_3      & 1 &  8 & 22 &  32 & \sm{4 & 2 \\ 2 & 0} & \rvec{20}{12} \\
\ref{ex:mm2}_{10}     & 1 & 16 &  6 &  24 & \sm{8 & 4 \\ 4 & 0} & \rvec{28}{12} \\
\ref{ex:mm2}_{17} & 1 & 24  & 2 & 28 & \sm{4 & 7 \\ 7 & 6} & \rvec{22}{26} \\
\ref{ex:mm2}_{27} & 1 & 38  & 0 & 40 & \sm{2 & 5 \\ 5 & 4} & \rvec{18}{22} \\
\ref{ex:mm2}_{32}     & 2   & 2^3 \cdot 6 & 0 & 50 & \sm{2 & 4 \\ 4 & 2}
& \rvec{18}{18} \\ 
\ref{ex:mm2}_{35}     & 2   & 2^3 \cdot 7 & 0 & 58 & \sm{4 & 4 \\ 4 & 2} &
\rvec{22}{18} \\[2mm] 
\ref{ex:p1cube}   & 2 & 2^3 \cdot 6 &  0 &  50 & 
\sm{0 & 2 & 2 \\ 2 & 0 & 2 \\ 2 & 2 & 0} & \sm{12 & 12 & 12} \\[2mm]
\ref{ex:deg7} & 1 & 4 & 2 & 12 & \sm{4 & 9 \\ 9 & 8}
& \rvec{22}{32} \\
\ref{ex:deg_sextic_ord} & 2 & 0 & 42 & 44 & \sm{2 & 2 \\ 2 & 0} & \rvec{16}{12}
\\
\ref{ex:rk2blp}_2 & 2 & 2^3 & 20 & 30 & \sm{4 & 4 \\ 4 & 2}
& \rvec{20}{16} \\
\ref{ex:rk2blp}_3 & 2 & 2^3 \cdot 2 & 10 & 28 & \sm{6 & 6 \\ 6 & 4}
& \rvec{24}{20} \\
\ref{ex:rk2blp}_4 & 2 & 2^3 \cdot 3 & 4 & 30 & \sm{8 & 8 \\ 8 & 6}
& \rvec{28}{24} \\
\ref{ex:rk2blp}_5 & 2 & 2^3 \cdot 4 & 0 & 34 & \sm{10 & 10 \\ 10 & 8}
& \rvec{32}{28} \\
\ref{ex:con}_2 & 2 & 2^3 \cdot 2 & 0 & 18 & \sm{4 & 6 \\ 6 & 2} & \rvec{20}{18} \\
\ref{ex:con}_3 & 2 & 2^3 \cdot 3 & 0 & 26 & \sm{6 & 6 \\ 6 & 2} & \rvec{24}{18} \\
\ref{ex:con}_4 & 2 & 2^3 \cdot 4 & 0 & 34 & \sm{8 & 6 \\ 6 & 2} & \rvec{28}{18} \\
\ref{ex:con}_5 & 2 & 2^3 \cdot 5 & 0 & 42 & \sm{10 & 6 \\ 6 & 2} & \rvec{32}{18} \\
\ref{ex:Qf}_1 & 2 & 2^3 & 8 & 18 & \sm{2 & 4 \\ 4 & 0} & \rvec{16}{12} \\
\ref{ex:Qf}_2 & 2 & 2^3 \cdot 2 & 6 & 24 & \sm{4 & 4 \\ 4 & 0} & \rvec{20}{12}  \\
\ref{ex:Qf}_3 & 2 & 2^3 \cdot 3 & 4 & 30 & \sm{6 & 4 \\ 4 & 0} & \rvec{24}{12}  \\
\ref{ex:Qf}_4 & 2 & 2^3 \cdot 4 & 2 & 36 & \sm{8 & 4 \\ 4 & 0} & \rvec{28}{12}  \\
\ref{ex:Qf}_5 & 2 & 2^3 \cdot 5 & 0 & 42 & \sm{10 & 4 \\ 4 & 0} & \rvec{32}{12}  \\
\ref{ex:octic} & 1 & 0 & 6 & 8 & \sm{8 & 8 \\ 8 & 0} & \rvec{28}{24}
\\
\bottomrule
\end{array} 
\]
\smallskip
\caption{Blocks of rank 2 and 3 from Construction \ref{constr:block}}
\label{table:ordblocks}
\vspace{-5mm}
\end{table}}

We now proceed with a selection of building blocks obtained from Fanos and
semi-Fanos of rank 2 or 3. For later use we prioritise ones with index 2.
We collect in Table \ref{table:ordblocks} the key data for these blocks
highlighted in \S\ref{subsec:data}, along with the index $r$, the anticanonical
degree $-K_Y^3$ and the Betti number $b_3(Y)$ of the (semi-)Fano $Y$ used.
(Table \ref{table:ordblocks} also includes two blocks
from \S\ref{subsec:adhoc} that result from applying
Construction \ref{constr:block} to 3-folds that are not semi-Fano.)

\begin{ex}
\label{ex:mm2}
Construction \ref{constr:block} can be applied to all but the first of the 36
entries in the Mori-Mukai list of classes of rank 2 Fano 3-folds.
We will refer to blocks resulting from the $k$th entry as
Example \ref{ex:mm2}$_k$.
The invariants of the resulting blocks can be found in \cite[Table 3]{exotic}.
Let us briefly describe those classes that we will make use of later.
\begin{enumerate}[align=left]
\item[$k = \phantom{0}3$]  Double cover of $\PP^3$ branched over a quartic, blown up in the
pre-image of a line (which is \hspace*{1.5mm} an elliptic curve).
\item[$k = 10$] Complete intersection of two quadrics in $\PP^5$, blown up
in the intersection of two hyperplanes.
\item[$k = 17$] Blow-up of a smooth quadric in $\PP^4$ along an elliptic curve of degree 5.
\item[$k = 27$] Blow-up of $\PP^3$ along a twisted cubic.
\item[$k = 32$]
A (1,1) divisor in $\PP^2 \times \PP^2$.
\item[$k = 35$]
The blow-up of $\PP^3$ in a point.
\end{enumerate}
The last two cases (\ie $k = 32$ and $35$) are the only rank 2 Fanos of index 2.
\end{ex}

\begin{ex}
\label{ex:p1cube}
The only rank 3 Fano 3-fold of index 2
is $Y = \PP^1 \times \PP^1 \times \PP^1$. It has
\[ N \cong \begin{pmatrix} 0 & 2 & 2 \\ 2 & 0 & 2 \\ 2 & 2 & 0 \end{pmatrix}, \]
$b_3(Z) = 50$ and $\bar c_2(Z) = \sm{12 & 12 & 12}$. 
\end{ex}

\ordblocks

\subsection{Semi-Fano 3-folds of rank 2}
\label{subsec:rk2sf}

Smooth weak Fano 3-folds must have Picard rank at least 2, and there is a
classification programme for Picard rank exactly 2, see \eg Jahnke-Peternell
\cite{jahnke08}, Blanc-Lamy \cite{blanc12}, Arap-Cutrone-Marshburn
\cite{arap17}, Cutrone-Marshburn \cite{cutrone13} and Fukuoka \cite{fukuoka17}.
We will not explore this fully, but focus on the cases that will prove most
relevant later.

As seen in Examples \ref{ex:mm2}$_{17}$ and \ref{ex:mm2}$_{27}$,
rank 2 Fano 3-folds are often obtained by blowing up curves of small genus
and degree in $\PP^3$. Blanc and Lamy study cases where the degree is a little
larger relative to the genus, and produce many semi-Fano 3-folds this way.

\enlargethispage{2\baselineskip}

\begin{ex}
\label{ex:deg7}
Let $Y$ be the blow-up of $\PP^3$ in an elliptic curve of degree 7.
Then $Y$ is semi-Fano---indeed, $-K_Y$ is a small contraction according to
Blanc-Lamy \cite[Table 1]{blanc12}. In the basis formed by
the pull-back of the hyperplane class from $\PP^3$ and $-K_Y$ (which also
span the nef cone) the Picard lattice is
\[ N \cong \begin{pmatrix} 4 & 9 \\ 9 & 8 \end{pmatrix} . \]
Compute as above that $b_3(Y) = 2$ and $b_3(Z) = 12$.
Since $Z$ can be viewed as the result of performing two blow-ups, we can
apply Lemma \ref{lem:c2recurse} and \eqref{eq:c2onH} twice
to find $\bar c_2(Z) = \rvec{22}{32}$. 
\end{ex}

We could produce blocks from 21 further cases in \cite[Table 1]{blanc12}
in a similar way, but let us instead restrict attention to the case of rank 2
``semi del Pezzo 3-folds'' (\ie semi-Fanos of index 2), where Jahnke-Peternell
\cite{jahnke08} have provided a complete classification.

Example \ref{ex:mm2}$_{35}$ produced a Fano 3-fold of index 2 by blowing
up $\PP^3$ at a point.
It is true more generally that the canonical bundle being even is preserved
by blowing up a point, but the Fano condition is not. However, for 4 of the 5
families of index 2 Fanos the blow-up has small anticanonical morphism.
(The remaining case is considered in \S\ref{subsec:adhoc}.)

\pagebreak

\begin{ex}
\label{ex:rk2blp}
For $2 \leq d \leq 5$, let $X'$ be a Fano of rank 1, index 2 and degree $d$
as in Example~\ref{ex:rk1index2}$_d$. Blowing up $X'$ at a generic point $p$
yields a semi-Fano $X$ \cite[Theorem 3.7]{jahnke08}.

$H' := \pi^*(-\half K_{X'})$ clearly spans one edge of the nef cone of $X$
(the corresponding morphism is just the blow-down $X \to X'$), and $X$ being
semi-Fano means that $H := -\half K_X = H' - E$ spans the other (where $E$ is
the class of the exceptional divisor). In the basis $H, E$ the Picard form
of $X$ is simply $\sm{2d & 0 \\ 0 & -2}$, so with respect to the basis $H, H'$
for the nef cone we get
\[ N \cong \begin{pmatrix} 2d & 2d \\ 2d & 2d-2 \end{pmatrix} . \]
We see from \eqref{eq:c2onH} that $c_2(X) + c_1(X)^2$ evaluates to
$24 + 8d-8$ on $-K_X$. On the other hand, since $-K_{X'}$ can be represented by
a divisor that does not contain the blow-up point, \cite[Lemma 5.15]{cym} gives
$(c_2(X) + c_1(X)^2)\pi^*(-K_{X'}) = (c_2(X') + c_1(X')^2)(-K_{X'})
= 24 -K_{X'}^3 = 24 + 8d$.
Hence $\bar c_2(Z) = c_2(X) + c_1(X)^2$ is represented by $\rvec{12+4d}{8+4d}$
with respect to the basis of $N^*$ dual to~$H, \, H'$.
\end{ex}

By Jahnke-Peternell \cite{jahnke08}, the remaining classes of rank 2 weak
del Pezzos with small anticanonical morphism fall into two categories:
conic bundles over $\PP^2$ and quadric bundles over $\PP^1$.

\begin{ex}
\label{ex:con}
For $2 \leq d \leq 5$, according to \cite[Theorem 3.7]{jahnke08}
there are degree $d$ weak del Pezzos with small anticanonical morphism
of the form $Y = \PP(E)$, where $E \to \PP^2$ is a rank 2 holomorphic vector
bundle with $c_1(E) = -1$ and $c_2(E) = 7-d$.

Then $-K_Y = \det E - 2T + 3F = 2(-T+F)$, where $F$ is the pull-back of the
hyperplane class from $\PP^2$ and $T$ is the tautological bundle of $\PP(E)$.
As basis for the Picard lattice, we take $-T+F$ and $F$, which also span the
nef cone. Note that $T^2 = c_1(E)T - c_2(E) = -TF + (d-7)F^2$ and $F^3 = 0$ to
find that the Picard lattice is represented with respect to our chosen basis
by
\[ N = \begin {pmatrix} 2d & 6 \\ 6 & 2 \end{pmatrix} . \]
Patently $b_3(Y) = 0$, so $b_3(Z) = -K_Y^3 + 2 = 8d+2$.

To compute $c_2(Y)$, note that $TY$ is stably isomorphic
to $(-T) \otimes E \oplus F^{\oplus 3}$.
We have $c_2((-T) \otimes E) = c_2(E) - Tc_1(E) + T^2 = 0$, so
\[ c_2(Y) = 3F^2 + 3Fc_1(E) + c_2(E) = -6FT . \]
Hence
\[ c_2(Y) + c_1(Y)^2 = -6FT + 4(-T+F)^2 = -18FT + (4d-24)F^2 . \]
This evaluates to 18 on $F$ and to $4d+12$ on $-T+F$, \ie $\bar c_2(Z)$
is represented with respect to our chosen basis by the row vector
$\rvec{4d+12}{18}$.

We refer to the building blocks arising from these semi del Pezzos as
Example \ref{ex:con}$_d$.
\end{ex}

\begin{ex}
\label{ex:Qf}
For each $1 \leq d \leq 5$, according to \cite[Theorem 3.5]{jahnke08} there are
semi del Pezzo 3-folds $Y$ of degree $d$ that are divisors in the
projectivisation of a rank 4 bundle $E$ of $c_1 = 2-d$ over $\PP^1$.
The class of the divisor $Y$ is $-2T + (4-d)F$, where $F$ is the pull-back of
the hyperplane class of the $\PP^1$ base, and $T$ is the class of the
tautological bundle of $\PP(E)$---so the generic fibres of $Y \to \PP^1$ are
quadric surfaces in $\PP^3$.

The anticanonical class of $Y$ is
$-K_Y = -K_{\PP^1} + \det E - 4T - (-2T+(4-d)F) = -2T$.
$-T$~and $F$ form a basis for the Picard lattice.
Noting that on $\PP(E)$ we have $F^2 = 0$ and $T^4 = T^3c_1(E) = d-2$,
we see that the intersection form is represented in this basis by
\[ N \cong \begin{pmatrix} 2d & 4 \\ 4 & 0 \end{pmatrix} .\]
We compute the Chern classes of $Y$ from the tangent bundle of $\PP(E)$ being
stably isomorphic to $(-T) \otimes E \oplus F \oplus F$, and hence find
$b_3(Z) = 12 + 6d$ and $\bar c_2(Z) = \rvec{12+4d}{12}$.

We refer to the building blocks arising from these semi del Pezzos as
Example \ref{ex:Qf}$_d$.
\end{ex}

\pagebreak

\begin{rmk}
In \cite[Theorem 3.5]{jahnke08}, there are actually two different classes
with $d = 2$, corresponding to $E = \calo(-1, 0, 0, 1)$ or $E$ being trivial
over $\PP^1$ (\ie in the latter case $Y$ is a (2,2)-divisor
on $\PP^1 \times \PP^3$). However, these bundles can be deformed to each other,
and so can the semi del Pezzos, so as far as we are concerned they form a
single family of building blocks, \cf \cite[Example 6.11(i)]{cym}.
\end{rmk}

\begin{rmk}
\label{rmk:cubics}
Any rank 2 semi-Fano  whose anticanonical morphism is a small contraction
can be flopped, \ie the anticanonical model has another small resolution
that is also a rank 2 semi-Fano. In some cases the flop is in the same
class as the original semi-Fano, but in some cases it can belong to a
different family.

Consider for instance Example \ref{ex:rk2blp}$_4$, the blow-up $X$ of
the complete intersection $X'$ of two quadrics in $\PP^5$ at a point $p \in X'$.
The morphism defined by $-\half K_X$ can be interpreted as the projection from
$p$ to a hyperplane; it contracts the 4 lines passing through $p$, and
the image (\ie the anticanonical model) is a cubic hypersurface $X''$
that contains a plane $\Pi$. The pre-image of $\Pi$ in $X$ is the exceptional
divisor of the blow-up $X \to X'$, whose intersection number with the
contracted lines in 1. We therefore find that $X$ is the small resolution of
$X''$ obtained by blowing up a quadric surface in $X''$ that intersects $\Pi$
in the singularities of $X''$.

If we instead resolve $X''$ by blowing up $\Pi$ itself, then we obtain a semi
del Pezzo from the class in Example \ref{ex:Qf}$_3$. Indeed we can see in Table
\ref{table:ordblocks} that Examples \ref{ex:rk2blp}$_4$ and \ref{ex:Qf}$_3$
have equal $b_3(Y)$ and $-K_Y^3$ and isometric polarising lattices. However,
the nef cones and $\bar c_2(Z)$ are not identified by that lattice isometry, so
these blocks will produce different extra-twisted connected sums
(see Examples \ref{ex:pi6deg5} and \ref{ex:pi6deg5flop}).

Similarly Examples \ref{ex:con}$_4$ and \ref{ex:rk2blp}$_5$ are both small
resolutions of a singular intersection of two quadrics in $\PP^5$,
while Examples \ref{ex:con}$_5$ and \ref{ex:Qf}$_5$ are both small
resolutions of a singular del Pezzo 3-fold of degree 5.
\end{rmk}

\subsection{Involution blocks from index 2 Fanos}

We now wish to construct building blocks with involution, essentially by
applying Construction \ref{constr:block} to Kähler 3-folds $Y$ that already
admit an involution. One situation where the involution on the resulting block
has the features required in Definition \ref{def:inv_block} is when $Y$ is a
double cover of a smooth K\"ahler 3-fold $X$, branched over an anticanonical
divisor.
It is expedient for us to set up the construction
starting from $X$. 

\begin{constr}
\label{constr:invblock}
Let $X$ be a simply-connected non-singular complex 3-fold with $-K_X$ even, and
suppose there are smooth divisors $\kd \in |{-}K_X|$ and $H \in |{-}\half K_X|$
with transverse intersection~$C$.

Let $Y$ be the double cover of $X$ branched over $\kd$, and
$Z$ the blow-up of $Y$ in $C$.
Because $C$ is contained in the branch set of $Y$, we can lift
the branch-switching involution $\tau$ on $Y$ to an involution on $Z$.
The proper transform in $Z$ of $\kd$ is an anticanonical divisor.
Note that $H^*(Y)^{-\tau}$ has trivial image in $H^*(\kd)$. In particular,
$H^2(Y)$ and $H^2(X)$ have the same image $N$ in $H^2(\kd) = L$.
\end{constr}

\begin{rmk}
Proposition \ref{prop:ind2block} establishes conditions that ensure that
$(Z,\tau)$ is a building block with involution.
Similarly to Proposition \ref{prop:block_from_fano}, these conditions
are satisfied for most semi-Fanos.

Lemma \ref{lem:ampcover}
can be used to prove that $Y$ is Fano/semi-Fano if and only if $X$ is.
Note however that there are usually Fano deformations of $Y$
that are not double covers. Example \ref{ex:p3cover} is one case where there
are not.
\end{rmk}

\begin{prop}
\label{prop:ind2block}
If $N \subset L$ is primitive and $H^3(X)$ is torsion-free then 
$(Z,\tau)$ is an involution block in the sense of
Definition \ref{def:inv_block}.
The image in $H^{1,1}(\kd)$ of the $\tau$-invariant Kähler
cone of $Z$ contains the image of the Kähler cone of $X$.
\end{prop}

\begin{proof}
That $Z$ is a building block in the sense of Definition \ref{def:block} follows
from \cite[Proposition 4.14]{cym}, and the claim about Kähler cones is also
analogous.
The proper transform of $\kd$ is a fixed component
of $\tau$.
The other fibre $\kd'$ preserved by $\tau$ is the pre-image of $H$, which is
a double cover of $H$ branched over $C$, and thus smooth.
Therefore $Z$ is a building block with involution in the sense of
Definition~\ref{def:inv_block}.
\end{proof}

If $Y$ is semi-Fano then $H^2(Y) \to L$ is injective. We already used in
\cite[Proposition 5.7]{cym} that this implies $K = 0$,
the first of the conditions for the involution block to be pleasant.
Crucially, it implies the second condition \eqref{eq:pleasant} too.
Let $\rho := b_2(X) = \rk N$.

\begin{prop}
\label{prop:s_fano}
If, in addition to the hypotheses of Proposition \ref{prop:ind2block},
$H^2(Y) \to L$ is injective then so is $H^2(V) \to L$
(\ie the building block $Z$ has $K = 0$), and
\begin{enumerate}
\item $b_2(Z) - 1 = b_2(V) = b_2(Y) = \rho$.
\item $b_3(Z) = b_1(C) + b_3(Y) = b_1(C) + 2b_3(X) + 22 - 2\rho$.
\item $b_3^+(Z) = b_1(C) + b_3(X)$.
\item $s = b_3^-(V)$. %
\end{enumerate}
In particular, $Z$ is pleasant.
\end{prop}

\begin{proof}
Since $H^2(Y)$ and $H^2(X)$ have the same image in $L$, assuming
$H^2(Y) \to L$ injective implies that $H^2(X) \cong H^2(Y)$.

Let $W := Y \setminus \kd$ and $U := X \setminus \kd$.
Then
\begin{align*}
\chi(W) = \chi(Y) - 24 &= 2\rho - b_3(Y) - 22, \\
\chi(U) = \chi(X) - 24 &= 2\rho - b_3(X) - 22 .
\end{align*}
Therefore $\chi(W) = 2\chi(U)$ implies
\[ b_3(Y) = 2b_3(X) + 22 - 2\rho . \]
$b_3^+(Z) = b_3(Z^0)$, where $Z^0$ is the singular quotient $Z/\tau$.
Let $E \subset Z^0$ be the image of the exceptional divisor in $Z$,
so that $Z^0 \setminus E \cong X \setminus C$ and $E \cong C \times \PP^1$.
Comparing the long exact sequences of $X$ relative to $C$ and $Z^0$ relative
to $E$ gives an exact sequence
$0 \to H^3(X) \to H^3(Z^0) \to H^3(E) \to H^4(Z^0, E)$.
The kernel of the last map is free of rank equal to $b_3(E) = b_1(C)$,
so $b_3(Z^0) = b_3(X) + b_1(C)$. 
Moreover, this shows $H^3(Z^0)$ to be torsion-free, so $Z$ is pleasant by 
Lemma \ref{lem:pleasant}.
\end{proof}

To compute the Chern class data, it is convenient to use that
$TY \oplus \pi^*(-K_X) \cong \pi^*(TX \oplus (-\half K_X))$ implies
$c_2(Y) = \pi^*(c_2(X))$.
If we have already computed $c_2(X) + c_1(X)^2$ then we can use
\begin{equation}
\label{eq:c2double}
c_2(Y) + c_1(Y)^2 = \pi^*\left(c_2(X) + c_1(X)^2\right) - 3c_1(Y)^2
\end{equation}
to say that the building block $Z$ constructed from $Y$ has
$c_2(Z) = g(\bar c_2(Y)) + h$ for
$\bar c_2(Y) = 2\bar c_2(X) - 3\flat(-K_Y) \in N^*$,
and $h \in H^4(Z)$ the Poincar\'e dual of a $\PP^1$ fibre over
the blow-up curve as before.

\begin{rmk}
\label{rmk:h_inv}
Such a $\PP^1$ fibre is $\tau$-invariant, so Lemma \ref{lem:h_inv_pd} implies
the class $\wh B(h)$ from \eqref{eq:h_inv_fail} can be taken to be zero.
\end{rmk}

We now apply Construction \ref{constr:invblock} to the various index 2
Fano 3-folds and semi-Fano 3-folds that we have already considered in
\S\ref{subsec:ordfano}--\ref{subsec:rk2sf}. We collect the data for the
resulting pleasant involution blocks in Table \ref{table:invblocks};
for convenience the table also includes a few blocks
from \S\ref{sec:k3blocks}. The table displays the key data discussed
in \S\ref{subsec:data}, along with the Euler characteristic of the fixed
curve $C \subset Z$ of the involution (corresponding to $-K_Y^3$ for semi-Fano
type blocks). Note that all the blocks in the table could equally well be
used as ordinary blocks if we choose to forget about the involution
(but then there is some redundancy with Table \ref{table:species1}).

\newcommand{\invblocktable}{
\begin{table}
\[
\renewcommand{\arraystretch}{1.25}
\begin{array}[b]{lccccccc}
\toprule
 \hspace{1ex}\text{Ex} & -\chi(C) & b_3(Z) & b_3^+(Z) & N & \bar c_2(Z) \\ \midrule
\ref{ex:p3cover} & 16 & 38 & 18 & \gen{4} & 20 \\ 
\ref{ex:rk1index2}_1 & 2 & 108 & 46 & \gen{2} & 26 \\    
\ref{ex:rk1index2}_2 & 4 & 66 & 26 & \gen{4} & 28 \\ 
\ref{ex:rk1index2}_3 & 6 & 48 & 18 & \gen{6} & 30 \\    
\ref{ex:rk1index2}_4 & 8 & 38 & 14 & \gen{8} & 32 \\
\ref{ex:rk1index2}_5 & 10 & 32 & 12 & \gen{10} & 34 \\
\ref{ex:Dmm2}_6         & 12 & 32 & 14 & \sm{2 & 4 \\ 4 & 2} & \rvec{18}{18} \\ 
\ref{ex:Dmm2}_8         & 14 & 34 & 16 & \sm{4 & 4 \\ 4 & 2} & \rvec{20}{18} \\[2mm] 
\ref{ex:p1cubecover}   &  12 &  30 & 14 & 
\sm{0 & 2 & 2 \\ 2 & 0 & 2 \\ 2 & 2 & 0} & \sm{12 & 12 & 12} \\[2mm]
\ref{ex:deg_sextic} & 0 & 104 & 44 & \sm{2 & 2 \\ 2 & 0} & \rvec{26}{24} \\
\ref{ex:inv2blp}_2 & 2 & 62 & 24 & \sm{4 & 4 \\ 4 & 2} & \rvec{28}{26} \\
\ref{ex:inv2blp}_3 & 4 & 44 & 16 & \sm{6 & 6 \\ 6 & 4} & \rvec{30}{28} \\
\ref{ex:inv2blp}_4 & 6 & 34 & 12 & \sm{8 & 8 \\ 8 & 6} & \rvec{32}{30} \\
\ref{ex:inv2blp}_5 & 8 & 28 & 10 & \sm{10 & 10 \\ 10 & 8} & \rvec{34}{32} \\
\ref{ex:Dcon}_2 & 4 & 24 & 6 & \sm{4 & 6 \\ 6 & 2} & \rvec{28}{18} \\
\ref{ex:Dcon}_3 & 6 & 26 & 8 & \sm{6 & 6 \\ 6 & 2} & \rvec{30}{18} \\
\ref{ex:Dcon}_4 & 8 & 28 & 10 & \sm{8 & 6 \\ 6 & 2} & \rvec{32}{18} \\
\ref{ex:Dcon}_5 & 10 & 30 & 12 & \sm{10 & 6 \\ 6 & 2} & \rvec{34}{18} \\
\ref{ex:DQf}_1 & 2 & 38 & 12 & \sm{2 & 4 \\ 4 & 0} & \rvec{26}{12} \\
\ref{ex:DQf}_2 & 4 & 36 & 12 & \sm{4 & 4 \\ 4 & 0} & \rvec{28}{12} \\
\ref{ex:DQf}_3 & 6 & 34 & 12 & \sm{6 & 4 \\ 4 & 0} & \rvec{30}{12} \\
\ref{ex:DQf}_4 & 8 & 32 & 12 & \sm{8 & 4 \\ 4 & 0} & \rvec{32}{12} \\
\ref{ex:DQf}_5 & 10 & 30 & 12 & \sm{10 & 4 \\ 4 & 0} & \rvec{34}{12} \\
\ref{ex:p1p1smooth} & 16 & 96 & 32 & \sm{0 & 2 \\ 2 & 0} & \rvec{12}{12} \\
\ref{ex:p1rsmooth}_1 & 18 & 108 & 36 & \gen{2} & 18 \\
\ref{ex:p1rsmooth}_2 & 16 & 96 & 32 & \sm{2 & 2 \\ 2 & 0} & \rvec{18}{12} \\[2mm]
\ref{ex:p1rsmooth}_3 & 14 & 84 & 28 &
\sm{2 & 2 & 2 \\ 2 & 0 & 2 \\ 2 & 2 & 0} & \sm{18 & 12 & 12} \\[1mm]
\bottomrule
\end{array}
\]
\vspace{-2mm}
\caption{Examples of pleasant involution blocks}
\vspace{-5mm}
\label{table:invblocks}
\end{table}}

\begin{ex}
\label{ex:p3cover}
Perhaps the simplest example does not in fact use an index two Fano, but rather
the unique one of index 4.
Take $X = \PP^3$, and let $Y$ be the double cover branched over a smooth
quartic $\kd$. (In this case, all deformations of the Fano $Y$ are in fact
branched double covers of $\PP^3$.)

$\rho = 1$ and $b_3(X) = 0$, and $C$ is a degree 8 curve so has
$b_1(C) = 18$. Hence
\[ b_3(Z) = 38, \quad b_3^+(Z) = 18 . \]
The Picard lattice of $Y$ is $N \cong \gen{4}$. Because $Y$ has index 2,
$(-K_Y)^3 = 16$ and $\bar c_2(Z) = \frac{24+16}{2} = 20 \in N^* \cong \bbz$
by \eqref{eq:c2onH}.
(Some of this simply recovers the data for Example \ref{ex:spec1}$^2_2$
in Table \ref{table:species1}.)

Note that the the other preserved fibre of
$\tau$ on $Z$ is a double cover of a quadric, branched over a bidegree $(4,4)$
curve in $\PP^1 \times \PP^1$, or equivalently a K3 with non-symplectic
involution and Picard lattice $\sm{0 & 2 \\ 2 & 0}$
(\cf Example \ref{ex:p1p1smooth}).
So the other preserved fibre is more special than $\kd$.
\end{ex}

\begin{ex}
\label{ex:rk1index2}
There are 5 families of Fano 3-folds $X$ of rank 1 and index 2, and the
computation of the invariants of a double cover $Y$ branched over an
anticanonical K3 divisor $\kd$ and its blow-up $Z$ in an anticanonical curve
$C \subset \kd$ follow the same pattern. We refer to the resulting building
blocks as Example \ref{ex:rk1index2}$_d$, where $d = 1, \ldots, 5$ is
the degree of $X$. Let us provide some varying amounts of additional detail in
the 5 cases.
\begin{enumerate}
\item
$X$ is a smooth sextic hypersurface in %
$\PP^4(3, 2, 1, 1, 1)$, such that the anti\-canonical section
$\kd := \{X_1 = 0\}$ is smooth (where $X_1$ is the weight 2 coordinate).
The double cover $Y$ of $X$ branched over $\kd$ is a sextic
hypersurface in $\PP^4(3, 1, 1, 1, 1)$; it is a double cover of $\PP^3$
branched over a sextic surface.

Let $C \subset \kd$ be the intersection with a hyperplane (of weight 1, like
$\{X_2 = 0\}$). $C$ is a double cover of $\PP^1$ branched over 6
points, so has $b_1(C) = 4$. Let $Z$ be the blow-up of $Y$ at~$C$.
$\rho = 1$ and $b_3(X) = 42$, so
\[ b_3(Z) = 108, \quad b_3^+(Z) = 46 . \]

The Picard lattice of $Y$ is $N \cong \gen{2}$, and $\bar c_2(Z) = 26$
by \eqref{eq:c2onH}.

The other fixed fibre is a double cover of a hyperplane section of $X$,
which is a degree 1 del Pezzo surface; that fibre is therefore a K3 with
non-symplectic involution and diagonal Picard lattice
$\gen{2} \oplus \gen{-2}^8$.

\item $X$ is a double cover of $\PP^3$, as appeared in
Example \ref{ex:p3cover}. In this case the branched double cover $Y$ of $X$
is isomorphic to a quartic 3-fold in $\PP^4$. Note, however, that a generic
quartic in $\PP^4$ is not a double double cover of $\PP^3$
(those in the form $X_0^4 + X_0^2 Q_2(X_1, \ldots X_4) + Q_4(X_1, \ldots, X_4)$
up to projective equivalence are).

\item
Let $X \subset \PP^4$ be a smooth cubic (which has $b_3(X) = 10$) and
$\kd \subset X$ smooth section by a quadric.
The double cover $Y$ over $X$ branched over $\kd$ can be identified
with the complete intersection of a cubic and a quadric in $\PP^5$. Let
$C$ be a hyperplane section of $\kd$ (a genus 4 curve), and $Z$ the
blow-up of $Y$ in $C$. Then $b_3(Z) = 48$, $b_3^+(Z) = 18$, $N \cong \gen{6}$
and $\bar c_2(Z) = 30$.

\item
Let $X \subset \PP^5$ be a complete intersection of two quadrics, $\kd \subset
X$ smooth section by another quadric. The double cover $Y$ of $X$ branched over
$\kd$ embeds as a complete intersection of 3 quadrics in $\PP^6$.

$b_3(X) = 4$, $b_1(C) = 10$, $b_3(Y) = 28$, so $b_3(Z) = 38$ and %
$b_3^+(Z) = 14$. %
$N \cong \gen{8}$, and $\bar c_2(Z) = 32$.

\item $X$ is a section of the Grassmannian $Gr(2,5) \subset \PP^9$ by a
codimension 3 plane.
\end{enumerate}
\end{ex}

\invblocktable

\begin{ex}
\label{ex:Dmm2}
In the Mori-Mukai list of rank 2 Fano 3-folds, two entries are double covers
of index~2 Fanos.
\begin{enumerate}[align=left]
\item[$k = 6$] A branched double cover of a (1,1) divisor $X \subset \PP^2 \times \PP^2$ (\cf Example \ref{ex:mm2}$_{32}$)
\item[$k = 8$] A branched double cover of the blow-up of $\PP^3$ in a point
(\cf Example \ref{ex:mm2}$_{35}$).
\end{enumerate}
In both cases we can read off the topological data from \cite[Table 3]{exotic}.
\end{ex}

\begin{ex}
\label{ex:p1cubecover}
Let $X = \PP^1 \times \PP^1 \times \PP^1$. Then
\[ N \cong \begin{pmatrix} 0 & 2 & 2 \\ 2 & 0 & 2 \\ 2 & 2 & 0 \end{pmatrix}, \]
$b_3(Y) = 16$, $b_1(C) = 14$, $b_3(Z) = 30$, and $b_3^+(Z) = 14$.
\end{ex}

\begin{ex}
\label{ex:inv2blp}
Let $Z$ be the building block obtained by applying
Construction \ref{constr:invblock} to the blow-up of a degree $d$ del Pezzo
3-fold of rank 1 (\cf \ref{ex:rk2blp}$_d$).
We work out $b_3(Z)$ and $b_3^+(Z)$ from $b_3(X) = b_3(X')$
and $b_1(C) = -K_Y^3 = 2d-2$.
By \eqref{eq:c2double},
$c_2(Y) + c_1(Y)^2$ is represented by
$\rvec{24+8d}{16+8d} -3\rvec{2d}{2d-2} = \rvec{24+2d}{22+2d}$.

We refer to these involution blocks as Example \ref{ex:inv2blp}$_d$.
\end{ex}

\begin{ex}
\label{ex:Dcon}
For $2 \leq d \leq 5$, let $Z$ be the building block resulting from applying
Construction \ref{constr:invblock} to the conic-fibred semi del Pezzo 3-fold of
degree $d$ (\cf Example \ref{ex:con}$_d$). 
\eqref{eq:c2double} yields $\bar c_2(Z) = \rvec{26 + 2d}{18}$.

$b_3(Z) = 20+2d$, $b_3^+(Z) = 2 + 2d$.
\end{ex}

\begin{ex}
\label{ex:DQf}
For $1 \leq d \leq 5$, let $Z$ be the building block resulting from applying
Construction \ref{constr:invblock} to the quadric-fibred semi del Pezzo 3-fold
of degree $d$ (\cf Example \ref{ex:Qf}$_d$). 
\eqref{eq:c2double} yields $\bar c_2(Z) = \rvec{24 + 2d}{12}$.

$b_3(Z) = 40-2d$, while $b_3^+(Z) = 12$.
\end{ex}

\subsection{Ad hoc blocks}
\label{subsec:adhoc}

As we have seen, classes of semi-Fano 3-folds often come in sequences.
Sometimes these will be part of a bigger sequence, where the borderline case
fails to be semi-Fano, yet satisfies the hypotheses of
Construction \ref{constr:block}. However, not being able to apply Propositions
\ref{prop:block_from_fano} or \ref{prop:generic_fano} means it takes a bit
more work to employ such blocks.
We carry this out in two cases that lead
to blocks with useful polarising lattices of rank 2---with unusually small
and unusually large discriminants respectively.

The first case comes from extrapolating the classes Example \ref{ex:rk2blp}
consisting of one-point blow-ups of rank 1 del Pezzo 3-folds
of degree $d = 2, \ldots, 5$. This leads us to consider 
$X'$  a rank 1 del Pezzo 3-fold of degree 1, \ie a smooth sextic hypersurface
in $\PP^4(3, 2, 1, 1, 1)$ (this is the family appearing in
Example \ref{ex:rk1index2}$_1$), and let $X$ be the blow-up of $X'$ at a point
$p$, say $p = (0 {:} 0 {:} 0 {:} 0 {:} 1)$.
Then $X$ fails to be weak Fano---in fact, generically $-K_X$ does not even
have any irreducible sections:
$H^0(-K_X)$ is spanned by $X_2^2$, $X_2X_3$ and $X_3^2$.

We can however restrict attention to the case when
$X' \subset \PP^4(3, 2, 1, 1, 1)$ is tangent to $\{X_1 = 0\}$ at $p$.
Then the section $\kd' := \{X_1 = 0\} \cap X'$ has a double point at $p$;
generically it is an ordinary double point, and the proper transform
$\kd \subset X$ is a smooth section of $-K_X$. Now $|{-}K_X|$ is spanned by
$X_1, X_2^2, X_2X_3$ and $X_3^2$, and defines a morphism onto a quadric
cone in $\PP^3$ (mapping $p$ to the vertex of the cone); it is defined
everywhere because the conditions
$p \in X'$ and tangency with $\{X_1 = 0\}$ at $p$ imply that the defining
polynomial of $X'$ has no $X_4^6$ or $X_0X_4^3$ coefficients, so that
$p$ is the only point on $X'$ with $X_1 = X_2 = X_3 = 0$.
(Geometrically, the morphism resolves the projection of $X'$ onto
$\{X_0 = X_4 = 0\} \cong \PP^2(2,1,1) \subset \PP^4(3, 2, 1, 1, 1)$).

Since $-K_X$ is evidently not big, even this non-generic blow-up fails to be
weak Fano. We can nevertheless apply Construction \ref{constr:block} to
construct a building block from $X$, or Construction \ref{constr:invblock}
to construct an involution block from the double cover $Y$ of $X$ branched over $\kd$.
However, it takes more work since we now have to
check some properties, which are automatic if $Y$ is semi-Fano, by hand:
\begin{itemize}
\item In the description of the example that follows, we show that
$H^2(Y) \to L$ is injective with image $N$ primitive.
Then the hypotheses of Propositions \ref{prop:ind2block} and \ref{prop:s_fano}
hold, so that $Z$ is a pleasant involution block.
\item In Lemma \ref{lem:deg_hyper} we show that any generic $N$-polarised K3
appears as an anticanonical divisors in some member of the family of blocks.
\end{itemize}

\begin{ex}
\label{ex:deg_sextic}
Note that there exist sections of $\calo(-1)$ passing through $p$ that meet
$X'$ transversely, defining smooth $H' \in |{-}\half K_{X'}|$. The proper
transform $H \subset X$ of such a divisor is in $|{-}\half K_X|$.
Let $C \subset \kd$ be the intersection with such a section.
It is a double cover of $\PP^2$ branched over 4 points, so $C$ is an elliptic
curve (and $b_1(C) = 2$).
The nef cone of $X$ is spanned by $H$ and $\pi^*H' = H+E$, where $E$ is the
exceptional~$\PP^2$.

Let $Y$ be the double cover of $X$ branched over $\kd$, and
let $Z$ be the blow-up of
$Y$ at $C$. The pre-image $\tilde H \subset Y$ of $H$ is a smooth anticanonical
divisor. The pencil $|\tilde H:\kd| \subseteq |{-}K_Y|$ has base locus $C$, and
yields an anticanonical fibration of $Z$.

$\rho = 2$ and $b_3(X) = 42$, so
\[ b_3(Z) = 104, \quad b_3^+(Z) = 44 . \]
The Picard lattice of $Y$ is $N \cong \sm{2 & 0 \\ 0 & -2}$ with respect to the basis
$\{\tilde H {+} \tilde E, \tilde E\}$, where $\tilde E$ is the exceptional
$\PP^1 \times \PP^1 \subset Y$. Meanwhile the
Picard group of $\kd$ is generated by the hyperplane section and the exceptional
$\PP^1$. Thus we see directly that $H^2(Y) \to L$ is injective with primitive
image.

The other fixed fibre $\tilde H$ has diagonal Picard lattice
$\gen{2} \oplus \gen{-2}^9$, since it is a branched double cover of $H$, which
is a blow-up of a degree 1 del Pezzo $H'$ at a point. The non-genericity of the
choice of blow-up point $p \in X'$ is reflected in the fact that $H$ is the
result of blowing up $H'$ in the nodes of a sextic with 9 nodes
rather than 9 generic points; $\tilde H$ is a K3 with non-symplectic involution
whose fixed set is single elliptic curve (the proper transform of the
nodal sextic) isomorphic to $C$, as appears in Remark \ref{rmk:nodal_sextic}.

In the basis for $N$ given by the edges $\tilde H + \tilde E, \tilde H$
of the nef cone
\[ N = \begin{pmatrix} 2 & 2 \\ 2 & 0 \end{pmatrix} . \]
Analogously to Example \ref{ex:inv2blp} we find that
$\bar c_2(Z) = \rvec{26}{24}$ with respect to this basis.
\end{ex}

\begin{ex}
\label{ex:deg_sextic_ord}
Without taking double cover, we get an ordinary block with
$b_3(Z) = b_3(X) + (-K_X)^3 +2 = 42 + 0 + 2 = 44$,
and $\bar c_2(Z)  = \rvec{16}{12}$.

(Now the blow-up curve is just a fibre of the morphism to the quadric
cone---which is generically smooth as required.)
\end{ex}

\pagebreak

Most of our building blocks have been obtained by applying Construction
\ref{constr:block} to semi-Fano 3-folds. In turn, many semi-Fano 3-folds $Y$
are obtained by blowing up a curve $C$ on simpler Fano 3-fold $X$.
In a sense, for $Y$ to be Fano or semi-Fano requires $C$ to be contained in
sufficiently many anticanonical divisors of $X$. But even if $C$ lies on just
a pencil of anticanonical divisors, $Y$ may still satisfy the conditions
for applying Construction \ref{constr:block}, like in our second ad hoc
example.

\begin{ex}
\label{ex:octic}
Let $Y$ be the blow-up of a complete intersection of quadrics $X \subset \PP^5$
along an elliptic curve of degree~8; that such $X$ exist can be seen as a
consequence of Lemma \ref{lem:octic_elliptic}.
The polarising lattice is spanned by the pull-back $H$ of the hyperplane class
and the exceptional divisor $E$. The nef cone is spanned by $H$ and
$-K_Y = 2H-E$. With respect to that basis, the polarising lattice is represented by $\sm{8 & 8 \\ 8 & 0}$, and (applying Lemma \ref{lem:c2recurse} and \eqref{eq:c2onH} twice) $\bar c_2(Z)$ by $\rvec{28}{24}$.

$b_3(Y) = 6$, and $b_3(Z) = b_3(Y) -K_Y^3 + 2 = 8$.
\end{ex}

\section{Genericity results}
\label{sec:genericity}

In \S\ref{sec:examples}, we will exhibit examples of extra-twisted connected
sums using blocks constructed in \S\ref{sec:sfblocks}. To match pairs of
blocks in the required way (\ie to find \hk rotations in the sense of
Definition \ref{def:hkr} between the K3 surface factors in their asymptotic
cross-sections), we will apply Theorem \ref{thm:matching}. That relies
on establishing that the families of blocks used have certain genericity
properties in the sense of Definition \ref{def:generic}.
As explained in \S\ref{sec:match}, precisely what genericity property is
needed depends on what action on cohomology one tries to achieve for
the \hk rotation, and in some examples, what is needed is stronger than
what Proposition \ref{prop:generic_fano} provides.
We therefore collect here the genericity results that will prove necessary
for our selected examples.

Given a family of building blocks $\fbb$ with polarising lattice $N$,
the problem is basically to
establish sufficient conditions for an overlattice $\Lambda \subset L$ of $N$
that ensure that any K3 surface $\kd$ with $\Pic \kd \cong \Lambda$ embeds
as an anticanonical divisor in some element of $\fbb$.
If the conclusion holds, then elements of $N \subset \Lambda$ are given some
geometric meaning, \eg if elements of $\fbb$ are described in terms of some
embedding into projective space, then there is an element $H \in N$
corresponding to the hyperplane class.
The general strategy to reconstruct these embeddings into projective space
from knowing that $\Pic \kd \cong \Lambda$.

The first step is to recall that the positive cone of a complex K3 has
a chamber structure, where walls are planes orthogonal to $(-2)$-classes in
$\Pic \kd$, and the chambers are possible nef cones. Thus for a marked $K3$
with $\Pic \kd = \Lambda$ and $H \in \Lambda$ such that
has $H^2 > 0$ and $H$ is orthogonal to all $(-2)$-classes in $\Lambda$,
we can always choose a different marking (composing the original choice with
reflections in $(-2)$-classes) to assume WLOG that $H$ is a nef class for
the marked K3.

Once we have a nef class $H$, we can try to apply results of
Saint-Donat \cite{saintdonat74} to prove that $H$ is very ample, \ie that its
sections define an embedding
$\kd \into \PP(H^0(H)) \cong \PP^{\frac{\bm^2}{2} + 1}$.

\begin{lem}[{See Reid \cite[Chapter 3]{reid97}}] 
\label{lem:embed}
Let $\kd$ be a K3 surface, and $\bm \in \Pic \kd$ a nef class.
\begin{enumerate}[label=\textup{(\alph*)}]
\item If $\bm^2 \geq 4$, $\bm$ is not twice an element of square 2, and
\begin{enumerate}[label=\textup{(\roman*)}]
\item \label{it:elliptic}
there is no $v \in \Pic \kd$ such that $v.H = 2$ and $v^2 = 0$
\end{enumerate}
then $|\bm|$
defines a birational morphism to $\PP^{\frac{\bm^2}{2} + 1}$, which is an
isomorphism away from a set of contracted $(-2)$-curves.
If in addition 
\begin{enumerate}[label=\textup{(\roman*)},resume]
\item
\label{it:nodal}
there is no $v \in \Pic \kd$ such that $v.H = 0$ and $v^2 = -2$
\end{enumerate}
then $H$ is very ample.
\item \label{it:hyper}
If $H^2 = 2$ and \ref{it:nodal} holds then $|H|$ defines a double cover of
$\PP^2$, branched over a sextic curve.
(\ref{it:nodal} implies \ref{it:elliptic} in this case.) 
If we instead of \ref{it:nodal} assume that there is no $v \in \Pic \kd$ such
that $v.H = 1$ and $v^2 = 1$, then $|H|$ is basepoint-free and defines
a generically 2-to-1 map $\kd \to \PP^2$, but may contract some $(-2)$-curves.
\end{enumerate}
\end{lem}

Using such a map to projective space, one can then proceed to try to ``build
an element of $\fbb$ around $\kd$'', but the details depend on $\fbb$.
These problems are studied more systematically by
Wallis \cite[\S7.7]{wallis:thesis}, but here we are content to note a
handful of consequences of Lemma \ref{lem:embed} that suffice for
the examples in \S\ref{sec:examples}.

\subsection{Hyper-elliptic K3s}

\begin{prop}
\label{prop:hyper}
Let $\Lambda \subset L$ be a primitive lattice, with $H \in \Lambda$ such that
$H^2 = 2$. Suppose that there is no $v \in \Lambda$ such that
\begin{enumerate}
\item $v.H = 2$ and $v^2 = 0$, or
\item $v.H = 0$ and $v^2 = -2$,
\end{enumerate}
Then for any K3 with Picard lattice exactly $\Lambda$, we can choose a marking
such that the linear system $|H|$ defines a double cover $\kd \to \PP^2$,
branched over a smooth sextic curve. 

In particular, the families of blocks from Examples \ref{ex:spec1}$^2_1$ and
\ref{ex:rk1index2}$_1$ (essentially the same as \ref{ex:spec1}$^1_2$) are
$(\Lambda, H\bbrp)$-generic.
\end{prop}

\begin{proof}
That $\kd$ is branched over a smooth sextic is just a
restatement of Lemma \ref{lem:embed}\ref{it:hyper}.

Now let $F$ be the polynomial defining the sextic curve.
Then for a generic homogeneous quadric~$Q$ and quartic $C$ in three variables,
the sextic hypersurface
\begin{equation}
\label{eq:sextic}
X := \{X_0^2 + X_1 C(X_2, X_3, X_4) + X_1^2 Q(X_2, X_3, X_4) + F(X_2, X_3, X_4 = 0 \} \subset \PP^4(3, 2, 1, 1, 1)
\end{equation}
is a smooth degree 1 del Pezzo 3-fold, with $\{X_1 = 0\} \cong \kd$ as
anticanonical divisor. Blowing up a curve on $X$ yields an building block in
the family of Example \ref{ex:spec1}$^2_1$. Taking a double cover $Y$ of $X$
branched over $\kd$ and then blowing up yields an element of the family
Example \ref{ex:rk1index2}$_1$.

Thus a generic $\Lambda$-polarised K3 embeds as an anticanonical divisor in
Examples \ref{ex:spec1}$^2_1$ and \ref{ex:rk1index2}$_1$ as required.
\end{proof}

\begin{lem}[{\cite[Lemma 7.7]{exotic}}]
\label{lem:deg_hyper}
Let $N \subset L$ be a primitive rank 2 lattice, with quadratic form
represented with respect to a basis $G, H$  by $\sm{0 & 2 \\ 2 & 2}$,
let $\Amp \subset N_\R$ be the open cone spanned by $G$ and $H$.
Let $\Lambda \subset L$ be an overlattice of $N$, and suppose that
\begin{enumerate}
\item
there is no $v \in \Lambda$ such that
$v.H = 1$ and $v^2 = 0$, and
\item there is no $v \in \Lambda$ other than $\pm(H-G)$
such that $v.H = 0$ and $v^2 = -2$,
\end{enumerate}
Then for any K3 with Picard lattice exactly $\Lambda$, we can choose a marking
such that the linear system $|H|$ defines a morphism $\kd \to \PP^2$,
contracting a $(-2)$-curve $E \subset \kd$ to a point $p \in \PP^2$,
which is 2-to-1 except over a sextic curve $C \subset \PP^2$ that is smooth
apart from an ordinary double point at $p$.

In particular, the families of building blocks from Examples
\ref{ex:deg_sextic} and~\ref{ex:deg_sextic_ord}
are $(\Lambda, \Amp)$-generic.
\end{lem}

\begin{proof}
The first part is immediate from Lemma \ref{lem:embed}\ref{it:hyper}.

Let $F$ be the sextic polynomial that defines the curve with ordinary double
point at $p$. Then a generic sextic hypersurface of the form \eqref{eq:sextic}
is a smooth degree 1 del Pezzo 3-fold tangent to the hyperplane $\{X_1 = 0\}$
at $p$, so we can proceed to construct building blocks as in Examples
\ref{ex:deg_sextic} and~\ref{ex:deg_sextic_ord}.
\end{proof}

\subsection{Quartic K3s}

The conditions on $\Pic \kd$ for $\kd$ to embed as a quartic in $\PP^4$ are
immediate from Lemma \ref{lem:embed}(a).
We also use the following result from \cite[Lemma 7.7, case \#27]{exotic} in an example.
 
\begin{lem}
\label{lem:twisted}
Let $N \subset L$ be a primitive rank 2 lattice, with quadratic form
represented with respect to a basis $G, H$  by $\sm{2 & 5 \\ 5 & 4}$,
let $\Amp \subset N_\R$ be the open cone spanned by $G$ and $H$.
Let $\Lambda \subset L$ be an overlattice of $N$, and suppose that there is
no $v \in \Lambda$ such that
\begin{enumerate}
\item $v.H = 2$ and $v^2 = 0$; or
\item $v.H = 0$ and $v^2 = -2$; or
\item $v.H = 1$ and $v^2 \geq -2$.
\end{enumerate}
Then for any K3 with Picard lattice exactly $\Lambda$, we can choose a marking
such that the linear system $|H|$ defines an embedding $\kd \to \PP^3$,
whose image is a smooth quartic hypersurface and $2G - H$ is represented
by a twisted cubic curve $C$. 

In particular, the family of building blocks from Example \ref{ex:mm2}$_{27}$
is $(\Lambda, \Amp)$-generic.
\end{lem}

\pagebreak

\subsection{Sextic K3s}

\begin{prop}
\label{prop:cubics}
Let $\Lambda \subset L$ be a primitive lattice, with $H \in \Lambda$ such that
$H^2 = 6$. Suppose that there is no $v \in \Lambda$ such that
\begin{enumerate}
\item $v.H = 2$ and $v^2 = 0$; or
\item $v.H = 0$ and $v^2 = -2$.
\end{enumerate}
Then for any K3 with Picard lattice exactly $\Lambda$, we can choose a marking
such that the linear system $|H|$ defines an $\kd \to \PP^4$,
whose image is the intersection of a quadric (which may be singular) and a smooth cubic.

In particular, the families of blocks from Examples \ref{ex:spec1}$^2_3$ and
\ref{ex:rk1index2}$_3$ (essentially same as \ref{ex:spec1}$^1_6$) are
$(\Lambda, H\bbrp)$-generic.
\end{prop}

\begin{proof}
Lemma \ref{lem:embed} gives that $H$ is very ample. It is well known that the
image is then a complete intersection of a quadric and a cubic, and that the
cubic may be taken to be smooth
(see Saint-Donat \cite[Theorem 6.1]{saintdonat74}).
\end{proof}

\begin{prop}
\label{prop:deg_cubics}
Let $N \subset L$ be a primitive rank 2 lattice, with quadratic form
represented with respect to a basis $H, \Gamma$  by $\sm{6 & 2 \\ 2 & -2}$.
Let $\Lambda \subset L$ be an overlattice of $N$, and suppose that there is
no $v \in \Lambda$ such that
\begin{enumerate}
\item $v.H = 2$ and $v^2 = 0$; or
\item $v.H = 0$ and $v^2 = -2$; or
\item $v.H = 1$ and $v^2 = -2$.
\end{enumerate}
Then for any K3 with Picard lattice exactly $\Lambda$, we can choose a marking
such that the linear system $|H|$ defines an embedding $\kd \to \PP^4$,
whose image is the intersection of a quadric $Q$ and a
cubic $C$, and contains a conic representing the class $\Gamma$.

The cubic $C$ can be chosen so that it contains the plane $\Pi$ of the conic,
and so that it has no singularities other than 4 ordinary double points along
$\Pi$.

Further if $\Amp_\pm \subset N_\R$ is the open cone spanned by $H$
and $H\pm \Gamma$, then Examples \ref{ex:rk2blp}$_4$ and \ref{ex:inv2blp}$_4$
are $(\Lambda, \Amp_+)$-generic, and
Examples \ref{ex:Qf}$_3$ and \ref{ex:DQf}$_3$ are $(\Lambda, \Amp_-)$-generic.
\end{prop}

\begin{proof}
Using (i) and (ii),  Lemma \ref{lem:embed} implies that the class $H$ is very
ample, so $\kd$ embeds as a degree 6 surface in $\PP^4$, %
so has to be a complete intersection of a quadric $Q$ and a cubic $C$
\cite[Theorem 6.1]{saintdonat74}.

Since the $(-2)$-class $\Gamma$ has positive intersection with $H$ it is
effective.  (iii) implies that $\Gamma$ is irreducible, so represented by a
smooth rational curve. The image in $\PP^4$ is a smooth rational curve of
degree 2, so a conic as required.

Recall from Remark \ref{rmk:cubics} that the semi-Fano 3-folds in Examples
\ref{ex:rk2blp}$_4$ and \ref{ex:Qf}$_3$ (whose double covers are used in
Examples \ref{ex:inv2blp}$_4$ and \ref{ex:DQf}$_3$) are small resolutions
of a cubic containing a plane. Let us therefore consider the unique
plane $\Pi \subset \PP^4$ that contains the conic $\Gamma$.

As a variety in $\Pi$, $C$ is defined by the vanishing of $q := Q_{|\Pi}$.
Since $C \cap \Pi$ contains $\Gamma$, we can write $C_{|\Pi} = q\ell$
for a line $\ell$ on $\Pi$. If we take $L$ to be any hyperplane in $\PP^4$
intersecting $\Pi$ in $\ell$, then by replacing $C$ with $C - L Q$ we can
assume without loss of generality that $C$ contains $\Pi$ as well as~$\kd$.

Without loss of generality, $\Pi = \{x_0 = x_1 = 0\}$.
We obtain a 3-dimensional space of cubic polynomials of the
form $(a_0x_0 + a_1x_1) Q + a_2 C$ with base locus exactly $\kd \cup \Pi$.
By Bertini's
theorem a generic element of this linear system is smooth away from the base
locus. On the other hand, it must also be smooth along the smooth
Cartier divisor $\kd$, so any singularities must lie on $\Pi$.
 
If we write $C = x_0 R_0 + x_1 R_1$ for some quadrics $R_0$, $R_1$, then
the singularities of 
$(a_0x_0 + a_1x_1) Q + a_2 C = x_0(a_0 Q + a_2 R_0) + x_1 (a_1 Q + a_2 R_1)$
in $\Pi$ correspond to the intersection points of 
$a_0 q + a_2 r_0$ and $a_1 Q + a_2 r_1$, where $r_i := R_{i|\Pi}$.
The smoothness of $Q \cap C$ implies that $r_0, r_1$ and $q$ have no common
zeros, \ie the linear system that they span is basepoint-free. Therefore for
generic $a_0, a_1, a_2$, the quadrics  
$a_0 q + a_2 r_0$ and $a_1 Q + a_2 r_1$ intersect transversely in 4 points,
and $(a_0x_0 + a_1x_1)Q + a_2C$ is smooth except for ordinary double points
at those 4 points.

Blowing up $C$ in $\Pi$---or equivalently blowing up $\PP^4$ in $\Pi$ and
taking the proper transform of~$C$---gives a semi-Fano del Pezzo $Y_-$ of the
class from Example \ref{ex:Qf}$_3$, with $\kd$ as an anticanonical divisor.
The nef cone of the blow-up of $\PP^4$ is spanned by $H_-$ and $H_- - E_-$,
where $H_-$ is the pull-back of the hyperplane class and $E_-$ is the
exceptional divisor. The restriction to $\kd$ corresponds to $\Amp_-$, so
Example \ref{ex:Qf}$_3$ is $(\Lambda, \Amp_-)$-generic. Since Examples
\ref{ex:Qf}$_3$ and \ref{ex:DQf}$_3$ have the same anticanonical divisors,
Example \ref{ex:DQf}$_3$ is $(\Lambda, \Amp_-)$-generic too. 

Finally, consider the intersection of $C$ with a generic hyperplane that
contains $\Pi$. This intersection will be the union of $\Pi$ and a smooth
quadric surface $S$ that passes through the singularities of $C$.
Blowing up $C$ in $S$ yields another semi del Pezzo $Y_+$, which belongs
to the class from Example \ref{ex:rk2blp}$_4$. If $E_+$ is the exceptional
divisor of the corresponding blow-up of $\PP^4$, then the nef cone is generated
by $H_+$ and $2H_+ - E_+$. The restriction of $E_+$ to $\kd$ is $H - \Gamma$,
so the image of the nef cone of $Y_+$ in $H^2(\kd;\R)$ is spanned by $H$ and
$2H - (H - \Gamma) = H + \Gamma$. Thus Example \ref{ex:rk2blp}$_4$ is
$(\Lambda, \Amp_+)$-generic, as is Example \ref{ex:inv2blp}$_4$.
\end{proof}
\begin{lem}[{\cite[Lemma 7.7, case \#17]{exotic}}]
\label{lem:deg5generic}
Let $N \subset L$ be a primitive rank 2 lattice, with quadratic form
represented with respect to a basis $G, H$  by $\sm{4 & 7 \\ 7 & 6}$,
let $\Amp \subset N_\R$ be the open cone spanned by $G$ and $H$.
Let $\Lambda \subset L$ be an overlattice of $N$, and suppose that there is
no $v \in \Lambda$ such that
\begin{enumerate}
\item $v.H = 2$ or $3$ and $v^2 = 0$, or
\item $v.H = 0$ and $v^2 = -2$, or
\item $v.H = 1$ or 2, and $v^2 \geq -2$
\end{enumerate}
Then for any K3 with Picard lattice exactly $\Lambda$, we can choose a marking
such that the linear system $|H|$ defines an embedding $\kd \to \PP^4$,
whose image is contained in a smooth quadric 3-fold, and $2H - G$ is
represented by an elliptic curve of degree 5.

In particular, the family of building blocks from Example \ref{ex:mm2}$_{17}$
is $(\Lambda, \Amp)$-generic.
\end{lem}

\subsection{Octic K3s}

We quote the following result from
Wallis \cite[Proposition 7.7.38]{wallis:thesis}.

\begin{prop}
\label{prop:octic}
Let $\Lambda \subset L$ be a primitive lattice, with a primitive class
$H \in \Lambda$ such that $H^2 = 8$.
Suppose that there is no $v \in \Lambda$ such that
\begin{enumerate}
\item $v.H = 2$ or $3$ and $v^2 = 0$; or
\item $v.H = 0$ and $v^2 = -2$.
\end{enumerate}
Then for any K3 with Picard lattice exactly $\Lambda$, we can choose a marking
such that the linear system $|H|$ defines an embedding $\kd \to \PP^5$,
whose image is the complete intersection of three smooth quadrics.
Moreover, one can choose two of those quadrics to intersect transversely.

In particular, the families of blocks from Examples \ref{ex:spec1}$^2_4$ and
\ref{ex:rk1index2}$_4$ (essentially the same as \ref{ex:spec1}$^1_8$) are
$(\Lambda, H\bbrp)$-generic.
\end{prop}

\begin{lem}
\label{lem:octic_elliptic}
Let $N \subset L$ be a primitive rank 2 lattice, with quadratic form
represented with respect to a basis $G, H$  by $\sm{0 & 8 \\ 8 & 8}$, and
let $\Amp \subset N_\R$ be the open cone spanned by $G$ and $H$.
Let $\Lambda \subset L$ be an overlattice of $N$, and suppose that there is
no $v \in \Lambda$ such that
\begin{enumerate}
\item $v.H = 2$ and $v^2 = 0$; or
\item $0 < v.H \leq 4$ and $v^2 \geq -2$; or
\item $v.H = 0$ and $v^2 = -2$.
\end{enumerate}
Then for any K3 with Picard lattice exactly $\Lambda$, we can choose a marking
such that the linear system $|H|$ defines an embedding $\kd \to \PP^5$,
whose image is a complete intersection of 3 smooth quadrics, and
$2H - G$ is represented by a smooth elliptic curve of degree 8. 

In particular, the family of building blocks from Example \ref{ex:octic}
is $(\Lambda, \Amp)$-generic.
\end{lem}

\begin{proof}
That $|H|$ defines an embedding $\kd \to \PP^5$ whose image is a
complete intersection follows from Proposition \ref{prop:octic}.
Now $E := 2H - G$ is a class with $E^2 = 0$ and $H.E = 8$.
To show that it is represented by an elliptic curve we use the
following argument from the proof of Lemma \ref{lem:deg5generic}
from \cite[Lemma 7.7, case \#17]{exotic}.

(ii) rules out the existence of irreducible classes
in $\Pic \kd$ with $d \leq 4$, so $E$ is irreducible. In particular $E$ does
not have any $(-2)$-curve components, so $E$ is nef.
Therefore \cite[Theorem 3.8(b)]{reid97} implies that $|E|$ is basepoint-free.
A generic $C \in |E|$ is therefore a smooth elliptic curve of degree~8.

Finally let $X \subset \PP^5$ be the complete intersection of a generic pencil
in the 3-dimensional space of quadrics cutting out the image of $\kd$. Then the
blow-up $Y$ of $X$ in the image of $C$ belongs to the family of 3-folds
from which Example \ref{ex:octic} is constructed.
\end{proof}

\subsection{Divisors in \texorpdfstring{$\PP^2 \times \PP^2$}{P2 x P2}}

Our final genericity result is slightly different in that we are concerned
with embedding a K3 surface as an anticanonical divisor not into a rank 1
Fano or a blow-up of a rank 1 Fano,
but rather into a primitive rank 2 Fano.

\begin{prop}
\label{prop:11div}
Let $N \subset L$ be a primitive rank 2 lattice, with quadratic form
represented with respect to a basis $G,H$  by $\sm{2 & 4 \\ 4 & 2}$.
Let $\Lambda \subset L$ be an overlattice of $N$, and suppose that there is
no $v \in \Lambda$ such that
\begin{enumerate}
\item $v^2 = -2$, and $(v.G)(v.H) \leq 0$; or
\item $v^2 = 0$ and $v.(G+H) = 2$; or
\item $v^2 = 4$ and $v.G = v.H = 2$
\end{enumerate}
Then for any K3 with Picard lattice exactly $\Lambda$, we can choose a marking
such that the linear systems $|G|$ and $|H|$ defines morphisms $\kd \to \PP^2$,
and their product embeds $\kd$ as a smooth $(1,1)$ divisor
in~$\PP^2 \times \PP^2$.

In particular, if we let $\Amp \subset N_\R$ be the open cone with edges
spanned by $G$ and $H$, then the families of blocks from Examples
\ref{ex:mm2}$_{32}$ and \ref{ex:Dmm2}$_6$ are $(\Lambda,\Amp)$-generic.
\end{prop}

\begin{proof}
Because (i) rules out the existence of any $(-2)$-class $v \in \Lambda$ such
that $v.G$ and $v.H$ have opposite sign, $G$ and $H$ belong to the same
chamber of the positive cone in $\Lambda_\R$.
Hence it is possible to choose a marking so that $G$ and $H$ both belong
to the nef chamber.

Using (i) again, Lemma \ref{lem:embed} ensures
that $|G|$ and $|H|$ both define branched double covers $\kd \to \PP^2$.
Since they are not the same double cover, the product
$|G| \times |H| : \kd \to \PP^2 \times \PP^2$ is birational onto its image.

Meanwhile, the class $G+H$ is nef too. As $(G+H)^2 = 12$, $h^0(G+H) = 8$.
Using (ii), and since (i) prevents the existence
of any $v \in \Lambda$ such that $v^2 = -2$ and $v.(G+H) = 0$,
Lemma \ref{lem:embed} implies that $G+H$ is very ample, embedding
$\kd \into \PP^7$.

Consider now the product map $H^0(G) \otimes H^0(H) \to H^0(G+H)$. As the
domain has dimension~9, the kernel has dimension at least 1. If the kernel has
dimension at least 2, then the image of
$|G| \times |H| : \kd \to \PP^2 \times \PP^2$ is a component of the
intersection of two $(1,1)$-divisors, which is impossible by
degree as $|G| \times |H|$ is birational onto its image.
Thus the image of $\kd$ lies on a unique (1,1)-divisor
$Y \subset \PP^2 \times \PP^2$.

Since $H^0(G) \otimes H^0(H)$ maps onto $H^0(G+H)$, the composition
of $|G| \times |H|$ with the Segre embedding $\PP^2 \times \PP^2 \to \PP^8$
equals the composition of the embedding $|G+H| : \kd \to \PP^7$ with
inclusion into $\PP^8$. In particular, $|G| \times |H| : \kd \to Y$
is an embedding.

It remains to show that the (1,1)-divisor $Y$ is smooth, \ie that the
bilinear form $\bbc^3 \times \bbc^3 \to \bbc$ that defines it has rank 3.
If the rank were 1 then $Y$ would be reducible, which is absurd.
So it remains to rule out that the bilinear form has rank 2, \ie  $Y$
being isomorphic to
\[ \{((X_0 : X_1 : X_2), (Y_0 : Y_1 : Y_2)) \in \PP^2 \times \PP^2 :
X_1Y_2 = X_2Y_1 \}. \]
Then $Y$ would have a small resolution given by the blow-up $\wt Y$ of $\PP^3$
at the points $(1:0:0:0)$ and $(0:1:0:0)$, induced by the rational map
\[ \PP^3 \dashedrightarrow Y, \;
(Z_0: Z_1 : Z_2: Z_3) \mapsto ((Z_0 : Z_2 : Z_3) , (Z_1 : Z_2 : Z_3)) . \]
As the image of $\kd$ in $Y$ is smooth, its proper transform in either $\wt Y$
or its flop would be isomorphic to $\kd$. In either case, there would be a class
$v \in \Pic \kd$ (corresponding to $\calo_{\PP^3}(1)$ in the $\wt Y$ case) such
that $v^2 = 4$ and $v.G = v.H = 2$. That contradicts (iii), so $Y$ must be
a smooth divisor as desired.
\end{proof}

\section{Building blocks from K3s with non-symplectic involution}
\label{sec:k3blocks}

Since involution blocks always contain a K3 fibre with non-symplectic
involution by Remark~\ref{rmk:nonsymp_fibre}, it is natural to consider the
construction of Kovalev and Lee \cite{kovalev03} of building blocks starting
from K3s with non-symplectic involution. We find that these do indeed also lead
to building blocks with involution.
Moreover, by modifying their construction we can also find some pleasant
building blocks with involution.

\subsection{K3s with non-symplectic involution}
\label{subsec:k3nonsymp}

Let $\kd$ be a K3 surface with a
non-symplectic involution, i.e. a holomorphic involution $\tau$ which
acts as $-1$ on $H^{2,0}(\kd)$. Such involutions are classified by Nikulin
\cite{nikulin81}
in terms of the fixed part $N$ of $H^2(\kd; \bbz)$ under the action of $\tau$.
We now summarise the relevant part of the theory. %

The discriminant group of $N$ is 2-elementary, \ie $N^*/N$ is of the
form~$\cg{2}^a$.
The discriminant form of $N$ is the symmetric $\Q/\Z$-valued form $b$
on $N^*/N$ induced by the integral form on $N$; because $N^*/N$ is
2-elementary, $b$ takes values in $\half \Z/\Z$. (Because the lattice $N$ is
even, $b$ also has a $\half \Z/2\Z$-valued quadratic
refinement, but that is unimportant to us.)
The primitive lattice $N$, and hence the deformation family of $(\kd,\tau)$,
is characterised by the rank $r$, the discriminant rank $a$, and a further
invariant $\delta \in \{0,1\}$ defined by
\[ \delta := \left\{ \begin{array}{l}0 \textrm{ if } b(\alpha, \alpha) = 0 \textrm{ for all }
\alpha \in N^*/N, \\ 1 \textrm{ otherwise.} \end{array} \right. \]

The quotient $\rs = \kd/\tau$ is a smooth complex surface, which is rational
when the fixed set $C$ of $\tau$ is non-empty (by Castelnuovo's theorem
\cite{castel95}; if $C$ is empty then $\rs$ is an Enriques surface, but this
case is of no further interest to us). $\kd$ is a double cover of $\rs$,
branched over a smooth reduced divisor $C \in |-2K_\rs|$, and $\tau$
corresponds to the branch-switching involution. With a few exceptions, $C$ has
$k+1$ components, where one has genus $g$ and the other $k$ are $\PP^1$s, for
\[ k = \frac{r-a}{2}, \quad g = \frac{22-r-a}{2} . \]

\noindent
The pull-back of the quotient map gives an inclusion
$H^2(\rs) \to H^2(\kd)$. Denote
\[ N' := \im(H^2(\rs) \to H^2(\kd) . \]
Then $N'$ is a subgroup of $N$, but not in general primitive; $N$ is a
finite index sublattice of $N$. Note that since the quotient map has degree~2,
the intersection form on $N'$ is exactly twice the unimodular form on
$H^2(\rs)$. Its discriminant group is therefore $\cg{2}^r$. Since $N$ is an
overlattice with discriminant group $\cg{2}^a$, the index must be equal to
$2^k$.
(This can also be seen from the long exact sequence \eqref{eq:coverseq}.)

The quotient $N/N' \cong \cg{2}^k$ is generated by the Poincar\'e duals of the
$k+1$ components $C_i$ of the fixed set of $\tau$; the sum of these classes is
contained in $N'$ (as it is the image of $-K_\rs \in H^2(\rs)$), but
(when $k > 0$) the individual classes are not.

\begin{lem}
\label{lem:pdc}
Let $P \in N$ be the Poincar\'e dual of the fixed set $C$;
equivalently, $P := \pi^*(-K_\rs) \in N' \subseteq N$. Then
\begin{enumerate}
\item \label{it:mod4} $P.x = x^2 \mmod 4$ for any $x \in N'$
\item \label{it:divpd}
$\alpha(P) = 2b(\alpha, \alpha) \mmod 2$ for all $\alpha \in N^*$,
where $b$ is the discriminant form. In particular
\begin{itemize}
\item $P$ has even product with all elements of $N$.
\item $P$ is an even element of $N$ if and only if $\delta = 0$.
\end{itemize}
\end{enumerate}
\end{lem}

\begin{proof}
\ref{it:mod4} 
By Wu's theorem, $-K_\rs = c_1(\rs) = w_2(\rs) \in H^2(\rs)$ is characteristic
for the intersection form, \ie
\[ -K_\rs.x = x^2 \mmod 2 \]
for any $x \in H^2(\rs)$. Hence for any $\pi^* x \in N'$,
\[ P.\pi^* x = -2K_\rs.x = 2x^2 = (\pi^* x)^2 \mmod 4 . \]

\ref{it:divpd}
Any $\alpha \in (N')^*$, and hence also any
$\alpha \in N^* \subseteq (N')^*$, can be written as $\half \flat(y)$
for some $y = \pi^* x \in N'$, where $\flat : N \to N^*$ is induced by the
intersection form.
Then
\[ \alpha(P) = \half y . P = \half y^2 \mmod 2, \]
while by definition of the discriminant form,
\[ b(\alpha, \alpha) = (\half y)^2 = \quart y^2 \in \Q/\Z . \qedhere \]
\end{proof}

Let us now make some remarks on Picard lattices and ample cones, needed later
in the context of genericity of families of building blocks in the technical
sense of Definition \ref{def:generic}.
For any K3 surface $\kd$ with non-symplectic involution, the fixed set
$N \subset H^2(\kd)$ is contained in $\Pic \kd$.
By the next lemma, the intersection of the ample cone of $\kd$ with $N_\R$ is
simply the image of the ample cone of $\rs := \kd/\tau$.

\begin{lem}
\label{lem:ampcover}
Let $\kd \to \rs$ be a branched double cover.
A class $\kclass \in \Pic \rs$ is ample if and only if its image
$\pi^*\kclass \in \Pic \kd$ is ample.
\end{lem}

\begin{proof}
This is a special case of a well-known property of surjective morphisms of
proper schemes with finite fibres, see
Hartshorne \cite[Exercise III.5.7(d)]{hartshorne77} (though the special case
of double covers can also be proved by elementary arguments).
\end{proof}

In particular, the orthogonal complement of $N$ in $\Pic \kd$ cannot contain
any $(-2)$-classes.
Conversely

\begin{prop}
\label{prop:involution_exists}
For any K3 surface $\kd$ such that $\Pic \kd$ contains a primitive
2-elementary sublattice $N$, and the orthogonal complement of $N$ in $\Pic \kd$
contains no $(-2)$-classes, there exists a non-symplectic involution on $\kd$
with fixed lattice $N$.
\end{prop}

As one deforms $\kd$ and $\rs$, the ample cone of $\rs$ can jump due to the
appearance of exceptional curves, \ie a $(-2)$ class in $\Pic Y$ could be
represented by a curve for some $Y$ in the family but not others.

\begin{ex}
Let $\alpha, \beta$ be linearly independent sections of $\calo(1) \to \PP^1$,
and for $t \in \bbc$ consider the rank 2 bundle
$E_t := \{ (x, y, z) \in \calo(0,0,1) : \alpha x + \beta y + tz = 0 \}$ over
$\PP^1$.
Then $E_0 \cong \calo(1,-1)$ while $E_t$ is trivial for $t \not= 0$.
If we let $Y_t = \PP(E_t)$ then $Y_0$ is the Hirzebruch surface $\mathbb{F}_2$,
while $Y_t \cong \PP^1 \times \PP^1$ for $t \not= 0$.

We can choose a basis $G, H$ for $\Pic Y_t$ so that the intersection form is
represented by $\sm{0 & 1 \\ 1 & 0}$. For $t \not= 0$ the ample cone of $Y_t$ is
spanned by $G$ and $H$, but for $t = 0$ the $(-2)$-class $G-H$ is represented
by a section of the bundle, and the ample cone is smaller, spanned by $G+H$
and $H$.
\end{ex}

Helpfully this change in the ample cone leaves a trace in $\Pic \kd$.
If there is a $(-2)$-curve in $Y$, then that will not meet any smooth
divisor in $|-2K_Y|$, so the pre-image in $\kd$ will be a disjoint union of
two $(-2)$-curves $C, C'$ that are swapped by the branch-switching
involution. In particular, they represent $(-2)$-classes
in $\Pic \kd \setminus N$.
Conversely, because the orthogonal complement
of $N$ in $\Pic \kd$ a priori cannot contain any $(-2)$-classes,
any $(-2)$-classes in $\Pic \kd \setminus N$ must come in pairs
like this (and be half the sum of two classes of square $-4$, one in $N$ and
one in its orthogonal complement in $\Pic \kd$).

\begin{defn}
We call a K3 surface with involution \emph{degenerate} if
$\Pic \kd \setminus N$ contains a $(-2)$-class.
\end{defn}

In the moduli space of K3 surface with involution with a fixed $N$, the
non-degenerate ones form a connected moduli space, with essentially constant
ample cone.

\begin{lem}[{\cf Nikulin-Saito \cite[page 5 ($\mathcal{D}$)]{nikulin05}}]
\label{lem:saito}
Let $N \subset L$ be a primitive 2-elementary lattice. Then there exists an
open cone $\Amp_N \subset N_\R$ such that for any non-degenerate K3 surface
with non-symplectic involution $(\kd,\tau)$ and a marking $H^2(\kd) \to L$
mapping the fixed set of $\tau$ to $N$, the intersection of the image
of the ample cone of $\kd$ with $N_\R$ equals $\Amp_N$.
\end{lem}

If $\rs$ is a del Pezzo surface, then $N \subset L$ is a totally even
primitive sublattice of rank $\leq 9$. Because $\rs$ does not contain any
$(-2)$-curves, $(\kd,\tau)$ must be non-degenerate.
The converse also holds.

\begin{lem}
Let $(\kd, \tau)$ be a K3 surface with non-symplectic involution. Then the
quotient $\kd/\tau$ is a del Pezzo surface if and only if $(\kd,\tau)$ is
non-degenerate and $N$ is totally even of rank $\leq 9$.
\end{lem}

\begin{proof}
The intersection forms of del Pezzo surfaces are precisely the unimodular
lattices of rank $\leq 9$. For a del Pezzo surface $\rs$, a smooth section of
$-2K_\rs$ is connected, so the resulting K3 surface with involution has
$N = N'$ totally even.

Conversely, if $(\kd, \tau)$ is non-degenerate with $N$ totally even of rank
$r \leq 9$, then $P = \pi^*(-K_\rs)$ has a smooth connected section, and
$P^2 = 20-2r \geq 2$, so $P$ is nef.

If we set $\bm = 3P$, then condition %
\ref{it:elliptic} of Lemma \ref{lem:embed} certainly holds.
By Lemma \ref{lem:pdc}\ref{it:mod4}, there can be no $(-2)$-classes in $N$ that
are orthogonal to $P$. The non-degeneracy condition means that there are no
other $(-2)$-classes in $\Pic \kd$, so condition \ref{it:nodal} holds too.
Hence Lemma \ref{lem:embed} shows that $3P$ is very ample.
By Lemma \ref{lem:ampcover}, $-K_\rs$ must therefore be ample.
\end{proof}

\subsection{Kovalev-Lee blocks}

Let $\kd$ be a K3 with non-symplectic involution $\tau$, and let
$\psi : \PP^1 \to \PP^1$ be the holomorphic involution
$\psi : (x:y) \mapsto (y:x)$.
Kovalev and Lee \cite[\S 4]{kovalev-lee08} use the following complex 3-folds
$Z$ as blocks in the twisted connected sum construction.
The quotient $Z_0$ of $\kd \times \PP^1$ by $\tau \times \psi$ has orbifold
singularities along the $2k+2$ components of $C \times \{(1:1), (1:-1)\}$. 

\begin{constr}
\label{constr:kl}
Let $Z$ be the blow-up of $Z^0$ along its singular locus.
\end{constr}

Kovalev and Lee computed the rational cohomology of these 3-folds. By computing
the integral cohomology, we find that $Z$ are indeed building blocks also in
the sense of Definition \ref{def:block}.
Moreover, if we let $\sigma : \PP^1 \to \PP^1$
be the involution $(x:y) \mapsto (x:-y)$, which commutes with $\psi$, then
$\Id_\kd \times \sigma$ induces an involution on~$Z$, making it a building
block with involution in the sense of Definition~\ref{def:inv_block}.

\begin{prop}
Let $\kd$ be a K3 surface with non-symplectic involution $\tau$, and non-empty
fixed set $C$. Then
\begin{align*}
b_2(Z) &= r + 2k+3 = 2r-a+3,\\
b_3(Z) &= 4g = 44-2r-2a, \\
\rk K &= 2k+2 = 2 + r -a.
\end{align*}
Further $H^3(Z)$ is torsion-free, and the image of
$H^2(Z) \to H^2(\kd)$ is the fixed lattice $N$ of $\tau$ (which is primitive).
In particular, $Z$ is a building block in the sense of Definition \ref{def:block}.
\end{prop}

\begin{proof}
The Betti numbers were computed in
\cite[Proposition 4.3, and (4.3)]{kovalev-lee08}.

$Z_0$ can be viewed as the result of gluing two copies of
$U_0 = (\kd \times \Delta)/(\tau, -1)$, along their common boundary which is the
mapping torus $T$ of $\tau$. $\pi_1(T) \cong T$, and by a Mayer-Vietoris
sequence
\[ H^2(T) \cong \ker (1-\tau^*) = N, \quad
H^3(T) \cong \coker (1-\tau^*) \cong N^* \times \cg{2}^{2g} . \]
The restriction map $H^2(T) \to H^2(\kd)$ for the slices $\kd \subset T$ is
the natural inclusion $N \into L$.

$U_0$ deformation retracts to the simply-connected rational surface
$\rs = \kd/\tau$.
The restriction map $H^2(U_0) \to H^2(T)$
corresponds to the inclusion $N' \into N$.

Let $U$ be the blow-up of $U_0$ at its singular locus. Comparing the long exact
sequences of $U$ and $U_0$ relative to neighbourhoods of the exceptional
divisor $E$ and singular set $C$, respectively, shows that the difference
between $H^*(U)$ and $H^*(U_0)$ is the same as the difference between
$H^*(E) \cong H^*(C) \otimes H^*(\PP^1)$ and $H^*(C)$, \ie
\[ H^2(U) \cong N' \times \bbz^{k+1}, \quad
H^3(U) \cong \bbz^{2g} . \]
However, the added factors are \emph{not} simply generated by duals of cycles
in the exceptional set, so it does \emph{not} follow that $H^*(U)$ and
$H^*(U_0)$ have the same image in $H^*(T)$ (though this is the case with real
coefficients). For example, for a component
$C_i$ of $C$, consider the proper transform in $U$ of the image of
$C_i \times \Delta$ in $U_0$, and let $c_i \in H^2(U)$ be the class it
represents. Then the image of $c_i$ in $H^2(T) \cong N \subset H^2(\kd)$
corresponds to the dual of $C_i$ in $H^2(\kd)$, which is precisely one of the
generators for $N/N'$ we described before. So $H^2(U) \to H^2(T)$ is
surjective.
The class in $H^2(U)$ represented by the exceptional set over $C_i$ is $2c_i$
modulo the image of $H^2(U_0)$ in $H^2(U)$.

Now Mayer-Vietoris for $Z$ as a union of two copies of $U$ shows that
\[ H^2(Z) \cong \bbz \times \bbz^{2k+2} \times N,
\quad H^3(Z) \cong \bbz^{4g} . \]
So the cohomology is torsion-free, the image of $H^2(Z) \to H^2(\kd)$
is the primitive sublattice $N$, $\rk K = 2k + 2 = r-a+2$ and
$b^3(Z) = 4g = 44-2r-2a$.
\end{proof}

\enlargethispage{0.4\baselineskip}

\begin{prop}
\label{prop:k3_generic}
Fix a primitive 2-elementary lattice $N \subset L$, and let $\fbb$ be the set
of building blocks obtained by applying Construction \ref{constr:kl} to K3s
with non-symplectic involution with fixed lattice~$N$.
Then there exists an open cone $\Amp \subset N_\R$ such that if 
$\Lambda \subset L$ is primitive sublattice that contains $N$ and
$\Lambda \setminus N$ does not contain any $(-2)$-classes, then
$\fbb$ is $(\Lambda, \Amp)$-generic.
\end{prop}

\begin{proof}
Immediate from Proposition \ref{prop:involution_exists} and Lemma
\ref{lem:saito}.
\end{proof}

\subsection{Smoothing}

Let $\kd$ be a K3 surface with non-symplectic involution $\tau$, and
$Z_0 := \kd \times \PP^1 / \tau \times \psi$ as above.
Instead of desingularising $Z_0$ by blowing up each component of the singular
set, we can attempt to smooth those components that have positive genus while
blowing up the $\PP^1$s. Further, we can carry out the smoothing in such a way that the involution $\Id \times \sigma$ on $Z_0$ persists, yielding a building block with involution.

For simplicity, we consider only the cases when the
the fixed curve $C \subset \kd$ of $\tau$ has no $\PP^1$ components.
Moreover, we ignore the cases where $C$ consists of elliptic curves
($a = 10$ and $r = 8$ or $10$). That leaves precisely the 10 cases where $\rs$
is a del Pezzo surface, one each for $a = r \in \{1, 3, 4, \ldots, 9\}$,
and two with $a = r = 2$.

We can regard $Z_0$ as the double cover of $\rs \times \PP^1$ branched over the
zero set of the reducible section $(x^2 + y^2)s$ of $\calo_{\PP^1}(2) -2K_\rs$,
where $s$ is a section of $-2K_\rs$ cutting out $C$. The normal crossing
singularities of the divisor correspond precisely to the orbifold singularities
of $Z_0$.

Considering instead a double cover of $\rs \times \PP^1$ branched over a
smooth divisor in $|\calo_{\PP^1}(2) -2K_\rs|$ we obtain a smoothing
of $Z_0$, which is moreover a building block in the sense of
Definition \ref{def:block}. It is convenient to consider the following concrete realisation of the double cover.

\begin{constr}
\label{constr:smoothing}
Let $\rs$ be a del Pezzo surface, and $z \in \PP^1$.
Let $f$ be a section of the line bundle $\calo_{\PP^1}(2) -2K_Y$ over
$Y \times \PP^1$, such that both its zero locus $D$ and
$C := D \cap Y \times \{z\}$ are smooth.
Thinking of $f$ as a homogeneous quadratic polynomial on $\bbc^2$, taking
values in sections of $-2K_\rs$, we can define a smooth subvariety $Z$ of
the total space $G$ of the projectivisation of $-K_\rs \oplus \bbc^2 \to \rs$
by
\begin{equation}
\label{w1eq}
Z := \{(\alpha:\beta:\gamma) \in G : %
\alpha^2 = f(\beta,\gamma) \} .
\end{equation}
The projection map $p : G \dashedrightarrow \PP^1$,
$(\alpha:\beta:\gamma) \mapsto (\beta:\gamma)$ is defined away from
the section $\beta = \gamma = 0$, and hence in particular on $Z$.
If $\pi : G \to \rs$ is the bundle projection map, then the restriction
$\pi \times p : Z \to \rs \times \PP^1$ realises $Z$ as the double cover
branched over $D$. 
The fibre
\[ \kd := p^{-1}(z) \]
is a double cover of $\rs$ branched
over $C \in |{-}2K_\rs|$, so is a K3 surface with non-symplectic involution.
\end{constr}

\begin{prop}
\label{prop:smoothing}
$(Z,\kd)$ is a building block.
Moreover, the image of $H^2(Z) \to H^2(\kd)$ is precisely~$N$, the subset
invariant under the action of the branch-switching involution of $\kd \to \rs$.
Further $K = 0$.
\end{prop}

\begin{proof}
The canonical bundle of $G$ is %
$\pi^*(K_\rs - \det (-K_\rs \oplus \bbc^2)) + 3T = 2\pi^*(K_\rs) + 3T$, where
$T$ is the tautological bundle.
$Z$ is defined by a degree 2 homogeneous
polynomial taking values in $-2\pi^*K_\rs$, \ie it is cut out by a section of
$-2T-2\pi^*K_\rs$. Therefore its canonical bundle is
$T_{|Z}$; this equals the pull-back of the tautological bundle of
$\PP^1$ by $p : Z \to \PP^1$, so the
fibres of $p$ are anticanonical divisors. (Each of the fibres is a double
cover of $\rs$ branched over a divisor in the linear system
$\im f \subseteq |-2K_\rs|$, so they are deformations of $\kd$ with
non-symplectic involution.)

The fact that $-2K_\rs$ is very ample on the del Pezzo surface $\rs$ implies
that the sections of $-2T -2\pi^*K_\rs$ define a morphism
$G \to \bbp(H^0(-2T-2\pi^*K_\rs)^*)$, and it is easy to
see that the only set contracted by this morphism is the section
$\{\beta = \gamma = 0\} \subset G$.
In particular the morphism is semi-small, and
the ``relative Lefschetz theorem with large fibres''
of Goresky-MacPherson \cite[Theorem~1.1,\ page~150]{goresky88}
implies $H^3(Z)$
torsion-free, and $H^2(Z) \cong H^2(G) \cong H^2(\rs) \oplus H^2(\PP^1)$.
Since $a = r$ implies that $H^2(\rs) \to H^2(\kd)$ has image $N$, the image of
$H^2(Z) \to H^2(\kd)$ is also precisely~$N$. So $Z$ is a building block, with
$K = 0$.
\end{proof}

Note that since $\pi^* : H^2(\rs \times \PP^1) \to H^2(Z)$ is an isomorphism,
$\pi^* : H^4(\rs \times \PP^1) \to H^4(Z)$ must have image exactly $2H^4(Z)$.
Let $h \in H^4(Z)$ be half the image of the generator of $H^4(\rs)$.

\begin{rmk}
\label{rmk:h_smoothing}
Geometrically, the pull-back of the generator of $H^4(\rs)$ is the
Poincar\'e dual of the pre-image in $Z$ of $\{x\} \times \PP^1$ for any
$x \in \rs$. For generic $x$ that preimage is itself a $\PP^1$ (a~double
cover of $\PP^1$ branched over 2 points). However, for $x$ in the zero locus of the discriminant $\Delta \in -4K_Y$ of $f$ (considered as a quartic with
coefficients in $-2K_Y$), the pre-image is a disjoint union of two lines, and
$h$ is the Poincar\'e dual of either of these two lines.
\end{rmk}

\pagebreak[2]

\begin{lem}
\label{lem:smoothing_top}
$b_3(Z) = 12(10-r)$, and $c_2(Z) = 24h + 3\pi^*K_\rs$.
\end{lem}

\begin{proof}
As a complex bundle, $TG = T_{vert}G \oplus \pi^*TY$ is stably isomorphic to
$T^{-1} \otimes \pi^*(-K_Y \oplus \bbc^2) \oplus TY$.
Using that $\pi^*(-K_\rs)^2 = (20-2r)h \in H^4(\rs)$ and $T^2_{|Z} = 0$
we find
\[ c(Z) = \frac{\pi^*c(Y) (1-T)^2(1-T-\pi^*K_Y)}{1-2T-2\pi^*K_Y} = 
1 - T + (3T\pi^*K_Y + 24h) + (116-14r)Th \in H^*(Z) . \]
This gives the claimed value of $c_2(Z)$, and also shows
$\chi(Z) = -116 + 14r$. This we can determine $b_3(Z)$, since we know the
other Betti numbers:
\[ \chi(Z) = 2 + 2(1+r) - b_3(Z) . \]
Alternatively, we can compute $\chi(Z)$ from
\[ \chi(Z) = 2\chi(\rs \times \PP^1) - \chi(D) . \]
In turn, we can understand $\chi(D)$ by considering the projection $D \to \rs$.
Generically, the linear system $\im f \subseteq |{-}2K_\rs|$ is
base-point free, so that the projection does not contract any curves.
Then the projection is a double cover, whose branch locus $B \subset \rs$ is
cut out by the discriminant of $f$, which is a section of $-4K_\rs$.
By adjunction, $K_B = 3K_{\rs|B}$,
so $\chi(B) = (3K_\rs)(-4K_\rs) = -12(10-r)$, and
\[ \chi(D) = 2 \chi(\rs) - \chi(B) . \] 
Hence
\[ \chi(Z) = 2\chi(\rs) - 12(10-r) , \]
giving the same result as above.
\end{proof}

By considering more special smoothings of $Z_0$ we obtain building blocks with
involution.
Let $\sigma : \PP^1 \to \PP^1, (x:y) \mapsto (x:-y)$ like before,
an involution with fixed points $(1:0)$ and $(0:1)$.
The subset of the space of sections of
$\calo_{\PP^1}(2) - 2K_\rs$ that is invariant under the action of
$\Id_\rs \times \sigma$, consists of elements of the form
$x^2s + y^2s'$, for $s, s'$ sections of $-2K_\rs$. This linear
system is base-point free, so a general element is smooth.

\begin{constr}
\label{constr:inv_smoothing}
Let $\rs$ be a del Pezzo surface,
and let $f$ be a section of $\calo_{\PP^1}(2) -2K_Y$ that is invariant under
$\Id \times \sigma$, such that both its zero locus $D$ and
$C := D \cap (Y {\times} \{(1:0)\})$ are smooth.
Define $G$, $Z$ and $\kd$ as in Construction \ref{constr:smoothing}.
Define an involution $\tau : Z \to Z$ as the restriction of the involution
$(\alpha:\beta:\gamma) \mapsto (\alpha:-\beta:\gamma)$. Then $\tau$ fixes
$\kd := p^{-1}(1:0)$, and acts as a non-symplectic involution on
$\kd' := p^{-1}(0:1)$. (If we instead lifted $\Id \times \sigma$ to $Z$ as
$(\alpha:\beta:\gamma) \mapsto (\alpha:\beta:-\gamma)$, then the lift would
fix $\kd'$ and map $\kd$ to itself by a non-symplectic involution.)
\end{constr}

\begin{prop}
$(Z,\kd)$ is a pleasant involution block.
\end{prop}

\begin{proof}
We already know from Proposition \ref{prop:smoothing} that $K = 0$.
Since $C$ is connected, to apply Lemma~\ref{lem:pleasant} it remains only to
check that $H^3(Z^0)$ is torsion-free for $Z^0 := Z/\tau$.

Now observe that the branched double cover $G \to G$,
$(\alpha: \beta:\gamma) \mapsto (\alpha:\beta^2:\gamma^2)$ induced an embedding
$Z^0 \into G$. If $f = (x^2 - y^2)s + (x^2 + y^2)s'$, then the image of
$Z^0$ in $G$ is
\[ \{(\alpha:\beta:\gamma) \in G :
\alpha^2 = \beta((\beta-\gamma)s + (\beta+\gamma)s') \} . \]
So $Z^0$ is cut out by a section of the line bundle
$-2T - 2\pi^*K_\rs$, which we argued to be semi-ample in the proof of
Proposition \ref{prop:smoothing}.
While $Z^0$ is singular along the curve $\alpha = \beta = s' = 0$, that is no
obstacle to applying Goresky-MacPherson's Lefschetz theorem with large fibres
as in Proposition \ref{prop:smoothing} to deduce
that $H^3(Z^0)$ is torsion-free.
\end{proof}

\begin{rmk}
\label{rmk:h_inv2}
The pre-image $h \in H^4(Z)$ for the generator of $H^4(\kd)$ chosen above
is patently $\tau$-invariant. To understand the action of $\tau$ on its cochain representatives, we can think geometrically in terms of
the conic fibration $Z \to \rs$ like in
Remark \ref{rmk:h_smoothing}.

In this case, the discriminant
$\Delta = ss'$ is reducible. At non-singular points $x \in \rs$ where
$s(x) = 0$, the fibre over $x$ is a union of two lines that intersect $\kd$ in
distinct points, so each of these lines is mapped to itself by $\tau$.
Thus, $h$ can be viewed as the Poincar\'e dual to a $\tau$-invariant
submanifold, %
and Lemma \ref{lem:h_inv_pd} implies that $\wh B(h) \in H^3_{cpt}(V)$
from \eqref{eq:h_inv_fail} can be taken to be 0.

Meanwhile, at a non-singular point where $s'(x) = 0$, the fibre is a union of
two lines that intersect $\kd$ in their common intersection point, and $\tau$
interchanges the lines. At a point where $s(x) = s'(x) = 0$, the fibre is a
single line tangent to $\kd$.
\end{rmk}

Applying \eqref{eq:b3+z}, with $\rho = r$ and $\chi(C) = 2r-20$, we obtain
\begin{equation}
\label{eq:b3+smoothing}
b_3^+(Z) = \half(120-12r -20 +2r + 2r-20) = 40 - 4r .
\end{equation} 

\begin{ex}
\label{ex:p1p1smooth}
Consider the del Pezzo $\rs = \PP^1 \times \PP^1$, and a double cover
$\kd$ branched over a bidegree $(2,2)$ divisor $C$. 
The intersection form on the invariant lattice $N \subset H^2(\kd)$ is twice
that on $H^2(\PP^1 \times \PP^1)$, \ie in the obvious basis given by the
pull-backs of the generators of $H^2$ of the two $\PP^1$ factors,
\[ N \cong \begin{pmatrix} 0 & 2 \\ 2 & 0 \end{pmatrix} . \]
These basis vectors also span the nef cone. In terms of this basis,
$-\pi^*K_\rs = \cvec{2}{2}$, and the image $\bar c_2(Z) \in N^*$ of $-3\pi^*K_\rs$
is $\rvec{12}{12}$.
\end{ex}

\begin{ex}
\label{ex:p1rsmooth}
For $r \in \{1, \ldots, 9\}$, consider the blow-up $\rs$ of $\PP^2$ in $r-1$
points in general position.
\end{ex}

\begin{rmk}
\label{rmk:nodal_sextic}
There are two non-symplectic involutions with $r = a = 10$, one of which
corresponds to $\rs$ being an Enriques surface (which is of no interest to us,
since the involution has no fixed points), and the other to $Y$ being
$\PP^2$ blown up in 9 points that are
the nodes of a nodal sextic curve. In the latter case, $|{-}2K_\rs|$ is a
pencil spanned by the proper transform of the given sextic (which is an
elliptic curve) and the square of the unique cubic passing through them.
A double cover branched over a generic section of $|{-}2K_\rs|$ therefore gives
a K3 with non-symplectic involution whose fixed set is an elliptic curve.
We can construct a complex 3-fold $Z$ as a double cover of $\rs \times \PP^1$
branched over a smooth divisor $D \in |\calo_{\PP^1} -2K_\rs|$ as above.
However, because $-K_\rs$ is not ample, we cannot apply the Lefschetz
hyperplane theorem to prove that $H^3(Z)$ is torsion-free; indeed,
considering $D$ as a branched double cover of $\rs$ shows that the conclusions
of the Lefschetz theorem are in fact false. 
\end{rmk}

Finally, we note that the blocks obtained by smoothing have the same convenient
genericity features as the ones obtained by blow-up.

\begin{prop}
\label{prop:smoothing_generic}
Let $N \subset L$ be a primitive sublattice, isometric to twice the
intersection lattice of a del Pezzo surface $Y$. Let $\Amp \subset N_\R$
be the subcone corresponding to the ample cone of $Y$, and
let $\fbb$ be the set of building blocks obtained by applying
Construction \ref{constr:inv_smoothing} to the deformation family of $Y$.
Then $\fbb$ is $(\Lambda, \Amp)$-generic for any primitive sublattice 
$\Lambda \subset L$ that contains $N$ such that
$\Lambda \setminus N$ does not contain any $(-2)$-classes.
\end{prop}

\begin{proof}
Combine Proposition \ref{prop:involution_exists} and Lemmas \ref{lem:ampcover}
and \ref{lem:saito}.
\end{proof}

\section{The matching problem}
\label{sec:match}

To use the extra-twisted connected sum construction to produce closed \gtmfd s
it not enough to produce some examples of ACyl Calabi-Yau
3-folds $V_\pm$---possibly with involutions---as in \S\ref{subsec:reducible}
and pick a compatible torus isometry $\tormat$ as in \S\ref{subsec:angles},
since we also
need the asymptotic K3s of $V_\pm$ to be related by a
$\thet$-\hk rotation $\hkr$.
It is helpful to rearrange the problem as: fix a pair $\fbb_+, \fbb_-$ of
deformation families of building blocks with automorphism groups $\Gamma_\pm$,
fix $\tormat$, and \emph{then} construct the \emph{pair} $V_+, V_-$ with the
desired $\hkr$ from elements of $\fbb_\pm$. 

A key step is that we note in \S\ref{subsec:necessary} that if one prescribes
the action of $\hkr$ on $H^2$ of the K3s (captured by the ``configuration'' in
Definition \ref{def:config})
then that defines certain overlattices $\Lambda_\pm$ of the polarising lattices
$N_\pm$ of the building blocks, such that the K3s in a solution to the
matching problem will be $\Lambda_\pm$-polarised. In \S\ref{subsec:sufficient}
we turn that around to say roughly that if any generic $\Lambda_\pm$-polarised
K3 appears as the anticanonical divisor in some member of $\fbb_\pm$
(see Definition \ref{def:generic}) then the matching problem can be solved.

The argument is largely the same as that for matching in rectangular twisted
connected sums in \cite[\S 6]{g2m} or \cite[\S 5]{exotic} (more closely
following the latter), the main difference being how the description of the
lattices $\Lambda_\pm$ depends on the gluing angle $\thet$.

\subsection{Matchings and \hk rotations}

Let us consider the consequences of the $\thet$-\hk rotation condition for the
action of $\hkr$ on cohomology. Let $N^\R_\pm \subset H^2(\kd_\pm)$
be the image of $H^2(V_\pm; \R) \to H^2(\kd_\pm;\R)$ (generated by the
polarising lattice $N_\pm$ as defined in \ref{def:block}), and let
$\Pi_\pm \subset H^2(\kd_\pm)$ be \emph{period} of $\kd_\pm$, \ie
the space of classes of type (2,0) + (0,2).
Then $[\omega^I_\pm] \in N^\R_\pm$, and it is moreover the restriction of
a Kähler class from $Z_\pm$. Meanwhile $\Pi_\pm$ is orthogonal to $N^\R_\pm$,
and is spanned by $[\omega^J_\pm]$ and $[\omega^J_\pm]$.
If we let $\pi_\pm : H^2(\kd_\pm;\R) \to N^\R_\pm$ be the orthogonal
projection, and $\pi^\perp_\pm = \Id - \pi_\pm$, then $\hkr: \kd_+ \to \kd_-$
satisfying \eqref{eq:hkr} implies the following condition also holds.

\begin{defn}
\label{def:matching}
Given building blocks $(Z_+, \kd_+)$ and $(Z_-, \kd_-)$ and $\thet \not= 0$,
call a diffeomorphism
$\hkr : \kd_+ \to \kd_-$ a \emph{$\thet$-matching} if there are Kähler classes
on $Z_\pm$ whose restrictions $\kclass_\pm \in H^2(\kd_\pm;\R)$ satisfy
\begin{itemize}
\item $\pi_+ \hkr^* \kclass_- = (\cos \thet) \kclass_+$ and
$\pi_- (\hkr^{-1})^* \kclass_+ = (\cos \thet) \kclass_-$;
\item $\pi^\perp_+ \hkr^* \kclass_- \in \Pi_+$ and
$\pi^\perp_- (\hkr^{-1})^* \kclass_+ \in \Pi_-$ and moreover
\item $\hkr^* \Pi_- \cap \Pi_+$ is non-trivial.
\end{itemize}
\end{defn}

\begin{lem}
Given blocks $(Z_\pm, \kd_\pm)$, a diffeomorphism $\hkr : \kd_+ \to \kd_-$
is a $\thet$-matching if and only if there exist \hk triples $\omega^I_\pm,
\omega^J_\pm, \omega^K_\pm$ on $\kd_\pm$ such that $[\omega^I_\pm]$ is the
restriction of a Kähler class from $Z_\pm$, and $\hkr$ is a $\thet$-\hk
rotation
with respect to the triples.
\end{lem}

\begin{proof}
If $\hkr$ is a $\thet$-\hk rotation then taking $\kclass_\pm = [\omega^I_\pm]$
satisfies the first two conditions in Definition \ref{def:matching},
while $[\omega^K_+] \in \hkr^*\Pi_- \cap \Pi_+$.

For the converse, note that $\pi_+^\perp \hkr^*\kclass_-$ is a non-zero element
of $\Pi_+$, but is not in $\hkr^*\Pi_- \cap \Pi_+$.
Therefore $\kd_+$ has a holomorphic 2-form $\omega^J_+ + i\omega^K_+$ with 
$[\omega^J_+] \in \pi_+^\perp \hkr^*\kclass_-$ and
$[\omega^K_+] \in \hkr^*\Pi_- \cap \Pi_+$.
By the Calabi-Yau theorem, there is a Ricci-flat Kähler metric
$\omega^I_+ \in \kclass_+$.

Choosing $\omega^I_-, \omega^J_-, \omega^K_-$ analogously and normalising
ensures that
$[\hkr^* \omega^I_-] = (\cos \thet)[\omega^I_+] + (\sin \thet)[\omega^J_+]$,
$[(\hkr^{-1})^*\omega^I_+] =
(\cos \thet)[\omega^I_-] + (\sin \thet)[\omega^J_-]$
and $[\hkr^*\omega^K_-] = [\omega^K_+]$.
Uniqueness of Ricci-flat Kähler metrics in their Kähler class implies
\eqref{eq:hkr}, so $\hkr$ is a \hk rotation.
\end{proof}

Note that in combination with Theorem \ref{thm:acyl}, whenever we find a
$\thet$-matching of a pair of building blocks we can also construct a pair of
ACyl Calabi-Yau manifolds with a \hk rotation.
If we have also chosen a torus matching with gluing angle $\thet$, and the
blocks have any necessary involutions, then we have all the ingredients needed
to apply Construction~\ref{constr:xtcs}.

\subsection{Marked K3s and configurations}

To understand the topology of the extra-twisted connected sum $M$ arising
from some $\thet$-matching $\hkr : \kd_+ \to \kd_-$ of a pair of building
blocks $Z_+, Z_-$, we need to know not just some data about $Z_\pm$
(described in \S\ref{subsec:data}), but also something about the action
of $\hkr^* : H^2(\kd_-) \to H^2(\kd_+)$; for a start, $\hkr^*$ clearly plays
a role in the Mayer-Vietoris calculation of the cohomology of $M$
(see \S\ref{subsec:mv}).

At this point is convenient to switch to the language of marked K3 surfaces,
\ie choose isomorphisms $\hdg_\pm : L \to H^2(\kd_\pm)$ where $L$ is a fixed
copy of the unimodular lattice of signature $(3,19)$. Choices of markings
of anticanonical divisors in building blocks in particular identify the
polarising lattices $N_\pm$ with primitive sublattices of $L$.
Now, if we are given a $\thet$-\hk rotation or $\thet$-matching
$\hkr: \kd_+ \to \kd_-$, then we could choose
$\hdg_- := \hkr^* \circ \hdg_+$. Thus we obtain a pair of embeddings of
$N_+$ and $N_-$ into $L$, depending only on the choice of $\hdg_+$.

\begin{defn}
\label{def:config}
A \emph{configuration} of polarising lattices $N_+$, $N_-$ is a pair of
primitive embeddings $N_\pm \into L$. Two configurations are equivalent
if they are related by the action of the isometry group~$O(L)$.
\end{defn}

So in these terms any \hk rotation or matching has an associated configuration
whose equivalence class is well-defined.
As we see in \S\ref{sec:top}, the configuration captures enough information
that we can compute many topological invariants.

On the other hand, for a fixed pair of building blocks there is usually
little chance of finding a matching.
Following the pattern of \cite[\S 6]{g2m} and \cite[\S 5]{exotic}, it is more
fruitful to set up the matching problem as

\smallskip
{\narrower \noindent \em Given $\thet \in \R/2\pi\Z$ and a pair
$\fbb_+$, $\fbb_-$ of sets of building blocks with fixed topological type and
polarising lattices $N_\pm$, which configurations of embeddings
$N_\pm \subset L$ arise from some matching of elements of $\fbb_+$
and~$\fbb_-$?\par}
\smallskip
\noindent
Using the Torelli theorem, we can reduce the problem of finding building blocks with a $\thet$-matching compatible with a given
configuration to finding certain triples of classes in $L_\R$.

\begin{lem}
\label{lem:torelli}
Let $(Z_\pm, \kd_\pm)$ be a pair of blocks, and let
$\hdg_\pm : L \to H^2(\kd_\pm)$ be markings.
Then there exists a $\thet$-matching $\hkr : \kd_+ \to \kd_-$ with
$\hkr^* = \hdg_+ \circ \hdg_-^{-1}$ if and only if
there exists a triple of unit positive classes
$\kclass_0, \kclass_+, \kclass_-$ in $L_\R$ such that
\begin{itemize}
\item $\kclass_0 \perp \kclass_\pm$
\item $\kclass_+.\kclass_- = \cos \thet$, and
\item $\hdg_\pm(\kclass_\pm) \in H^2(\kd_\pm;\R)$
is the restriction of a Kähler class from $Z_\pm$,
\item  $\gen{\kclass_\mp - \cos \vartheta \, \kclass_\pm, \pm \kclass_0}$ is
the period of the marked K3 $(\kd_\pm, \hdg_\pm)$.
\end{itemize}
\end{lem}

\begin{proof}
Let
$\kclass'_\pm := \frac{\kclass_\mp - \cos \thet \, \kclass_\pm}{\sin \thet}$,
which is a unit class perpendicular to $\kclass_\pm$ and $\kclass_0$.
Let $\omega^J_- + i \omega^K_-$ be the holomorphic 2-form on $\kd_-$ in the
cohomology class
$\hdg_-(\kclass'_- -i\kclass_0)$,
and let $\omega^I_-$ be the unique Ricci-flat Kähler metric in $\kclass_-$.
Then $\omega^I_-, \omega^J_-, \omega^K_-$ is a \hk triple.
The closed complex 2-form
$\Omega' := -(\cos \thet)\omega^I_- + (\sin \thet)\omega^J_- - i \omega^K_-$
defines an integrable complex structure on~$\kd_-$.
 Let $\kd'_-$ denote $\kd_-$
equipped with this complex structure for which $\Omega'$ is holomorphic.
Then $\omega' := -(\sin \thet)\omega^I_- - (\cos \thet)\omega^J_-$ is a
Kähler form on $\kd'_-$.

Now $\hdg_+ \circ \hdg_-^{-1} : H^2(\kd'_-;\Z) \to H^2(\kd_+;\Z)$ is an
isometry that maps $[\Omega']$ to $\hdg_+(\kclass'_+ + i\kclass_0)$,
and $[\omega']$ to $\hdg(\kclass_+)$. Thus $\hdg_+ \circ \hdg_-^{-1}$
is a Hodge isometry, and by the Torelli theorem it is realised as $\hkr^*$
for some biholomorphism $\hkr : \kd_+ \to \kd'_-$.
\end{proof}

\begin{rmk}
\label{rmk:nikulin}
Given a configuration of $N_+$ and $N_-$, we obtain lattice
\[ W := N_+ + N_- \subseteq L \]
containing $N_+$ and $N_-$ as primitive sublattices.
In general, it is possible for $W$ to fail to be primitive in $L$
(see \cite[Example No 8]{g2m} for such a twisted connected sum), but
for simplicity we will not look for such configurations in this paper.
By only using examples of small rank and with
$W$ primitively embedded in $L$, the equivalence classes of the
configurations are completely characterised just by the embeddings of $N_\pm$
into $W$. This is a consequence of the following result
of Nikulin \cite[Theorem 1.12.4]{nikulin79}.
\end{rmk}

\begin{thm}
\label{thm:exist_emb}
Let $W$ be an even non-degenerate lattice of signature $(t_+, t_-)$,
and $L$ an even unimodular lattice of indefinite signature $(\ell_+,\ell_-)$.
If $t_+ \leq \ell_+$, $t_- \leq \ell_-$ and $2\rk W \leq \rk L$, then there
exists a primitive embedding $W \into L$, unique up to $O(L)$.
\end{thm}

\subsection{Necessary conditions for matching}
\label{subsec:necessary}

Let us next consider what necessary conditions Lemma \ref{lem:torelli} imposes
on a configuration for it to be realised by a matching of blocks.
Note first of all that one must have
$\kclass_\pm \in N_\pm$, while the period
$\gen{\kclass_\mp - \cos \vartheta \, \kclass_\pm, \pm \kclass_0}$ is orthogonal to
$N_\pm$. Hence $\pi_\pm k_\mp$ is precisely $\cos \vartheta \, k_\pm$, where
$\pi_\pm : L_\bbr \to N_\pm(\bbr)$ is the orthogonal projection.
Observe that $\pi_+\pi_- : N_+(\bbr) \to N_+(\bbr)$ is self-adjoint
(since $\inner{x, \, \pi_+\pi_- y} = \langle \pi_-x, \, \pi_- y \rangle$
is symmetric in $x,y \in N_+$)
so $N_+(\bbr)$ splits as a direct sum of eigenspaces.

\begin{notn}
\label{notn:espace}
For $\psi \in \R$, let $N_\pm^\psi \subset N_\pm(\R)$ denote the
$(\cos \psi)^2$-eigenspace of $\pi_\mp\circ \pi_\mp$.
\end{notn}

Clearly $\pi_+$ maps $N_+(\bbr)^\psi$ to
$N_-(\bbr)^\psi$, and is invertible if $\psi \not= 0$.
Of course, $N_+(\bbr)^0 = N_-(\bbr)^0 = N_+(\bbr) \cap N_-(\bbr)$.
For any $x \in N_+(\bbr)^\psi$ and $y := \pi_+x$
\begin{equation}
\label{eq:psi}
\frac{(x.y)^2}{(x.x)(y.y)} = (\cos \psi)^2 , \quad y.y = (\cos \psi)^2(x.x).
\end{equation}
In particular, it is necessary that $\kclass_\pm \in N_\pm^\thet$.

Here is a qualitative difference between the matching problem for rectangular
twisted connected sums ($\thet = \frac{\pi}{2}$) and extra-twisted connected
sums ($\thet \not= \frac{\pi}{2}$): in the former case we can choose
$\kclass_\pm \in N_\pm^{\frac{\pi}{2}} = N_\pm(\bbr) \cap N_\mp^\perp$
independently of each other, while in the latter case $\kclass_+$ and
$\kclass_-$ \emph{determine each other}.

\begin{rmk*}
If the ambient space $L$ were positive definite then the eigenvalues $\lambda$
of $\pi_+\circ \pi_-$ would obviously be forced to lie in $[0,1]$.
In a space of indefinite signature it could in general happen that
\begin{itemize}
\item $\lambda < 0$, if $x \in N_+$ such that $x^2$ and $(\pi_- x)^2$ have
opposite sign,
\eg if $N_\pm$ in hyperbolic space  with bilinear form $\sm{0 & 1 \\ 1 & 0}$
are spanned by $x_+ = (2,1)$ and $x_- = (1, -2)$ then
$\pi_\mp x_\pm = \pm\frac{3}{4}x_\mp$, and the
unique eigenvalue of $\pi_+\pi_-$ is $-\frac{9}{16}$; or that
\item $\lambda > 1$, if $x \in N_+$ and $\pi_- x \in N_-$
span an indefinite 2-dimensional subspace but $x^2$ and $(\pi_- x)^2$ have
the same sign, \eg if we take
$N_\pm$ in the hyperbolic space
to be the subspaces spanned by $x_+ = (2,1)$ and $x_- = (1,2)$, then
$\pi_\mp x_\pm = \frac{5}{4}x_\mp$, and the unique eigenvalue of $\pi_+\pi_-$
is $\frac{25}{16}$.
\end{itemize}
However, for matchings with the
given configuration to exist, we saw above that there must exist some
positive classes $\kclass_\pm \in N_\pm^\thet$, which are also orthogonal
to the positive class $\kclass_0$. That forces $N_+ + N_-$ to split as an
orthogonal direct sum of its intersection with the orthogonal complement in $L$
to the span of $\kclass_+, \kclass_-$ and $\kclass_0$, which is negative
definite, and the 2-dimensional positive definite span of $\kclass_+$ and
$\kclass_-$. That forces all eigenvalues of $\pi_+\pi_-$ to lie in $[0,1]$, so
that they can be written as $(\cos \psi)^2$. 
\end{rmk*}

The existence of a $\thet$-matching with a given configuration may also
impose constraints on the Picard lattices of the K3s $\kd_\pm$, beyond the
\emph{a priori} condition that $\Pic \kd_\pm$ contains $N_\pm$.
Let $N_\pm^{\not=\thet} \subset N_\pm$ denote the orthogonal complement
of $N_\pm^\thet$. Then $\kclass_0 \perp N_\pm^{\not= \thet}$ because
$N_\pm^{\not=\thet} \subset N_\pm$, while
$\kclass_+, \kclass_- \perp N_\pm^{\not=\thet}$ because
$\kclass_\pm \in N_\pm^\thet$. Therefore
\begin{equation}
\label{eq:Lambda}
\Lambda_\pm :=
\textrm{primitive overlattice of } N_\pm + N_\pm^{\not=\thet} \subset L
\end{equation}
is perpendicular to the period
$\gen{\kclass_\mp - \cos \vartheta \, \kclass_\pm, \pm \kclass_0}$ of
$\kd_\pm$, \ie $\Lambda_\pm \subset \Pic \kd_\pm$.

In summary, given a pair of families of building blocks $\fbb_\pm$, to find
some pair of elements $(Z_\pm, \kd_\pm) \in \fbb_\pm$ with a $\thet$-matching
$\hkr : \kd_+ \to \kd_-$ it is necessary that we can take the marked
$(Z_\pm, \kd_\pm, \hdg_\pm)$ such that
\begin{enumerate}
\item The intersection of $N_\pm(\R)^\thet$ with the image $\mathcal{K}_{Z_\pm}
\subset L_\R$ of the Kähler cone of $Z_\pm$ is non-empty.
Moreover, if $\thet \not= \frac{\pi}{2}$ then the intersection of
$\epsilon\pi_- (N_+ \cap \mathcal{K}_{Z_+})$ and $\mathcal{K}_{Z_-}$ is
non-empty too, where
\[ \epsilon := (\textrm{sign of } \cos \vartheta) \in \{\pm1\}. \]
\item $\kd_\pm$ is $\Lambda_\pm$-polarised.
\end{enumerate}

\subsection{Sufficient conditions for existence of matching}
\label{subsec:sufficient}

On the other hand, for the family $\fbb_\pm$ to be
$(\Lambda_\pm, \Amp_{\fbb_\pm})$-generic for some open cone
$\Amp_{\fbb_\pm} \subset N_\pm(\R)$ (Definition \ref{def:generic}) says roughly
that a generic $\Lambda_\pm$-polarised K3 can be embedded as an anticanonical
divisor in some block $Z_\pm \in \fbb_\pm$, and moreover in such a way that the
Kähler cone of $Z_\pm$ contains $\Amp_{\fbb_\pm}$. This genericity property is
enough to obtain a sufficient condition for the existence of $\thet$-matchings.

\begin{thm}
\label{thm:matching}
Let $\fbb_\pm$ be a pair of families of building blocks with polarising
lattices $N_\pm$, and $\thet \in \R \setminus \frac{\pi}{2}\Z$.
Let $N_\pm \into L$ be a configuration of the polarising lattices, and
define $\Lambda_\pm$ as in \eqref{eq:Lambda}.
Suppose that the family $\fbb_\pm$ is $(\Lambda_\pm, \Amp_{\fbb_\pm})$-generic.
If
\begin{equation}
\label{eq:amps}
\epsilon \pi_-(N_+(\bbr)^\thet \cap \Amp_{\fbb_+}) \cap \Amp_{\fbb_-}
\neq \emptyset .
\end{equation}
then there exist
$(Z_\pm, \kd_\pm) \in \fbb_\pm$ with an angle $\thet$ K3 matching
$\hkr : \kd_+ \to \kd_-$ with the prescribed configuration. 
\end{thm}

\begin{proof}
The proof uses the same basic idea as the $\thet = \frac{\pi}{2}$ case from
\cite[Proposition 6.18]{g2m}, but the way that $\kclass_+$ and $\kclass_-$
determine each other in this case makes it slightly different. 

Let $W_\pm$ be the orthogonal complement of $N_\pm(\bbr)^\thet$ in
$N_+(\bbr)^\thet \oplus N_-(\bbr)^\thet$, and $T$ the orthogonal complement
of $N_+(\bbr) + N_-(\bbr)$ in $L_\bbr$. $W_\pm$ and $T$ all
have signature $(1,\rk-1)$.
Note that $W_\pm \oplus T$ is the orthogonal complement of $\Lambda_\pm$.
Thus a pair of real lines in the positive cones of $W_\pm$ and $T$ span
a positive-definite 2-plane in $\Lambda_\pm^\perp$, so
\[ \PP\bigl(W_\pm^+\bigr)\times \PP \bigl(T^+\bigr)\]
can be regarded as a submanifold of $G_{\Lambda_\pm}$. Analogously to
\cite[Proposition 6.18]{g2m} it is
an analytic, totally real submanifold. Moreover, because $\Lambda_\pm^\perp$ is
exactly $W_\pm \oplus T$,
\[ \dim_\R \PP\bigl(W_\pm^+\bigr)\times \PP \bigl(T^+\bigr)
= \dim_\bbc G_{\Lambda_\pm} . \]
Therefore the intersection of the submanifold
$\PP\bigl(W_\pm^+\bigr)\times \PP \bigl(T^+\bigr) \subset G_{\Lambda_\pm}$
with the subset $U_{\fbb_\pm} \subset G_{\Lambda_\pm}$ from
Definition \ref{def:generic} is an open dense subset
of $\PP\bigl(W_\pm^+\bigr)\times \PP \bigl(T^+\bigr)$.

Now we wish to find
$(\ell_+, \ell_0) \in \PP(N_+(\bbr)^\thet) \times \PP(T^+)$ such that
\begin{enumerate}
\item $\ell_+ \in \Amp_{\fbb_+}$,
\item $\epsilon \pi_- \ell_+ \in \Amp_{\fbb_-}$,
\item $\big(w_+(\ell_+), \ell_0 \big) \in
\left(\PP\bigl(W_+^+\bigr)\times \PP \bigl(T^+\bigr)\right) \cap U_{\fbb_+}$,
\item $\big(w_-(\ell_+), \ell_0 \big) \in
\left(\PP\bigl(W_-^+\bigr)\times \PP \bigl(T^+\bigr)\right) \cap U_{\fbb_-}$,
\end{enumerate}
where $w_\pm : N_+(\bbr)^\thet \to W_\pm$ are the orthogonal projections
(which are both isomorphisms since $\thet \not= \frac{\pi}{2}$).
The first two conditions define open subsets whose intersection is non-empty by
the hypothesis \eqref{eq:amps}.
The intersection with the open dense subsets defined by the last two conditions
is therefore non-empty. Hence there is a pair $(\ell_+, \ell_0)$ satisfying
(i)-(iv).

By the definition of $\fbb_\pm$ being $(\Lambda_\pm, \Amp_{\fbb_\pm})$-generic,
this means there exist $(Z_\pm, \kd_\pm) \in \fbb_\pm$ with periods
$\big(w_+(\ell_\pm), \ell_0 \big)$ such that $\Amp_{\fbb_\pm}$ is contained in
the image of Kähler cone of $Z_\pm$
Taking $\kclass_+$, $\kclass_-$ and $\kclass_0$ to be the unit norm
representatives of $\ell_+$, $\epsilon \pi_-\ell_+$ and $\ell_0$ respectively,
we can therefore apply Lemma \ref{lem:torelli} to obtain the desired
$\thet$-matching $\hkr : \kd_+ \to \kd_-$.
\end{proof}

\subsection{Configuration angles and pure configurations}

The following invariants of a configuration turn out to have several uses.

\begin{defn}\label{def:angles}
Given a configuration $N_+, N_- \subset L$, let $A_\pm : L_\R \to L_\R$
denote the reflection of $L_\R := L \otimes \R$ in $N_\pm$ (with respect to the
intersection form of $L_\R$; this is well-defined since $N_\pm$ is non-degenerate).
Suppose $A_+ \circ A_-$ preserves some decomposition $L_\R = L^+ \oplus L^-$
as a sum of positive and negative-definite subspaces.
Then the \emph{configuration angles} are the arguments
$\alpha^+_1, \alpha^+_2, \alpha^+_3$ and
$\alpha^-_1, \ldots, \alpha^-_{19}$ of the eigenvalues of the restrictions
$A_+ \circ A_- : L^+ \to L^+$ and $A_+ \circ A_- : L^- \to L^-$ respectively.
\end{defn}

Note that if the configuration is to be realised by a $\thet$-\hk rotation,
then $A_+ \circ A_-$ preserves the decomposition of $L_\R$ into the subspaces
that self-dual and anti-self-dual with respect to the \hk metric, so the
configuration angles are defined. Further, the necessary condition (i) from
\S\ref{subsec:necessary} can be expressed in terms of the configuration angles
as requiring that $\alpha^+_1, \alpha^+_2, \alpha^+_3$ are precisely 0 and
$\pm 2\thet$.

In view of Proposition \ref{prop:generic_fano}, the hypothesis that the family
$\sff_\pm$ is $(\Lambda_\pm, \Amp_{\sff_\pm})$-generic (for some cone
$\Amp_{\sff_\pm}$) is easiest to verify for configurations where
$\Lambda_\pm = N_\pm$.
This amounts to requiring that $N_\mp^{\not=\thet}$ is contained in $N_\pm$,
or equivalently that $N_\mp$ is spanned (at least rationally) by
$N_\mp^0 = N_+ \cap N_-$ and~$N_\mp^\thet$.
Noting that for $0 < |\psi| < \frac{\pi}{2}$
\begin{equation}
\label{eq:mult_angle}
\textrm{multiplicity of $2\psi$ as a configuration angle} \;=\;
\dim N_+^\psi \; = \; \dim N_-^\psi ,
\end{equation}
this is in turn equivalent to requiring that the only non-zero configuration
angles are $\pm 2\thet$.
This is in particular the case if $N_\pm^\thet = N_\pm$; we refer to such
configurations as having ``pure angle $\thet$''.

Configurations with pure angle $\frac{\pi}{2}$ are very easy to produce
(as long as $\rk N_+ + \rk N_- \leq 11$): 
simply apply Theorem \ref{thm:exist_emb} to embed
the perpendicular direct sum $N_+ \perp N_-$ primitively in $L$
On the other hand, for $\thet \not= \frac{\pi}{2}$, the existence of a pure
angle $\thet$ configuration of a given pair of lattices $N_+$, $N_-$ is a
non-trivial condition. To be able to define a bilinear form on
$W := N_+ \oplus N_-$ that restricts to the prescribed one on $N_\pm$ and such that $N_\pm^\thet = N_\pm$, it is necessary but not sufficient that the ranks
be equal.

Consider the case when $\rk N_\pm$ both have rank 1, with generator $n_\pm$
(chosen to be positive, \ie $n_\pm \in \Amp_{\fbb_\pm}$). Then there is only
a single cross-term to choose in $W$, and by \eqref{eq:psi} we must set
\begin{equation}
\label{eq:sqrt}
n_+ . n_- = (\cos \thet) \, \sqrt{(n_+.n_+)(n_-.n_-)} .
\end{equation}
Thus, in this case $W$ exists if and only if the RHS is an integer. 

\begin{ex}
\label{ex:rk1s1}
We can make a $\vartheta = \frac{\pi}{4}$ or $\vartheta = \frac{3\pi}{4}$ matching of the
involution block from Example \ref{ex:rk1index2}$_1$ and a regular block from
Example \ref{ex:spec1}$^4_1$ using
\begin{equation}
\label{eq:rk1}
W = \begin{pmatrix} 2 & \epsilon \, 2 \\ \epsilon \, 2 & 4\end{pmatrix} .
\end{equation}
(This leads to a 2-connected \fpif-twisted connected sum with $b_3(M) = 134$
and $p(M)$ divisible by 24, see Table \ref{table:rk1pi4xtcs}).
\end{ex}

\begin{rmk}
If there does exist a pure angle $\thet$ configuration between the polarising
lattices, then for $\thet \not= \frac{\pi}{2}$ it does not need to be unique,
and different pure angle matchings of blocks from the same families can lead to
non-diffeomorphic $\thet$-twisted connected sums;
see Examples \ref{ex:pi6main} and~\ref{ex:pi6main8}.
\end{rmk}

Let us think a moment about the meaning of changing the sign of $\vartheta$ or
replacing it by a complementary angle. For a start, the condition in Definition
\ref{def:matching} for $\hkr$ to be a $\thet$-matching is actually independent
of the sign of $\thet$, which is related to the earlier observation that a
$\pm \thet$-twisted connected sums of phase rotated ACyl Calabi-Yaus are
(orientation-reversing) diffeomorphic. So the sign is unimportant.

There are several natural ways to modify a matching in order to complement
the angle. We could change the signs of the cross-terms in $W$ like in
\eqref{eq:rk1} while keeping everything else the same, or equivalently, we
could change the sign of the marking on $(\kd_+, I_+)$ (keeping $W$ the same,
but multiplying $\Amp_{\fbb_\pm}$ by $-1$).
Alternatively, we could replace the block $Z_+$
by its complex conjugate; if we keep the marking the same, then $\Amp_{Z_+}$
is multiplied by $-1$.
This is precisely the same way of relating extra-twisted connected sums with
complementary angles as in Remark \ref{rmk:comp}. Any of these changes leaves
the cohomology and $p_1$ of the resulting $G_2$-manifolds unchanged, so we will
not be concerned with distinguishing between complementary angles in the
examples.

\section{Topology}
\label{sec:top}

We now turn to the problem of computing topological properties of extra-twisted connected sums. All our computations will be expressed in terms of data of
the building blocks listed in Tables \ref{table:species1},
\ref{table:ordblocks} and \ref{table:invblocks} (see \S\ref{subsec:data}),
along with the choice of torus isometry, and the configuration of the \hk
rotation in the sense of Definition \ref{def:config}.

The invariants we compute are the integral cohomology, torsion linking form,
and a spin characteristic class (more or less equivalent to the first
Pontrjagin class). Computing the cohomology is routine, though the details
of understanding in particular the torsion in $H^4$ are a bit tedious.
The computation of the spin characteristic class is more involved, and takes up
\S\ref{subsec:p}--\ref{subsec:p_end}. The pay-off is that---as explained in
\S\ref{subsec:nu_and_class}---the invariants we compute are sufficient to
apply classification results for 2-connected 7-manifolds to completely
determine the diffeomorphism types of most examples considered in this paper.
 
\subsection{Mayer-Vietoris generalities}
\label{subsec:mv}

It seems inevitable that computing the full integral cohomology of an
extra-twisted connected sum will involve some case by case checking
for different gluing angles $\vartheta$ . However, some parts of the
computation are common to all non-rectangular extra-twisted connected sums.

Let us briefly recap the context.
We are gluing two ACyl \gtmfd s $M_+$ and $M_-$, each of which is
either a product $S^1_{\lnx_\pm} \times V_\pm$ or
a mapping torus
$\mtt{V_\pm} = (S^1_{\lnx_\pm} \times V_\pm)/(a \times \tau_\pm)$ of an
involution $\tau_\pm$ ($a$ denotes the antipodal map on the circle).
The asymptotic cross-section is of the form
$S^1_{\lnx_\pm} \times S^1_{\lnn_\pm} \times \kd$
or
$(S^1_{\lnx_\pm} \times S^1_{\lnn_\pm})/(a \times a)$ accordingly.
To make the construction, we use a torus matching $\tormat$ and \hk rotation
$\hkr$ to identify the asymptotic cross-section from each
side with a single $T^2 \times \kd$.
We now want to apply the Mayer-Vietoris theorem to $M = M_+ \cup M_-$,
with $M_+ \cap M_- \simeq T^2 \times \kd$.

We set up notation for various cohomology classes
on this cross-section $T^2 \times \kd$,
mirroring that used in \S\ref{sec:blocks}.
On the asymptotic cross-section
$S^1_{\lnx_\pm} \times S^1_{\lnn_\pm} \times \kd$
of $S^1_{\lnx_\pm} \times V_\pm$ let
$\drx_\pm \in H^1(S^1_{\lnx_\pm} \times S^1_{\lnn_\pm} \times \kd)$
correspond to the generator of the ``external'' factor $H^1(S^1_{\lnx_\pm})$,
and let
$\drn_\pm \in H^1(S^1_{\lnx_\pm} \times S^1_{\lnn_\pm} \times \kd)$
correspond to the generator of the ``internal'' factor $H^1(S^1_{\lnn_\pm})$.
If $V_\pm$ has an involution, then like in Notation \ref{notn:abuse} we
abuse notation to denote cohomology classes on the asymptotic cross-section 
$(S^1_{\lnx_\pm} \times S^1_{\lnn_\pm})/(a \times a)$ of
$\mtt{V_\pm}$ identically with their pull-backs to
$S^1_{\lnx_\pm} \times S^1_{\lnn_\pm} \times \kd$.
Thus $2\drx_\pm$ and $2\drn_\pm$ denote primitive elements in
$H^1(T^2 \times \kd)$ in this case, but the subgroup they generate has index 2.
In particular, $\drx_-, \drn_- \in H^1(T^2 \times \kd)$ make sense
\emph{only} when $M_-$ is not a mapping torus, like for \fpif-matchings
in \S\ref{subsec:s1} (``square'' ones with $b_+ = 1$ and $b_- = 0$ in terms
of the discussion in~\S\ref{subsec:angles}).

$H^1(M_\pm) \to H^1(T^2 \times \kd)$ is an isomorphism onto the cyclic subgroup
of $H^1(T^2)$ dual to the internal circle factor, \ie the image is generated by
$\drx_\pm$ or $2\drx_\pm$ depending on whether $M_\pm$ comes from an
ordinary block or an involution block. The images never intersect,
so $H^1(M) = 0$.
The sum of the images is primitive precisely for the arrangements when $M$ is
simply connected; otherwise the contribution to $H^2(M)$ is (obviously)
the finite cyclic group $\pi_1(M)$, but we ignore this case from now on.

$H^2(M_\pm) \to H^2(T^2 \times \kd)$ is an isomorphism onto
$N_\pm \subset H^2(\kd)$, regardless of whether $M_\pm$ comes from an ordinary
or an involution block. Thus $H^2(M) = N_+ \cap N_-$, and we get a contribution
$\bbz \oplus L/(N_+ + N_-)$ to $H^3(M)$. Whether this is torsion-free
depends on the choice of push-out $W$ in the matching, and on whether we embed
$W$ primitively in $L$ or not.

Since $H^3(M_\pm)$ are torsion-free, there is
no other contribution to the torsion in $H^3(M)$.
Thus, we get $M$ 2-connected if and only if we use building blocks with $K_\pm = 0$ and a configuration such that $N_+ \cap N_- = 0$ and $N_+ \oplus N_-$ is
primitive in $L$.

To determine $H^3(M)$ we only need to deal with
$H^3(M_\pm) \to H^3(T^2 \times \kd)$ rationally;
the contribution to the torsion in $H^4(M)$ will have to be dealt with case
by case. The image of $H^3(M_\pm; \bbq)$ is the Lagrangian
$\drx_\pm N_\pm \oplus \drn_\pm T_\pm
\subset H^3(T^2 \times \kd; \, \bbq)$.
Since
\begin{equation}
\drx_+  = \cos \vartheta \drx_-  + \sin \vartheta \drn_- , \quad
\drn_+  = \sin \vartheta \drx_-  - \cos \vartheta \drn_- ,
\end{equation}
for $\drx_+ n_+ + \drn_+ t_+$ to equal $\drx_- n_- + \drn_- t_-$
for some $n_\pm \in N_\pm$ and $t_\pm \in T_\pm$ implies that
$\pi_\pm n_\mp = \cos \vartheta n_\pm$, and thus $n_\pm \in N_\pm^\thet$
in Notation \ref{notn:espace}. 
Hence the dimension of the intersection of the images of $H^3(M_\pm, \bbq)$
equals $d_\vartheta = \rk N_+^\thet = \rk N_-^\thet$
(or the multiplicity of $\thet$ as a configuration angle \eqref{eq:mult_angle}).
On the other hand, the kernel in $H^3(M; \bbq)$ is the $\tau$-invariant
subgroup $H^3(Z;\bbq)^\tau$, or
just $H^3(Z;\bbq)$ in the case of an ordinary block. Denoting the dimension of
that by $b_3^+(Z_\pm)$, we obtain
\begin{equation}
\label{eq:b3}
b_3(M) = 23 - \rho_+ - \rho_- + b_2(M) + b_3^+(Z_+) + b_3^+(Z_-)  + d_\vartheta .
\end{equation}

\begin{rmk}
\label{rmk:parity}
$b_3^+(Z_\pm)$ is always even since $H^3(Z_\pm)^\tau \subseteq H^3(Z_\pm)$ is
symplectic. %
Therefore
\[ 1 + b_2(M) + b_3(M) = \rho_+ + \rho_- + d_\vartheta \mmod 2. \]
Further,
$\rho_+ + \rho_- = \rk N_+^\frac{\pi}{2} + \rk N_-^\frac{\pi}{2} \mmod 2$, the
rank of the perpendicular parts. Hence the ``semi-characteristic'' of $M$
equals $d_\vartheta + \rk N_+^\frac{\pi}{2} + \rk N_-^\frac{\pi}{2}$.
\end{rmk}

\begin{rmk}
For $\vartheta = \frac{\pi}{2}$ we should interpret
$d_\vartheta$ to mean $\rk N_+^\frac{\pi}{2} + \rk N_-^\frac{\pi}{2}
= \rk (N_+ \cap T_-) + \rk (N_- \cap T_+)$.
In case of orthogonal matching we get 
$d_\vartheta = \rho_+ + \rho_- - 2b_2(M)$, and \eqref{eq:b3} recovers the
claim from \cite[(8.56)]{kovalev03} that
$b_2(M) + b_3(M) = 23 + b_3(Z_+) + b_3(Z_-)$ in this setting.
(And from Remark \ref{rmk:parity} we get that the semi-characteristic
is even for any rectangular twisted connected sum, equivalent to claim (i) of
the introduction.)
\end{rmk}

\begin{rmk}
\label{rmk:mvsplit}
When the involution blocks are pleasant, then $H^4(M_\pm)$ is
torsion-free, so the image $\delta(H^3(T^2 \times \kd))$ of the Mayer-Vietoris
boundary map is a direct summand of $H^4(M)$, and contains all torsion
in $H^4(M)$.
\end{rmk}

In \S\ref{subsec:s1}--\ref{subsec:h1} we study $\delta(H^3(T^2 \times \kd))$ in
further detail in the cases
$\vartheta = \frac{\pi}{4}$ and $\frac{\pi}{6}$.
We can make a general statement about the torsion linking form $b_M$ (\cf
\cite[Propositon 3.2]{goette20}).

\begin{lem}
\label{lem:torlink}
Let $M^7 = M_+ \cup_X M_-$ be a gluing of manifolds with boundary $X$,
and let $I_\pm \subseteq H^3(X)$ be the image of $H^3(M_\pm)$. Let
$p_1, p_2 \in H^3(X)$ be classes that are torsion modulo $I_+ + I_-$,
so that their images $\delta(p_1), \delta(p_2) \in H^4(M)$ under the Mayer-Vietoris boundary map are torsion classes.
Then we can write $mp_1 = p_1^+ - p_1^-$ for some $m \in \bbz$ and
$p_1^\pm \in I_\pm$, and
\[ b_M(\delta(p_1), \delta(p_2)) = \frac{1}{m}p_1^+p_2
= \frac{1}{m}p_1^-p_2 \in \Q/\Z . \]
\end{lem}

\begin{proof}
To compute the torsion linking form, we first need a pre-image of
$\delta(p_1)$ under the Bockstein map $\beta : H^3(M; \Q/\Z) \to H^4(M; \Z)$.
First let $q^\pm \in H^3(M_\pm; \Q)$ be a pre-image of
$\frac{1}{m}p_1^\pm \in H^3(X; \Q)$.

The Mayer-Vietoris sequences with coefficients $\Z, \Q$ or $\Q/\Z$ form a
commuting periodic grid with the change-of-coefficients sequences.
It is a general feature of such grids that equality of the images in
$H^3(X; \Q)$ of $p_1 \in H^3(X; \Z)$ and
$(q^+, q^-) \in H^3(M_+; \Q) \oplus H^3(M_-; \Q)$ implies that there exists
$q \in H^3(M; \Q/\Z)$ such that
$q_{|M_\pm} = \frac{1}{m}q^\pm \in H^3(M_\pm; \Q/\Z)$ while 
$\beta(q) = -\delta(p_1)$.

More explicitly, pick cochain representatives $\sigma$ of $p_1$ and $\rho^\pm$
of $q^\pm$. We can write $\sigma = \tau^+_{X} - \tau^-_{|X}$ for some
$(\tau^+, \tau^-) \in C^3(M_+; \Z) \oplus C^3(M_-;\Z)$.
Meanwhile $m\sigma - \rho^+_{|X} + \rho^-_{|X}$ is an exact cochain on~$X$,
so we can pick a pre-differential $\nu \in C^2(X; \Z)$, which we 
in turn write as $\nu = \mu^+_{|X} - \mu^-_{|X}$ for some 
$(\mu^+, \mu^-) \in C^2(M_+; \Z) \oplus C^2(M_-; \Q/\Z)$.
Then $\left(\frac{1}{m}(\rho^++ d\mu^+)-\tau^+ ,
\frac{1}{m}(\rho^- +d\nu^-)-\tau^- \right)$ has a pre-image in
$C^3(M; \Q)$. That is closed mod $\Z$, and we can take $q$ to be the class
represented by the mod $\Z$ reduction.

Using $-q$ as a pre-image of $\delta(p_1)$ in the definition of the torsion
linking form now gives
\[ b(\delta(p_1), \delta(p_2)) = (-q \cup \delta(p_2))[M]
= (q_{|X} \cup p_2)[X] = \frac{1}{m}(p_1^+p_2)[X]. \qedhere \]
\end{proof}

\subsection{The spin characteristic class}
\label{subsec:p}

Apart from the integral cohomology, the main invariant of an extra-twisted
connected sum that we are interested in is the spin characteristic
class $p(M) \in H^4(M)$. It is a refinement of the first Pontrjagin class
$p_1(M)$ in the sense that $p_1(M) = 2p(M)$ (so in the absence of 2-torsion in
$H^4(M)$, $p(M)$ is in fact determined by $p_1(M)$), see \eg
\cite[\S2.1]{7class}. Here are essentially the only facts we need about $p(M)$
beyond it being a characteristic class.

\begin{lem}[{See \cite[Lemmas 2.2 and 2.39]{7class}}]
\label{lem:p}
\hfill
\begin{enumerate}
\item $p(M) \in H^4(M)$ is even for any spin manifold of dimension $\leq 7$.
\item $p(M) = -c_2(M)$ for any $SU$-manifold.
\end{enumerate}
\end{lem}

While we should remember that the building blocks $Z_\pm$ are \emph{not} spin
(because $c_1(Z_\pm) = PD(\kd)$ is primitive, and in particular odd),
nevertheless $c_2(Z_\pm) \in H^4(Z_\pm)$ is always even, see
\cite[Lemma~5.10]{cym}.
Our plan is to think of $p(M)$ as the result of patching up
the classes $-c_2(Z_\pm) \in 2H^4(Z_\pm)$, and we make this precise in Theorem
\ref{thm:phalf_glue}. However, even once we have a formula for $p(M)$, one
needs to look carefully at the Mayer-Vietoris sequence to understand 
what it means (\eg what the greatest divisor in $H^4(M)$ is),
which we do in \S\ref{subsec:s1}--\ref{subsec:h1}. 

To apply classification results for 2-connected manifolds
(see \S \ref{subsec:nu_and_class}), all we need to
know about $p(M)$ is the class of $(H^4(M), p(M))$ up to isomorphisms of
abelian groups with a distinguished element.
If $H^4(M)$ is torsion-free this simply amounts to determining the greatest
integer dividing $p(M)$ (while in general one would also need to capture
information such as the greatest integer dividing $p(M)$ modulo torsion).
Since the image of $p(M)$ in $H^4(\kd)$ is divisible by exactly
$\chi(K3) = 24$, we effectively care about the value of $p(M)$ only modulo 24.

This proves practical to evaluate when $c_2(Z_\pm)$ has been computed in the
form \eqref{eq:c2cbar}, as we have done for all the pleasant involution
blocks in \S\ref{sec:blocks} and \S\ref{sec:k3blocks}. Recall also from
Remarks \ref{rmk:h_inv} and \ref{rmk:h_inv2} that the class
$\wh B(h) \in H^3_{cpt}(V)$ from \eqref{eq:h_inv_fail} vanishes in all those
examples.

\begin{thm}
\label{thm:phalf_main}
Write $c_2(Z_\pm) = g_\pm\bar c_2(Z_\pm) + 24h_\pm$ as in
\eqref{eq:c2cbar}, and suppose that $h_\pm$ is $\tau$-invariant
with $\wh B_\pm(h_\pm) = 0$. 
Then
\[ p(M) = \delta(\drn_+ \bar c_2(Z_+) - \drn_- \bar c_2(Z_-)) \mod 24 , \]
where $\delta : H^3(T^2 \times \kd) \to H^4(M)$ is the Mayer-Vietoris
snake map.
\end{thm}

Note that $\bar c_2(Z_\pm) \in N_\pm^* = L/T_\pm$ is
always even, say $\bar c_2(Z_\pm) = 2y_\pm \mod T_\pm$ for some $y_\pm \in L$.
Because the image of $H^3(M_\pm) \to H^3(T^2 \times \kd)$ always contains
$2\drn_\pm T_\pm$ (regardless of whether $M_\pm$ is of the form
$S^1 \times V_\pm$ or $\mtt{V_\pm}$), the value of the Mayer-Vietoris
map $\delta : H^3(T^2 \times \kd) \to H^4(M)$ on $2\drn_\pm y_\pm$ is
independent of the choice of $y_\pm$, and can be interpreted as a well-defined
element $\delta(\drn_\pm \bar c_2(Z_\pm)) \in H^4(M)$. 
(But there is in general no guarantee that these are even elements of $H^4(M)$,
even though their sum must be even, see
Remarks \ref{rmk:premagic} and \ref{rmk:magic}.)

We will prove Theorem \ref{thm:phalf_main} in the next three subsections.
In practice we apply the following special case. 

\begin{cor}
\label{cor:phalf2}
If in addition the building blocks $Z_+$ and $Z_-$ are both pleasant
then the equivalence class of $p(M) \in H^4(M)$ (modulo isomorphisms of the
abelian group $H^4(M)$) is determined by
\begin{equation}
\label{eq:mod24res}
\drn_+ \bar c_2(Z_+) - \drn_- \bar c_2(Z_-) \in H^3(T^2 \times \kd)
\mod I_+ + I_-
\end{equation}
where $I_\pm$ are the images of $H^3(M_\pm)$. In particular, the greatest
integer dividing $p(M)$ is $\gcd(24, n)$, where $n$ is the greatest integer
dividing \eqref{eq:mod24res}.
\end{cor}

\begin{proof}
We noted in Remark \ref{rmk:mvsplit} that if $Z_+$ and $Z_-$ are pleasant then
the image of $\delta$ is a direct summand in $H^4(M)$.
\end{proof}

If $M_\pm = S^1 \times V$ (\ie does not involve dividing by an involution)
then $I_\pm$ is simply $\drx_\pm N_\pm \oplus \drn_\pm T_\pm$, while
if $M_\pm = \mtt{V_\pm}$ and comes from a pleasant involution block then
$I_\pm$ is determined in Lemma \ref{lem:mtres}(ii).
However, even in the auspicious setting of Corollary \ref{cor:phalf2},
we \emph{still} need to work out more details about
$\delta(H^3(T^2 \times \kd)) \cong H^3(T^2 \times \kd)/(I_+ + I_-)$.

That will have to proceed case by case for different choices of gluing angle
$\vartheta$ and torus isometry (see in particular Propositions
\ref{prop:pi4detail} and \ref{prop:pi6detail}), but let us point out an
important qualitative difference between the cases $\vartheta = \frac{\pi}{2}$
and $\vartheta \not= \frac{\pi}{2}$: For
rectangular TCS the images of $\delta(\drn_+\bar c_2(Z_+))$ and
$\delta(\drn_-\bar c_2(Z_-))$ belong to two different direct summands in
$H^4(M)$ (the respective images of the push-forward maps
$H^4_{cpt}(M_\pm) \to H^4(M)$), so that it suffices to
compute the greatest divisors separately and then take their greatest
common divisor.
But for extra-twisted connected sums the images of $H^4_{cpt}(M_\pm) \to
H^4(M)$ can overlap, so there can be cancellation between
$\delta(\drn_+\bar c_2(Z_+))$ and $\delta(\drn_-\bar c_2(Z_-))$, and
we need to know both terms precisely.

\newcommand{\pmt}{\pi}
\newcommand{\hvt}[1]{H^4_{vt}(#1)}
\newcommand{\hvtmt}[1]{H^4_{vt}(\mtt{#1})}
\newcommand{\clps}{S^1 \widehat{\times} Z}
\newcommand{\pclps}{\widetilde \pmt}
\newcommand{\smashed}{R}
\newcommand{\smashmap}{\kappa}

\subsection{Gluing vertical cohomology classes}

Let 
\[ H^4(M_+) \oplus_0 H^4(M_-) = \{ (x_+, x_-) \in H^4(M_+) \oplus H^4(M_-) :
\gamma_+ x_+ = \gamma_- x_- \} , \]
the subspace of classes whose images under pull-back
$\gamma_\pm : H^4(M_\pm) \to H^4(T^2 \times \kd)$ by the inclusion
$T^2 \times \kd \into M_\pm$ agree.
At the most elementary level, the problem we need to deal with in describing
$p(M)$ is that the map
$H^4(M) \to H^4(M_+) \oplus H^4(M_-)$ in the Mayer-Vietoris sequence, whose
image is $H^4(M_+) \oplus_0 H^4(M_-)$,
does not have a canonical right inverse
$H^4(M_+) \oplus_0 H^4(M_-) \to H^4(M)$. Thus it is not possible to determine
$p(M)$ just from its restrictions $p(M_+)$ and $p(M_-)$.
We wish to exploit that we
do not just know $p(M_\pm) \in H^4(M_\pm)$, we also know
$p(\mtt{Z}) \in H^4(\mtt{Z})$ which contain much more information.
To be able to reconstruct $p(M)$ from that, we further need to exploit that
$p(\mtt{Z})$ is in some sense a ``vertical'' class.

Certainly, the restriction of $p(M)$ to a neighbourhood of
$S^1 \times \kd \subset \mtt{Z}$ is a pull-back of $p(\kd) \in H^4(\kd)$.
Now given cocycles on $\mtt{Z_+}$ and $\mtt{Z_-}$ whose restrictions to
neighbourhoods of $S^1 \times \kd$ are pull-backs of the same cocycle on $\kd$,
we could patch their pull-backs to $M_\pm$ to a cocycle on $M$.
The computation in \cite[Proposition 4.20]{g2m}
of $p(M)$ of a rectangular TCS is carried out in terms of a gluing map
\cite[Definition 4.15]{g2m} described in these terms, but it is complicated
and does not adapt well to the XTCS setting. Instead we wish to define
essentially the gluing map in terms of pull-backs of maps between certain
auxiliary spaces.

To this end we first consider a space $\clps$ obtained from
$\mtt{Z}$ by collapsing the external circle factor over $\kd \subset Z$,
and the projection map $\rho : \mtt{Z} \to \clps$.
(A cochain on $\mtt{Z}$ that near $\kd$ is a pull-back of a cochain on $\kd$
is thus roughly the same thing as a pull-back of a cochain from $\clps$.)

Further, given a pair of blocks that are used to form an extra-twisted connected
sum $M$, let $\smashed := \clps_+ \cup_\kd \clps_-$. We can define a collapsing
map
\[ \smashmap : M \to \smashed, \]
as well as obvious inclusion maps
\[ j_\pm : \clps_\pm \into \smashed . \]
By Mayer-Vietoris,
\[ (j_+^*, j_-^*) : H^4(\smashed) \to H^4(\clps_+) \times H^4(\clps_-) \]
is an \emph{isomorphism} onto the subgroup
$H^4(\clps_+) \oplus_0 H^4(\clps_-)$ of pairs with equal image
in~$H^4(\kd)$.  Thus composing the inverse with $\smashmap*$ gives a canonical
way to glue elements of $H^4(\clps_+) \oplus_0 H^4(\clps_-)$.

In a sense, this repackages the problem of gluing classes in $H^4(\mtt{Z_\pm})$
as a problem of finding pre-images of those classes in $H^4(\clps_\pm)$.
The issue now is that while $\rho^* : H^4(\clps_\pm) \to H^4(\mtt{Z})$ is
surjective, it is certainly not injective. We could now ask ourselves for 
which subsets of $H^4(\mtt{Z})$ there is a canonical right inverse to $\rho^*$,
and try to give an answer in terms of certain kinds of ``vertical'' classes
(\eg $\rho_*$ is injective on the kernel of a natural map
$H^4(\clps_\pm) \to H^3_{cpt}(V_\pm)$).

Something that is good enough for our purposes is to define a map
\begin{equation}
\label{eq:clplift}
2H^4(Z)^\tau \to H^4(\clps), \; 2x \mapsto \clplift{2x}
\end{equation}
as follows.
If $x$ is a $\tau$-invariant class, pick a cochain representative $\alpha$
whose restriction to a neighbourhood of $\kd$ is a pull-back of a cochain
on $\kd$. Then the cocycle $\alpha + \tau^* \alpha$ on $S^1 \times Z$
descends to $\mtt{Z}$. Because its restriction to a neighbourhood of the
collapsing set $S^1 \times \kd$ is a pull-back from $\kd$, it is moreover
a pull-back of a cocycle on $\clps$. The resulting class
$\clplift{2x} \in H^4(\clps)$ is clearly independent of the choice of
representative $\alpha$ of $x$.
Because $H^4(Z)$ is assumed to be torsion-free, it then depends only on $2x$
(and not on $x$).

\begin{defn}
\label{def:gluemap}
Define
\begin{equation}
\label{eq:vtgluemap}
Y : 2H^4(Z_+)^\tau \oplus_0 2H^4(Z_-)^\tau \to H^4(M)
\end{equation}
to be the composition of $2x \mapsto \clplift{2x}$,
the inverse of $(j_+^*, j_-^*)$, and $\smashmap^* : H^4(\smashed) \to H^4(M)$.
\end{defn} 

\noindent
In Theorem \ref{thm:phalf_glue} we express $p(M)$ as a gluing of $c_2(Z_\pm)$
in this sense, but to prove it
we need to know something about how
$\clplift{c_2(Z_\pm)} \in H^4(\clps_\pm)$ relate to some actual bundles.

\subsection{Pre-image of the vertical tangent bundle}

It is natural to ask
whether the pull-back of $\clplift{c_2(Z_\pm)}$ in $H^4(\mtt{Z_\pm})$ equals
the second Chern class of the vertical tangent bundle $T_{vt}(\mtt{Z_\pm})$.
We will in fact need the stronger claim that
there is a bundle $\wh E \to \clps$ such that $c_2(\wh E) = \clplift{c_2(Z)}$,
while the pull-back of $\wh E$ to $\mtt{Z}$ is $T_{vt}(\mtt{Z_\pm})$.

To describe $\wh E$ and related bundles, 
it is convenient to present bundles on $\mtt{Z}$ as mapping tori of a bundle
involution of a bundle on $Z$ that covers $\tau$ (just like $T_{vt}(\mtt{Z})$
itself could be described as the mapping torus of $D\tau : TZ \to TZ$).

Since the normal bundle of $\kd \subset Z$ is trivial, it has a tubular
neighbourhood $\Delta \times \kd$ for a disc $\Delta$. We can think of
$Z$ as the result of gluing $V$ to $\Delta \times \kd$ along
$\bbr \times S^1 \times \kd \cong \Delta^\times \times \kd$.

\begin{lem}
\label{lem:bundle_E}
There exists an $SU(3)$-bundle $E \to Z$ with a bundle isomorphism $\tau_E$
covering $\tau$ such that
\begin{enumerate}
\item $TZ \oplus \trivc$ is $\Z_2$-invariantly isomorphic to $E \oplus -K_Z$
(where $\Z_2$ acts trivially on $\trivc$ and by $D\tau^*$ on $-K_Z$);
\item $E_{|V} \cong TV$, identifying $\tau_E$ with $D\tau$;
\item $E_{|\Delta \times \kd} \cong \trivc \oplus T\kd$, and the restriction
of $\tau_E$ is the corresponding trivial lift.
\end{enumerate}
\end{lem}

\begin{proof}
Given (ii) and (iii), to construct $E$ all that remains is to describe how the
two pieces are glued together.
On a collar neighbourhood $\bbrp \times S^1_\lnn \times \kd$ of the boundary of
$V$, we use the isomorphism
$f : T(\bbrp \times S^1_\lnn \times \kd) \to \trivc \oplus T\kd$
coming from the ``obvious'' $\bbr \times S^1$-invariant trivialisation
of $T(\bbr \times S^1)$.
This matches up the action of $D\tau$ on $TV$ with the trivial action on
$\trivc \oplus T\kd$, so $\tau_E$ is well-defined.

To prove (i), we now describe $-K_Z$ and $TZ$ in similar terms.
If we glue $\trivc \to V$ to $T\Delta \to \Delta \times \kd$ by
$g : 1  \mapsto e^{iu}\contra{z}$, then the resulting line bundle over $Z$
clearly has a section vanishing precisely along~$\kd$, so in other words
it is the complex line bundle $-K_Z$.

We may also consider $TZ$ itself as being obtained by gluing
$TV$ over $V$ to $T(\Delta \times U)$ over $\Delta \times U$
by the derivative of
$\bbrp \times S^1_\lnn \times \kd \cong \Delta^* \times \kd$,
$(t,\anglen) \mapsto z = x+iy = e^{-t-i\anglen}$, 
which equals precisely $(g \times \Id_{T\kd}) \circ f$.

Now let us compare 
$\trivc \oplus TZ$ with $-K_Z \oplus E$.
By the above, we can regard both of them as the result of gluing
$\trivc \oplus TV$ to
$\trivc \oplus T\Delta \oplus T\kd$.
For the first, the gluing map is the composition of
$(\Id \times f) :
\trivc \oplus TV \to \trivc \oplus \trivc \oplus T\kd$ with
$\sm{\Id & 0 \\ 0 & g} \times \Id_{T\kd}$.
For the second, we instead compose with 
$\sm{0 & g \\ \Id & 0} \times \Id_{T\kd}$. All the maps are $\Z_2$-equivariant
provided that we choose the $\Z_2$ action on $\trivc \oplus \trivc \oplus T\kd$
over $\Delta^\times \times \kd$ to be the trivial lift of $\tau$.

Hence the composition of one gluing map with the inverse of the other is
the automorphism 
$\sm{0 & \Id \\ \Id & 0} \times \Id_{T\kd}$ of $\trivc^2 \oplus T\kd$,
which is trivially homotopic to the identity in the space of $\Z_2$-equivariant
complex vector bundle automorphisms, which proves (i).
\end{proof}

By taking the mapping torus of $\tau_E$ we obtain an $SU(3)$-bundle
$\wt E \to \mtt{Z}$, and Lemma \ref{lem:bundle_E}(i) shows that
$\wt E$ is stably isomorphic to
$T_{vt}(\mtt{Z}) \oplus \det(T_{vt}^*(\mtt{Z}))$.
Since $c_1(T_{vt}(\mtt{Z}))$ is Poincar\'e dual to $S^1 \times \kd$, which can be deformed off itself, it squares to 0, so
\begin{equation}
\label{eq:c2tilde}
c_2(\wt E) = c_2(T_{vt}(\mtt{Z})) .
\end{equation}

\begin{rmk}
\label{rmk:kth}
An alternative justification (which will be more crucial below) of
\eqref{eq:c2tilde} that does not rely on Lemma \ref{lem:bundle_E}(i)
is to start by noting that since $\wt E$ and $T_{vt}(\mtt{Z})$ are
isomorphic over $\mtt{V}$, the difference of their $c_2$s lies in
$H^4(\mtt{Z}, \mtt{V}) \cong H^4(\Delta \times \kd, \Delta^\times \times \kd)
\cong H^2_{cpt}(\Delta) \times H^2(\kd)$.
Following Atiyah \cite[\S 2.6]{atiyah67} we can make this more precise by considering the pair $(\wt E, T_{vt}(\mtt{Z}))$ together with the
natural isomorphism over $\mtt{V}$ as an element 
\[ \krel{\wt E}{T_{vt}(\mtt{Z})} \in K(\mtt{Z}, \mtt{V}) \cong
K(\Delta \times \kd, \Delta^\times \times \kd),\]
of relative K-theory. (To reduce notational clutter, we are a little careless
and omit the isomorphism from the notation for the difference element despite
its significance, instead relying on describing the isomorphism in the text.)
Now $D = \krel{\wt E}{T_{vt}(\mtt{Z})}$ has Chern classes in
$H^*(\mtt{Z}, \mtt{V})$ and we can write
\[ c_2(\wt E) = c_2(T_{vt}(\mtt{Z})) + c_1(T_{vt}(\mtt{Z}))c_1(D) + c_2(D) . \] 
Because the image of $D$
in $K(\Delta \times \kd, \Delta^\times \times \kd)$ is a pull-back from
$K(\Delta, \Delta^\times)$, it is clear that
$c_2(D) = c_1(T_{vt}(\mtt{Z}))c_1(D) = 0$.
\end{rmk}

\smallskip
Next, note that the mapping torus $\wt E$ of $\tau_E$ is by construction
identified
with $\trivc \oplus T\kd$ near $S^1 \times \kd \subset \mtt{Z}$. Thus the
fibres over each point on one of the collapsed $S^1$s are all identified,
defining a bundle $\wh E \to \clps$ such that $\rho^*\wh E = \wt E$.

\begin{prop}
\label{prop:bundle_clps}
$c_2(\wh E) = \clplift{c_2(Z)}$ 
\end{prop}

The remainder of this subsection is devoted to the proof of Proposition
\ref{prop:bundle_clps}.
We first construct a further $SU(3)$-bundle
$F \to Z$ with involution $\tau_F$ as follows.

Recall from \S\ref{subsec:inv_blocks} that in addition to the K3 divisor
$\kd$ that is fixed point-wise by $\tau$, there is a second invariant
K3 divisor $\kd' \subset Z$. The fixed set of $\tau$ is the union of $\kd$ and
a curve $C \subset \kd'$.
Consider a tubular neighbourhood $W$ of $C$ in $\kd'$ (so $W \cong$ unit disc
bundle in $N_{C/\kd} \cong T^*C$).
Then $\Delta \times W$ is a tubular neighbourhood of $C$ in $Z$.

We define $F$ as a gluing of $E_{|Z \setminus C}$ and
$T(\Delta \times W)$. The overlap region deformation-retracts to the unit
3-sphere bundle $\Sph$ of $T^*C \oplus \trivc \to C$ (using some arbitrary
hermitian metric on $T^*C$) and the restriction of both bundles to the overlap
is $TC \oplus T^*C \oplus \trivc$. We define $F$ using the gluing map 
\[ \Sph \to SU(TC \oplus T^*C \oplus \trivc), \;
(\alpha, z) \mapsto
\begin{pmatrix}
1 & 0 & 0 \\
0 & z & \alpha \\
0 & \bar \alpha & \bar z
\end{pmatrix} . \]
Next, define a bundle isomorphism $\tau_F : F \to F$ covering $\tau$
by patching up $\tau_E$ over $Z \setminus C$ (where $F \cong E$)
and the trivial lift of $\tau$ over $\Delta \times W$
(where $F \cong TC \oplus T^*C \oplus \trivc$).
This works because on the overlap,
$E \cong TV \cong TC \oplus T^*C \oplus \trivc$ identifies
$\tau_E \cong D\tau \cong \mathrm{diag}(1,-1, -1)$,
which equals the difference of the glue map evaluated at $p \in \Sph$
and $\tau(p)$.

Now because $\tau_F$ acts trivially over the fixed set $\kd \cup C$ of
$\tau$, the quotient defines a bundle $F^0 \to Z^0$, whose pull-back by
$Z \to Z^0$ is $F$.
$F^0$ can also be pulled back to a bundle $\wh F \to \clps$. 

\begin{lem}
\begin{enumerate}
\item
$c_2(F) = c_2(E) + PD(C) \in H^4(Z)$.
\item
$c_2(\wh F) = c_2(\wh E) + PD(S^1 \times C) \in H^4(\clps)$.
\end{enumerate}
\end{lem}

\begin{proof}
Since $F$ and $E$ are constructed to be isomorphic outside $C$, the
difference of their $c_2$s is in the image of $H^4(Z, Z\setminus C) \cong
H^4_{cpt}(\Delta \times W)$, \ie
the difference is a multiple of the Poincar\'e dual $PD(C) \in H^4(Z)$.
In turn, $H^4_{cpt}(\Delta \times W) \cong H^4_{cpt}(A)$ for any fibre $A$
of $T^*C  \oplus \trivc \to C$. We can reason like in Remark \ref{rmk:kth}
and consider the difference element $\krel{E}{F} \in K(Z, Z, \setminus C)$
defined by $E, F$
and the given isomorphism away from $C$, and its image in $K(A, A^\times)$.
The second Chern class of the latter is clearly the generator of
$H^4_{cpt}(A)$, which pins down the coefficient of $PD(C)$ in (i).

(ii) is proved by the same argument.
\end{proof}

\begin{lem}
$c_2(F^0) \in H^4(Z^0)$ is even.
\end{lem}

\begin{proof}
Consider the blow-up $\pi : \overline Z \to Z^0$ in the singular curve $C$.
We will exploit that $c_2(\overline Z) \in H^4(\overline Z)$ is even because
$\overline Z$ is a smooth complex 3-manifold with 
$c_1(\overline Z)^2 = 0$ \cite[Lemma 5.10]{cym}.

Let $E \subset \overline Z$ be the exceptional set (a $\PP^1$-bundle over $C$).
The sequence
$0 \to H^4(Z^0) \stackrel{\pi^*}{\to} H^4(\overline Z) \to H^4(E) \to 0$
is split exact, so it suffices to show that $c_2(\pi^* F^0)$ is even
in $H^4(\overline Z)$.

Analogously to Remark \ref{rmk:kth}, we consider the difference element
$\krel{\pi^*F}{TZ} \in K(\overline Z, \overline Z \setminus E)$
defined by the isomorphism between $\pi^* F^0$ and $TZ$ away from $E$,
and its second Chern class
$\ctrel{\pi^* F^0}{TZ} \in H^4( \overline Z, \overline Z \setminus E)
\cong H^4_{cpt}(U)$,
for a tubular neighbourhood $U$ of $E$.
Then $c_2(\pi^*F^0) - c_2(Z)$ is the image of
$\ctrel{\pi^*F^0}{TZ}$ (since $c_1(F^0) = 0$).

Thinking of $E$ as the projectivisation of the rank 2 bundle
$\trivc \oplus T^*C$ over $C$, $U$ is the total space of $\calo_E(-2)$.
$TU$ splits as $T_{vt}U \oplus TC$. $T_{vt}U$ can itself be further split as
a pull-back of the line bundle $\calo_E(-2)$ itself, and the pull-back
of the line bundle $T_{vt}E$ over $E$.
Meanwhile the restriction of $F^0$ to a neighbourhood of $C$ is by construction
the pull-back of $\trivc \oplus T^*C \oplus TC$ from $C$, and hence the same
is true for $\pi^* F^0$ over~$U$.

The identification of these bundles along the boundary of $U$ maps
\begin{itemize}
\item the $TC$ summand in $TU$ identically to the $TC$ summand in $\pi^* F^0$,
\item the $\calo_E(-2)$ summand to $\trivc$, taking the ``outward'' section of
$\calo_E(-2)$ to a constant one in $\trivc$
\item the pull-back of $T_{vt}E$ to $T^*C$.
\end{itemize}
Writing 
\[ \ctrel{TZ}{\pi^*F^0} = \ctrel{T_{vt}U}{\trivc \oplus T^*C} +
c_1(TC)\corel{T_{vt}U}{\trivc \oplus T^*C} , \]
the second term will always be even because $c_1(C)$ is.
In turn,
\[ 
\ctrel{T_{vt}U}{\trivc \oplus T^*C} =
c_1(\calo_E(-2))\corel{T_{vt}E}{T^*C} + \corel{\calo_E(-2)}{\trivc}c_1(T^*C) ,\]
and the factors $c_1(\calo_E(-2))$ and $c_1(T^*C) \in H^2(U)$ are both
even. (Looking closer at the identifications at the boundary one can also see
\eg that $\corel{\calo_E(-2)}{\trivc} \in H^2_{cpt}(U)$ is the Poincar\'e dual
to $E$, but that does not actually seem necessary if we just need to know
the parity of $\ctrel{TZ}{\pi^* F^0}$.)
\end{proof}

Now it is clear that $\clplift{c_2(F)} = c_2(\wh F)$,
and (since $PD(C)$ is even, which we could also see as a consequence of
Lemma \ref{lem:smoothing_top}) that $\clplift{PD(C)} = PD(S^1 \times C)$,
completing the proof of Proposition~\ref{prop:bundle_clps}.

\begin{rmk}
\label{rmk:premagic}
As a by-product of the above lemmas, 
we find that the mod 2 residue of $c_2(\wh E)$
is the Poincar\'e dual of $S^1 \times C$, so it is \emph{not} even in general.
At first sight it seems disconcerting that some of the
intermediate steps in the calculation of $p(M)$ are odd, even though $p(M)$
itself must always be even by Lemma \ref{lem:p}.
The explanation is that thanks to
Lemma \ref{lem:smoothing_top}, the parity of $PD(S^1 \times C)$ is controlled
by the bilinear form on $N$, which of course also controls the matchings. 

\end{rmk}

\subsection{Completing the proof of Theorem \ref{thm:phalf_main}}
\label{subsec:p_end}

We are now ready to express $p(M)$ in terms of the gluing map $Y$ from
Definition \ref{def:gluemap}.

\begin{thm}
\label{thm:phalf_glue}
$p(M) = -Y(c_2(Z_+), c_2(Z_-))$.
\end{thm}

\begin{proof}
We define a $Spin(7)$-bundle
$T \to \smashed$ such that
\begin{enumerate}
\item
$j_\pm^* T$ is isomorphic to $\wh E_\pm \oplus \trivr$, so in particular
$j_\pm^* p(T) = -c_2(\wh E_\pm)$, and
\item
$\smashmap^* T$ is isomorphic to $TM$.
\end{enumerate}
The claim then follows immediately Proposition \ref{prop:bundle_clps} and
the definition of $Y$.

If we define $T$ by gluing $\wh E_+ \oplus \trivr$ to $\wh E_- \oplus \trivr$
in any way over the overlap $\kd = \clps_+ \cap \clps_- \subset \smashed$
then (i) will automatically hold, so we just need to choose the gluing map
to ensure (ii) holds too.

The construction of $\wh E_\pm$ amounts to a gluing of $T_{vt}(\mtt{V_\pm})$ over
$V_\pm \subset \clps_\pm$ to $\trivc \oplus T\kd$ over a neighbourhood of $\kd$ in $\clps_\pm$,
using the bundle isomorphism $f_\pm : (x\contra{t}, y\contra{u}, v) \mapsto (x+iy, v)$
from the proof of Lemma \ref{lem:bundle_E}.
Thus the gluing map we use to construct $T$ over $\smashed$ should be a bundle map
$h : \trivr \oplus \trivc \oplus T\kd \to \trivr \oplus \trivc \oplus T\kd$ over $\kd$.

Meanwhile, the tangent bundle of the XTCS $M$ can be viewed as a gluing of
$T(\mtt{V_+})$ and $T(\mtt{V_-})$ by the derivative of the map $F$ from
\eqref{eq:gluemap} that is used to glue together $\mtt{V_+}$ to $\mtt{V_-}$.
The crucial point is that the bundle map $DF$ over $\bbr \times T^2 \times \kd$
clearly depends only on the $\kd$ factor. To make sense of this more formally
we first need to identify the bundles with pull-backs from $\kd$, so we define
\[ \tilde f_\pm : T(\mtt{V_\pm}) \to \trivr \oplus \trivc, \;
(x\contra{t}, y\contra{u}, s\contra{v}, w) \mapsto (s, x+iy, w) . \]
Then the composition $\tilde f_+^{-1} \circ DF \circ f_- :
\trivr \oplus \trivc \oplus T\kd \to \trivr \oplus \trivc \oplus T\kd$
is a pull-back of a bundle map $h$ over the $\kd$ factor.
(To be really explicit, the action of $h$ on the $\bbr \times \bbc$ factor is the
conjugation of $(s, z) \mapsto (-s, e^{i\vartheta} \bar z)$ by $(s, x+iy) \mapsto (x, s+iy)$,
so is independent of the coordinates on the base). If we use that $h$ in the
construction of $T$, then $\smashmap^* T = TM$ as desired.
\end{proof}

To complete the proof of Theorem \ref{thm:phalf_main}, it remains to
explain how to interpret Theorem \ref{thm:phalf_glue} in terms of $c_2(Z_\pm)$,
presented in the form \eqref{eq:c2cbar};
\ie we write $c_2(Z_\pm) = g_\pm\bar c_2(Z_\pm) + 24h_\pm$ for some
$\bar c_2(Z_\pm) \in N_\pm^*$ and some $h_\pm \in H^4(Z_\pm)$
whose image in $H^4(\kd)$ is a generator, with $h_\pm$ assumed to be
$\tau$-invariant if $Z_\pm$ is an involution block.

Recall that 
$g_\pm : H^2(\kd) \to H^4(Z_\pm)$ is the Poincar\'e dual
of the restriction map.
Define $\tilde g_\pm : H^2(\kd) \cong H^2(S^1 \times \kd) \to H^4(\mtt{Z_\pm})$
analogously, and recall from Notation \ref{notn:abuse} that on the
cross-section
$(S^1_{\lnx_\pm} \times S^1_{\lnn_\pm})/(\antip \times \antip) \times \kd
\cong T^2 \times \kd$ of $(S^1_{\lnx_\pm} \times M_\pm)/(\antip \times \tau)$,
$2\drn_\pm \in H^1(T^2 \times \kd)$ denotes the (primitive) element whose
pull-back to $S^1_\lnx \times S^1_\lnn \times \kd$ corresponds to twice the
generator $\drn_\pm$ from the internal $S^1$ factor.

\begin{lem}
\label{lem:lift_snake}
For any $y \in H^2(\kd)$%
\[ Y(2g_+(y)), 0) = \delta((2\drn_+) y)  \]
and
\[ Y(0, 2g_-(y))) = -\delta((2\drn_-) y) , \]
where $\delta : H^3(T^2 \times \kd) \to H^4(M)$ is the Mayer-Vietoris snake map.
\end{lem}

\begin{proof}
Recall that $g_+(y)$ can be described as $i_{+*} \partial_+ (\drn_+ y)$, where
$i_{+*} : H^4_{cpt}(V_+) \to H^4(Z_+)$ is the push-forward of the inclusion
$i : V_+ \to Z_+$ and $\partial_+ : H^3(S^1 \times \kd) \to H^4_{cpt}(V_+)$
is the snake map in the relative
cohomology sequence for the pair $(Z_+, \kd)$.
From the cochain description of \eqref{eq:clplift} it is clear that
\[ \clplift{2g_+(y)} =
\hat\iota_{+*} \tilde \partial_+ (2\drn_+ y)), \]
for the push-forward $\hat \iota_{+*} : H^4_{cpt}(\mtt{V_+}) \to H^4(\clps_+)$
of the inclusion $\hat \iota_+ : \mtt{V_+} \to \clps_+$,
and the snake map $\tilde \partial_+ : H^3(T^2 \times \kd) \to H^4_{cpt}(\mtt{V_+})$.

Now the composition $\smashmap^* \circ \hat \iota_{+*} : H^4_{cpt}(\mtt{V_+}) \to H^4(M)$
is simply the push-forward of the inclusion $\mtt{V_+} \to M$, and its
composition with
$\tilde \partial_+$ equals $\delta$.
Hence
\[ Y(2\tilde g_+(y), 0)
= \smashmap^*(\hat \iota_{+*} \tilde \partial_+((2\drn_+)y)))
= \delta((2\drn_+)y) . \qedhere \]
\end{proof}

\begin{cor}
\label{cor:phalf}
Suppose that $c_2(Z_\pm) = g_\pm\bar c_2(Z_\pm) + 24h_\pm$ as in
\eqref{eq:c2cbar}. Then 
\[ p(M) = \delta(\drn_+ \bar c_2(Z_+) - \drn_- \bar c_2(Z_-))
+ 12Y({2h_+}, {2h_-}). \]
\end{cor}

\pagebreak[2]

\begin{lem}
Let $(h_+, h_-) \in H^4(Z_+)^\tau \oplus_0 H^4(Z_-)^\tau$.
Then 
\[ Y(2h_+, 2h_-) = (k_+)_*(s_+(\wh B_+(h_+))) + (k_-)_*(s_-(\wh B_-(h_-)))
\mod 2 ,\]
where $k_\pm : \mtt{V} \to M$ are the obvious inclusions and
$s_\pm : H^3_{cpt}(V) \to H^4_{cpt}(\mtt{V})$ are the snake maps in the
compactly supported version of the exact sequence \eqref{eq:mv_maptorus}
for the cohomology of the mapping torus. 
\end{lem}

\begin{proof}
Note that the mod 2 residue of $\clplift{2h_\pm}$ in
$H^4(\clps_\pm)$ is $(\hat \iota_\pm)_\pm(s_\pm(\wh B_\pm(h_\pm)))$,
for $\hat \iota_\pm$ the inclusion $\mtt{V_\pm} \to \clps_\pm$ as before,
and that the maps $j_\pm$ and $\gamma$ in the definition of $Y$ satisfy
$j_\pm = \gamma \circ k_\pm \circ \hat \iota_\pm$.
\end{proof}

If $\wh B_\pm(h_\pm) = 0$, then 
the final term in Corollary \ref{cor:phalf} is divisible by 24,
completing the proof of Theorem \ref{thm:phalf_main}.
In general
\begin{itemize}
\item If either of $\wh B_\pm$ has non-zero image in
$H^3(V_\pm)/\im(\Id + \tau^*)$, then the image of $Y(2h_+, 2h_-)$ in $H^4(M_\pm)$ is odd, and the image of $p(M)$ in $H^4(M_\pm)$ is divisible by exactly 12.
Thus the class of $(H^4(M), p(M))$
is determined by the mod 12 residue of $p(M)$, which
Corollary \ref{cor:phalf} says is equal to
$\delta(\drn_+ \bar c_2(Z_+) - \drn_- \bar c_2(Z_-))$.
In particular, if the involution blocks are pleasant then
the greatest divisor of $p(M)$ is
$\gcd(12,n)$, where $n$ is the greatest integer dividing \eqref{eq:mod24res}.
\item If both $\wh B_\pm$ can be chosen to have zero image in $H^3(V_\pm)$,
then we can write $\wh B_\pm = \partial_\pm b_\pm$ for some $b_\pm \in L$,
and $p(M) = \delta\big(\drn_+ \bar c_2(Z_+) - \drn_- \bar c_2(Z_-)
+ 12\drx_+ b_+ - 12\drx_- b_-\big) \mod 24$. In particular, if the involution
blocks are pleasant then the greatest integer dividing $p(M)$ is $\gcd(24,n)$
for $n$ the greatest integer dividing
$\drn_+ \bar c_2(Z_+) - \drn_- \bar c_2(Z_-) + 12\drx_+ b_+ - 12\drx_- b_-
\in H^3(T^2 \times \kd) \mod I_+ + I_-$.
\end{itemize}

\subsection{\fpif-twisted connected sums}
\label{subsec:s1}

Now we describe how to work out the torsion in $H^4(M)$ and the divisibility
of $p(M)$ in the case $\vartheta = \frac{\pi}{4}$, and carry it out for some
examples.

As described before, we use a block $Z_+$ with involution and an ordinary block
$Z_-$. We assume that $Z_+$ is pleasant, in order that $H^4(M_+)$ is
torsion-free. Therefore the only contribution to the torsion comes from the
Mayer-Vietoris map $\delta : H^3(T^2 \times \kd) \to H^4(M)$, whose image is a
split summand in~$H^4(M)$.

By Lemma \ref{lem:mtres} the assumption that $Z_+$ is pleasant ensures that
the image of $H^3(M_+) \to H^3(T^2 \times \kd)$ is exactly
\begin{equation}
\begin{aligned}
I_+ := \{\drx_+ n + \drn_+ t :
\; n \in N_+, \, t \in T_+, \, n+t = 0 \mod 2L\} .
\end{aligned}
\end{equation}
On the other hand, the image of $H^3(M_-)$ is just
\[ I_- := \drx_- N_- \oplus \drn_- T_- . \]
The image $\delta(H^3(T^2 \times \kd))$ is isomorphic to
$H^3(T^2 \times \kd)/(I_+ + I_-)$.

To make this more manageable, note that $\{2\drn_+ , \drn_- \}$
is a basis of $H^1(T^2)$, and that we may define a surjective homomorphism
\begin{equation}
\begin{aligned}
\label{eq:s1proj}
H^3(T^2 \times \kd) &\to N_+^* \oplus N_-^*, \\
2\drn_+ x + \drn_- y &\mapsto (\flat^+(x), \flat^-(y))
\end{aligned}
\end{equation}
for $x, y \in L$, where $\flat^\pm : L \to N_\pm^*$ is defined by the
intersection form. Elements in the kernel of \eqref{eq:s1proj} have $x \in T_+$,
$y \in T_-$, so definitely lie in $I_+ + I_-$.
Hence
\begin{equation}
\label{eq:s1induced}
\delta(H^3(T^2 \times \kd)) \cong (N_+^* \oplus N_-^*)/(\bar I_+ + \bar I_-),
\end{equation}
where $\bar I_\pm$ is the image of $I_\pm$ under \eqref{eq:s1proj}. Using that 
\[ \drx_+ = \drn_+ + \drn_-, \quad
\drx_- = 2\drn_+ + \drn_- \]
in a $\vartheta = \frac{\pi}{4}$ matching, we find
\enlargethispage{0.8\baselineskip}
\begin{equation}
\label{eq:torsion_aid}
\begin{aligned}
\bar I_+ &= \{ (\half \flat^+(x), \flat^-(x)) : x \in \bar N_+ \}, \\
\bar I_- &= \{ (\flat^+(y), \flat^-(y)) : y \in N_- \},
\end{aligned}
\end{equation}
where $\bar N_+ = \{x \in N_+ : \flat^+(x) \in 2N_+^* \}$.

\begin{prop}
\label{prop:pi4detail}
Let $M$ be a $\frac{\pi}{4}$-twisted connected sum of blocks $Z_+$ and $Z_-$,
where $Z_+$ is pleasant, and let
\[ \wh W : \bar N_+ \oplus N_- \to N_+^* \oplus N_-^*,
\; (x, y) \mapsto (\half\flat^+(x) + \flat^+(y), \; \flat^-(x) + \flat^-(y) )
. \]
Then
\begin{enumerate}
\item $\delta(H^3(T^2 \times \kd)) \cong \coker \wh W$.
\item
Under the hypotheses of Corollary \ref{cor:phalf2},
this isomorphism maps $p(M) \mmod 24$ to the image of
$(\half \bar c_+, -\bar c_-)$.
\item
Let $z_1, z_2 \in \Tor \delta(H^3(T^2 \times \kd))$, let 
$(\alpha_1, \beta_1), (\alpha_2, \beta_2) \in N_+^* \oplus N_-^*$
be representatives of the images of $z_1$ and $z_2$ in $\coker \wh W$,
and pick $(x,y) \in \bar N_+ \times N_-$ such that
$m(\alpha_1, \beta_1) = \wh W(x,y)$. %
Then $b_M(z_1, z_2) = \frac{1}{m}(\alpha_2(x) + \beta_2(y)) \in \Q/\Z$.
\end{enumerate}
\end{prop}

\begin{proof}
(i) is proved in the preceding discussion, while (ii) is immediate from
Corollary \ref{cor:phalf2}.

For (iii), let $p_1, p_2$ be $\delta$-pre-images of $z_1, z_2$ in
$H^3(T^2 \times \kd)$
According to Lemma \ref{lem:torlink},
\[ b_M(z_1, z_2) = \frac{1}{m}p_1^- p_2, \]
where $mp_1 = p_1^+ - p_1^- \in I_+ + I_-$. Now $\delta(p_1) = z_1$
means that \eqref{eq:s1proj} maps $p_1$ to $(\alpha_1, \beta_1)$.
Therefore $m(\alpha_1, \beta_1) = \wh W(x,y)$ means that
$mp_1 = \drn_+ (x + 2y + t_+) + \drn_-(x+y+t_-)$
for some $t_\pm \in T_\pm$ (with $x + t_+$ even). Therefore
\begin{gather*}
p_1^+ = \drx_+ x + \drn_+ t_+ = \drn_+ (x + t_+) + \drn_- x, \\
-p_1^- = \drx_- y + \drn_- t_- = 2\drn_+ y + \drn_- (y+ t_-),
\end{gather*}
and in particular $y+t_- = -x \mmod m$.
Hence, writing $p_2 = 2\drn_+ w_+ + \drn_- w_-$ for $w_\pm \in L$
(so that $\alpha_2 = \flat^+(w_+)$ and $\beta_2 = \flat^-(w_-)$),
\begin{align*}
p_1^- p_2 &= -(\drx_- y + \drn_- t_-)(\drx_- w_+ + \drn_- (w_- - w_+))
= -t_- w_+ + y(w_- - w_+) \\
&= yw_- - w_+(y + t_-) = yw_- + w_+x = \beta_2(y) + \alpha_2(x) \mod m . \qedhere
\end{align*}
\end{proof}

Now consider the case when $N_+$ and $N_-$ are purely at angle $\frac{\pi}{4}$.
Let $\pi_\pm : N_\mp(\bbr) \to N_\pm(\bbr)$ be the orthogonal projections,
and recall that pure angle $\frac{\pi}{4}$ means that
$\pi_\pm(x). \pi_\pm(y) = \half x.y$ for any $x, y \in N_\mp(\bbr)$.
In particular, note that $\pi_-^*N_-^* \subset N_+^*(\bbr)$ equals
$(2\pi_+N_-)^*$. Therefore $N_+^* + 2\pi_-^*N_-^* = (N_+ \cap 2\pi_+N_-)^*$,
and we get a surjective homomorphism
\begin{equation}
\label{eq:pures1proj}
\begin{aligned}
N_+^* \oplus N_-^* &\to (N_+ \cap 2\pi_+N_-)^*, \\
 (\alpha, \beta) &\mapsto \alpha - \pi_-^*\beta .
\end{aligned}
\end{equation}
Note further that $\pi_-^* \circ \flat^-$ equals $\flat^+$ on $N_-(\bbr)$ and
$\half \flat^+$ on $N_+(\bbr)$. Therefore $\bar I_\pm$ are both contained in
the kernel of \eqref{eq:pures1proj}. The kernel is in fact
\[ \{ (\alpha, \beta) : \alpha = \pi_-^*\beta \in (N_+ + 2\pi_+N_-)^* \}, \]
isomorphic to $(N_+ + 2\pi_+N_-)^*$ by projection to the first component.
The images of $\bar I_\pm$ in there are simply $\flat^+(\half \bar N_+)$ and
$\flat^+(N_-)$, respectively. Notice that
\begin{equation}
\label{eq:imsum}
\flat^+(\half \bar N_+) + \flat^+(N_-)
\; = \; \{\half \flat^+(x) \; : \; x \in N_+ + 2\pi_+N_-, \;
\flat^+(x) \in 2(N_+ + 2\pi_+N_-)^* \} .
\end{equation}
Hence there is a surjection $f$ from the discriminant group $\Delta$ of the
even integral lattice $N_+ + 2\pi_+N_-$ to the coquotient of
$\im \wh W = \bar I_+ + \bar I_-$ in the kernel of \eqref{eq:pures1proj}, with
kernel precisely the 2-torsion $T_2\Delta$;
thus $\Tor \delta(H^3(T^2 \times \kd)) \cong \Delta/T_2\Delta$.

To evaluate the torsion linking form on a pair of elements in $\Tor H^4(M)$
corresponding to images in $\Delta$ of
$\alpha_1, \alpha_2 \in (N_+ + 2\pi_+N_-)^*$, note that the corresponding
elements in $N_+^* \oplus N_-^*$ are $(\alpha_i, 2\pi_+^*\alpha_i)$.
If $m\alpha_1 = \half \flat^+(x + 2\pi_+y)$ for $x \in \bar N_+$
and $y \in N_-$, then $m(\alpha_1, 2\pi_+^*\alpha_1) = \wh W(x,y)$ and
Proposition \ref{prop:pi4detail} gives
$b_M(f(\alpha_1), f(\alpha_2)) = \frac{1}{m}(\alpha_2(x) +
(2\pi_+^*\alpha_2)(y)) = \frac{1}{m}\alpha_2(x+ 2\pi_+y)$.
In summary

\begin{cor}
\label{cor:pi4pure}
For a pure $\frac{\pi}{4}$ matching where $Z_+$ is pleasant
\begin{itemize}
\item There is an isomorphism $f :  \Delta/T_2\Delta \to \Tor H^4(M)$. 
\item For $x, y \in \Delta$, $b_M(f(x), f(y)) = 2b_\Delta(x, y)$, where
$b_\Delta$ is the discriminant form of $\Delta$.
\item The free part of $\delta(H^3(T^2 \times \kd))$ is naturally isomorphic to
$(N_+ \cap 2\pi_+N_-)^* \stackrel{2\pi_+^*}{\cong} (\pi_-N_+ \cap N_-)^*$.
\item The image of $p(M) \mmod 24$ in the free part of
$\delta(H^3(T^2 \times \kd))$ corresponds to
$\half\bar c_+ + \pi_-^* \bar c_- \in (N_+ \cap 2\pi_+N_-)^*$, or
$\pi_+^* \bar c_+ + \bar c_- \in (\pi_-N_+ \cap N_-)^*$.
\end{itemize}
\end{cor}

Note in particular that if $N_+$ has 2-elementary discriminant then
automatically $\Delta = T_2\Delta$ and $2\pi_+N_- \subseteq N_+$
(and $\pi_-N_+ \supseteq N_-$), so $H^4(M)$ is torsion-free, and the direct
summand $\delta(H^3(T^2 \times \kd)) \subseteq H^4(M)$ is naturally isomorphic
to $N_-^*$.

\begin{rmk}
\label{rmk:magic}
In Remark \ref{rmk:premagic} we pointed out that some of the expressions
for $p(M)$ are not obviously even, even though Lemma \ref{lem:p} tells us that
$p(M)$ is even for any closed 7-manifold.
The appearance of $\half$ as the coefficient
of $\bar c_+$ in Corollary \ref{cor:pi4pure} is an instance of this:
$\bar c_- \in N_-$ is even, and hence so is its contribution to $p(M)$,
but it is not
obvious why that of $\half\bar c_+$ should be too. Indeed, $\bar c_+$ need not
be even considered as an element of $2N_+^* + \flat(N_+)$, and for
arbitrary even elements $c \in N_+^*$, the image of $\half c$ in
$(N_+ \cap 2\pi_+N_-)^*$ need not be even.

However, note that $N_+ \cap 2\pi_+N_- \subseteq \bar N_+$, and that any
$x \in N_+ \cap 2\pi_+N_-$ has $x^2$ divisible by~4.
Meanwhile Remark \ref{rmk:premagic} and Lemma \ref{lem:smoothing_top}
imply that $\bar c_+ x = x^2 \mmod 4$ for any $x \in \bar N_+$,
explaining why $\half \bar c_+$ must be even as an element of
$(N_+ \cap 2\pi_+N_-)^*$.
\end{rmk}

\subsection{\fpis-twisted connected sums}
\label{subsec:h1}

Now we move on to describing the torsion in $H^4(M)$ and the divisibility
of $p(M)$ in the case $\vartheta = \frac{\pi}{6}$.
The calculations are very similar to the case $\thet = \frac{\pi}{4}$, but the
details are just sufficiently different to require repetition.

We use a pair of involution blocks $Z_\pm$, but recall that there is a basic
asymmetry in the set-up (see Figure \ref{fig:1/6}).
We assume that $Z_\pm$ are both pleasant, in order that $H^4(M_\pm)$ are 
torsion-free. Therefore the only contribution to the torsion comes from the
Mayer-Vietoris map $\delta : H^3(T^2 \times \kd) \to H^4(M)$, whose image is a
split summand in $H^4(M)$.
By Lemma \ref{lem:mtres} the image of $H^3(M_\pm) \to H^3(T^2 \times \kd)$ is
exactly
\begin{equation}
\label{eq:s01}
\begin{aligned}
I_\pm := \{\drx_\pm n + \drn_\pm t :
\; n \in N_\pm, \, t \in T_\pm, \, n+t = 0 \mod 2L\} .
\end{aligned}
\end{equation}
The image $\delta(H^3(T^2 \times \kd))$ is isomorphic to
$H^3(T^2 \times \kd)/(I_+ + I_-)$.

Note that $\{2\drn_+ , 2\drn_- \}$
is a basis of $H^1(T^2)$, so that we may define a surjective homomorphism
\begin{equation}
\begin{aligned}
\label{eq:h1proj}
H^3(T^2 \times \kd) &\to N_+^* \oplus N_-^*, \\
2\drn_+ x + 2\drn_- y &\mapsto (\flat^+(x), \flat^-(y))
\end{aligned}
\end{equation}
for $x, y \in L$, where $\flat^\pm : L \to N_\pm^*$ is defined by the
intersection form. Elements in the kernel of \eqref{eq:h1proj} have
$x \in T_+$, $y \in T_-$, so definitely lie in $I_+ + I_-$. This reduces the
problem to understanding the image of the induced isomorphism 
\begin{equation}
\label{eq:h1induced}
\delta(H^3(T^2 \times \kd)) \cong (N_+^* \oplus N_-^*)/(\bar I_+ + \bar I_-),
\end{equation}
where $\bar I_\pm$ is the image of $I_\pm$ under \eqref{eq:h1proj}. Using that 
\[ \drx_+ = \drn_+ + 2\drn_-, \quad
\drx_- = 2\drn_+ + 3\drn_- \]
in a $\vartheta = \frac{\pi}{6}$ matching, we find
\begin{align*}
\bar I_+ &= \{ (\half \flat^+(x), \flat^-(x)) : x \in \bar N_+ \}, \\
\bar I_- &= \{ (\flat^+(y), \threehalf \flat^-(y)) : y \in \bar N_- \},
\end{align*}
where $\bar N_\pm = \{x \in N_\pm : \flat^\pm(x) \in 2N_\pm^* \}$.

\enlargethispage{1.5\baselineskip}

\begin{prop}
\label{prop:pi6detail}
Let $M$ be a $\frac{\pi}{6}$-twisted connected sum of pleasant involution
blocks $Z_+$ and $Z_-$, and let
\[ \wh W : \bar N_+ \oplus \bar N_- \to N_+^* \oplus N_-^*,
\; (x, y) \mapsto (\half\flat^+(x) + \flat^+(y),
\; \flat^-(x) + \threehalf\flat^-(y) ) . \]
Then
\begin{enumerate}
\item $\delta(H^3(T^2 \times \kd)) \cong \coker \wh W$.
\item
Under the hypotheses of Corollary \ref{cor:phalf2},
this isomorphism maps $p(M) \mmod 24$ to the image of
$(\half \bar c_+, -\half \bar c_-)$.
\item
Let $z_1, z_2 \in \Tor \delta(H^3(T^2 \times \kd))$, let 
$(\alpha_1, \beta_1), (\alpha_2, \beta_2) \in N_+^* \oplus N_-^*$
be representatives of the images in
$(N_+^* \oplus N_-^*)/(\bar I_+ + \bar I_-)$, and pick
$(x,y) \in \bar N_+ \times \bar N_-$ such that
$m(\alpha_1, \beta_1) = \wh W(x,y)$. 
Then $b_M(z_1, z_2) = \frac{1}{m}(\alpha_2(x) + \beta_2(y)) \in \Q/\Z$.
\end{enumerate}
\end{prop}

\begin{proof}
(i) is proved in the preceding discussion, while (ii) is immediate from
Corollary \ref{cor:phalf2}.

For (iii), let $p_1, p_2$ be $\delta$-pre-images of $z_1, z_2$ in
$H^3(T^2 \times \kd)$.
According to Lemma \ref{lem:torlink},
\[ b_M(z_1, z_2) = \frac{1}{m}p_1^- p_2, \]
where $mp_1 = p_1^+ - p_1^- \in I_+ + I_-$. Now
\begin{gather*}
p_1^+ = \drx_+ x + \drn_+ t_+ = \drn_+ (x + t_+) + 2\drn_- x, \\
-p_1^- = \drx_- y + \drn_- t_- = 2\drn_+ y + \drn_- (3y + t_-),
\end{gather*}
for some $t_\pm \in T_\pm$ (with $x + t_+$ and $y + t_-$ both even).
In particular, $\frac{3y+t_-}{2} = -x \mmod m$.
Hence, writing $p_2 = 2\drn_+ w_+ + 2\drn_- w_-$ for $w_\pm \in L$,
\begin{align*}
p_1^- p_2 &= -(\drx_- y + \drn_- t_-)(\drx_- w_+ + \drn_- (2w_- - 3w_+))
= \half(-t_- w_+ + y(2w_--3w_+)) \\
&= yw_- - w_+\frac{3y + t_-}{2} = yw_- + w_+x = \beta_2(y) + \alpha_2(x) \mod m . \qedhere
\end{align*}
\end{proof}

Now let us assume that $N_+$ and $N_-$ are purely at angle $\frac{\pi}{6}$.
Let $\pi_\pm : N_\mp(\bbr) \to N_\pm(\bbr)$ be the orthogonal projections,
and recall that pure angle $\frac{\pi}{6}$ means that
$\pi_\pm(x). \pi_\pm(y) = \threequart x.y$ for any $x, y \in N_\mp(\bbr)$.
In particular, see that $\twothird \pi_-^*N_-^* \subset N_+^*(\bbr)$ equals
$(2\pi_+N_-)^*$. We can therefore surjectively map
\begin{equation}
\label{eq:pureh1proj}
\begin{aligned}
N_+^* \oplus N_-^* &\to (N_+ \cap 2\pi_+N_-)^*, \\
 (\alpha, \beta) &\mapsto \alpha - \twothird \pi_-^*\beta .
\end{aligned}
\end{equation}
Note further that $\pi_-^* \circ \flat^-$ equals $\flat^+$ on $N_-(\bbr)$ and
$\threequart \flat^+$ on $N_+(\bbr)$. Therefore $\bar I_\pm$ are both contained
in the kernel of \eqref{eq:pureh1proj}. The kernel is in fact
\[
\{ (\alpha, \beta) :
\alpha = \twothird \pi_-^* \beta \in (N_+ + 2\pi_+N_-)^* \},
\]
isomorphic to $(N_+ + 2\pi_+N_-)^*$ by projection to the first component.
The images of $\bar I_\pm$ in there are simply $\flat^+(\half \bar N_+)$ and
$\flat^+(\bar N_-)$, respectively.
Like in the $\vartheta = \frac{\pi}{4}$ case, their sum is described by
\eqref{eq:imsum}, implying that
the coquotient of $\bar I_+ + \bar I_-$ in the kernel of \eqref{eq:pureh1proj}
is isomorphic to the discriminant group $\Delta$ of the even integral
lattice $N_+ + 2\pi_+N_-$ modulo its 2-torsion $T_2\Delta$.

Similarly to Corollary \ref{cor:pi4pure} we thus obtain

\begin{cor}
\label{cor:pi6pure}
For a pure $\frac{\pi}{6}$ matching where $Z_\pm$ are both pleasant 
\begin{itemize}
\item There is an isomorphism $f :  \Delta/T_2\Delta \to \Tor H^4(M)$. 
\item For $x, y \in \Delta$, $b_M(f(x), f(y)) = 2b_\Delta(x, y)$, where
$b_\Delta$ is the discriminant form of $\Delta$.
\item
the free part of $\delta(H^3(T^2 \times \kd))$ is naturally isomorphic to
$(N_+ \cap 2 \pi_+ N_-)^* \stackrel{2\pi_+^*}{\cong}
(\twothird \pi_- N_+ \cap N_-)^*$.
\item the image of $p(M) \mmod 24$ in the free part of $\delta(H^3(T^2 \times \kd))$
corresponds to
$\half \bar c_+ + \third \pi_+^* \bar c_- \in (N_+ \cap 2\pi_+ N_-)^*$,
or $\pi_+^* \bar c_+ + \half \bar c_- \in (\twothird \pi_- N_+ \cap N_-)^*$.
\end{itemize}
\end{cor}

Note in particular that if $N_+$ has 2-elementary discriminant then
automatically $\Delta = T_2\Delta$ and $2\pi_+N_- \subseteq N_+$
(or $N_- \subseteq \twothird \pi_-N_+$), so
$H^4(M)$ is torsion-free, and $\delta(H^3(T^2 \times \kd))$ 
is naturally isomorphic to $N_-^*$.
(The asymmetry of the construction entails that $N_-$ being 2-elementary is
not as helpful: note that $2\pi_+N_-$ is isometric to $N_-(3)$ which always
has some 3-primary discriminant.)

\subsection{Further invariants and classification results}
\label{subsec:nu_and_class}

Any metric of holonomy $G_2$ has an associated torsion-free \gtstr.
To a \gtstr{} $\varphi$ on closed 7-manifolds, \cite[Definition 1.2]{nu}
associates a value $\nu(\varphi) \in \Z/48$ which is invariant under
diffeomorphisms and homotopies, and can thus in particular distinguish
components of the moduli space of metrics of holonomy $G_2$.

A stronger invariant $\bar \nu(\varphi) \in \Z$ is introduced
in \cite[Definition 1.4]{eta}; for manifolds with holonomy $G_2$ the value
of $\nu(\varphi)$ is recovered by
\begin{equation}
\nu(\varphi) = \bar \nu(\varphi) + 24 \mod 48 .
\end{equation}
 For extra-twisted connected sums
(involving only involutions as in this paper), it can be computed purely in
terms of the gluing angle $\thet$ and the configuration angles of the matching
(Definition \ref{def:angles}).

\begin{thm}[{\cite[Corollary 2]{eta}}]
\label{thm:nubar}
Let $(M,\varphi)$ be an extra-twisted connected sum \gtmfd{} as in Construction
\ref{constr:xtcs} with gluing angle $\thet$. Set $\rho := \pi-2\thet$.
Then
\[ \bar\nu(\varphi) = -72\frac{\rho}{\pi} + 3(\mathrm{sign}\,\rho)
    \Bigl(\#\bigl\{\,j
		\bigm|\alpha_j^-\in\{\pi-|\rho|,\pi\}\,\bigr\}-1
        +2\,\#\bigl\{\,j
		\bigm|\alpha_j^-\in(\pi-|\rho|,\pi)\,\bigr\}\Bigr)\;,
\]
where $\alpha^-_1, \ldots, \alpha^-_{19}$ are the configuration angles of
the configuration of the \hk rotation used in the construction.
\end{thm}

There are a number of further invariants of closed 7-manifolds with \gtstr{}
that we do not compute: the quadratic refinement $q$ of the torsion linking
form \cite[Definition 2.32]{crowley02}, the generalised Eells-Kuiper invariant
$\mu$ that can detect different smooth structures \cite[(26)]{7class},
and the diffeomorphism and homotopy invariant $\xi(\varphi)$ of the \gtstr{}
\cite[Definition 6.8]{nu}.
The problem is that these invariants are defined in terms of coboundaries,
and we have not identified any explicit coboundaries of our extra-twisted
connected sums. (The invariant $\nu(\varphi)$ is also defined in terms of
coboundaries, but in this case the analytic formula for $\bar \nu$ above gives
an alternative method of calculation.)
In the case of 2-connected 7-manifolds we have good 
classification results, but they do in general rely on all of the invariants.

\begin{thm}[{\cite[Theorem 1.2 \& 1.3]{7class}}]
\label{thm:7class}
Let $M_1$ and $M_2$ be closed 2-connected 7-manifolds, and let $F : H^4(M_2) \to H^4(M_1)$ be a group isomorphism. Then $F$ is realised as $f^*$ for some
homeomorphism $f : M_1 \to M_2$ if and only if $F(p(M_2)) = p(M_1)$ and $F$
preserves $b$ and $q$.
$F$ is realised as $f^*$ of some diffeomorphism if and only if $F$ is in addition preserves $\mu$.
\end{thm}

\begin{thm}[{\cite[Theorem 6.9]{nu}}]
\label{thm:g2class}
Let $M_1$ and $M_2$ be closed 2-connected 7-manifolds with \gtstr s $\varphi_1$ and $\varphi_2$, and let $F : H^4(M_2) \to H^4(M_1)$ be a group isomorphism.
Then $F$ is realised as $f^*$ for some diffeomorphism $f : M_1 \to M_2$ such
that $f^*\varphi_2$ is homotopic to $\varphi_1$ if and only if
$F(p(M_2)) = p(M_1)$, $\nu(\varphi_1) = \nu(\varphi_2)$ and $F$
preserves $b$, $q$ and $\xi$.
\end{thm}

However, in many examples the invariants $q$, $\mu$ and $\xi$ are redundant.
The quadratic refinement $q$ is uniquely determined by $b$ unless $TH^4(M)$
has 2-torsion.
The Eells-Kuiper invariant is vacuous unless $p(M)$ is divisible by 8 modulo
torsion, and $\xi$ is completely determined by $\mu$ and $\nu$ when the
greatest divisor of $p(M)$ modulo torsion divides 112.
Therefore, even though we have not computed $q$, $\mu$ and $\xi$ we can
still apply the above classification theorems to many of the examples in
\S\ref{sec:examples}.
 
For rectangular twisted connected sums, $q$ and $\mu$ were computed in
\cite{exotic}, and $\xi$ by Wallis \cite{wallis18}.

\section{Examples of extra-twisted connected sums}
\label{sec:examples}

We now combine the preceding results to produce examples of
extra-twisted connected sums. We select 50 convenient examples that illustrate
some interesting phenomena.
All but Example~\ref{ex:intersection} are 2-connected, and their
properties are summarised in Tables \ref{table:rk1pi4xtcs} and
\ref{table:xtcs}. In each case, we describe a configuration of the polarising
lattices in terms of a push-out $W$ as described in Remark \ref{rmk:nikulin},
and deduce from Theorem \ref{thm:matching} that the given configuration
is realised by some $\thet$-matching.

\newcommand{\rkonetable}{
\begin{table}
\[
\renewcommand{\arraystretch}{1.1}
\setlength{\tabcolsep}{8pt}
\begin{array}[b]{llc@{\hspace{15pt}}c@{\hspace{12pt}}c@{\hspace{6pt}}c}
\toprule
\hspace{1ex} Z_+ & \hspace{1ex} Z_- & b_3 & d & TH^4 & b  \\ \midrule
\ref{ex:rk1index2}_4  & \ref{ex:spec1}^1_{16} &  60 & 24 &  4 & \frac14 \\
\ref{ex:p3cover}      & \ref{ex:spec1}^1_{18} &  64 & 24 &  2  \\
\ref{ex:rk1index2}_3  & \ref{ex:spec1}^1_{12} &  68 &  6 &  3 & \frac13 \\
\ref{ex:rk1index2}_2  & \ref{ex:spec1}^1_{18} &  72 & 12 &  2  \\
\ref{ex:rk1index2}_4  & \ref{ex:spec1}^2_{2}  &  74 & 12 &  4 & \frac14  \\
\ref{ex:p3cover}      & \ref{ex:spec1}^1_{8}  &  78 &  4 &  2  \\
\ref{ex:p3cover}      & \ref{ex:spec1}^2_{4}  &  78 & 24 &  2  \\
\ref{ex:p1rsmooth}_1  & \ref{ex:spec1}^1_{16} &  82 &  4 &     \\
\ref{ex:rk1index2}_2  & \ref{ex:spec1}^2_{4}  &  86 &  8 &  2  \\
\ref{ex:rk1index2}_2  & \ref{ex:spec1}^1_{8}  &  86 & 12 &  2  \\
\ref{ex:rk1index2}_1  & \ref{ex:spec1}^1_{16} &  92 &  4 &     \\
\ref{ex:p3cover}      & \ref{ex:spec1}^2_{1}  &  92 &  2 &  2  \\
\ref{ex:p1rsmooth}_1  & \ref{ex:spec1}^2_{2}  &  96 &  2 &     \\
\bottomrule
\end{array}
\qquad
\begin{array}[b]{llc@{\hspace{15pt}}c@{\hspace{12pt}}c@{\hspace{6pt}}c}
\toprule
\hspace{1ex} Z_+ & \hspace{1ex} Z_- & b_3 & d & TH^4 & b  \\ \midrule
\ref{ex:rk1index2}_2  & \ref{ex:spec1}^2_{1}  & 100 &  6 &   2 \\
\ref{ex:rk1index2}_4  & \ref{ex:spec1}^4_{1}  & 102 &  2 &   4 & \frac14  \\
\ref{ex:rk1index2}_4  & \ref{ex:spec1}^1_{4}  & 102 &  4 &   4 & \frac14  \\
\ref{ex:rk1index2}_1  & \ref{ex:spec1}^2_{2}  & 106 &  2 &     \\
\ref{ex:p1rsmooth}_1  & \ref{ex:spec1}^1_{4}  & 124 &  2 &     \\
\ref{ex:p1rsmooth}_1  & \ref{ex:spec1}^4_{1}  & 124 &  8 &     \\
\ref{ex:rk1index2}_1  & \ref{ex:spec1}^1_{4}  & 134 &  6 &     \\
\ref{ex:rk1index2}_1  & \ref{ex:spec1}^4_{1}  & 134 & 24 &     \\
\ref{ex:p3cover}      & \hspace{-1ex} \ref{ex:p1rsmooth}_1  & 148 &  4  &  2 \\
\ref{ex:p3cover}      & \ref{ex:spec1}^1_{2}  & 148 & 12 &  2 \\
\ref{ex:rk1index2}_2  & \ref{ex:spec1}^1_{2}  & 156 &  8 &  2 \\
\ref{ex:rk1index2}_2  & \hspace{-1ex} \ref{ex:p1rsmooth}_1  & 156 &  8  &  2 \\
\\
\bottomrule
\end{array}
\]
\vspace{-2mm}
\caption{Extra-twisted connected sums of rank 1 one blocks, with $\vartheta = \frac{\pi}{4}$}
\vspace{-4mm}
\label{table:rk1pi4xtcs}
\end{table}}

\newcommand{\xtcstable}{
\renewcommand{\thefootnote}{$\dagger$}
\begin{table}
\[
\renewcommand{\arraystretch}{1.4}
\setlength{\tabcolsep}{8pt}
\begin{array}[b]{lrllc@{\hspace{2ex}}c@{\hspace{2ex}}ccr}
\toprule
 \text{Ex} & \vartheta & Z_+ & Z_- & b_3 & d & TH^4 & b & \bar\nu\hspace*{3pt}  \\ \midrule
\ref{ex:pi4main}     & \frac{\pi}{4} & \ref{ex:deg_sextic} & \ref{ex:mm2}_3
& 97 & 2 &   &   & - 36 \\
\ref{ex:77pi4}       & \frac{\pi}{4} & \ref{ex:p1p1smooth} & \ref{ex:mm2}_{10}
& 77 & 2 &   &   &  -36 \\
\ref{ex:diag}        & \frac{\pi}{4} & \ref{ex:DQf}_2 & \ref{ex:mm2}_{10}
& 57 & 2 & 2 {\cdot} 2 & \sm{\frac12 & 0 \\ 0 & \frac12} & -36 \\[1mm]
\ref{ex:hyperb}      & \frac{\pi}{4} & \ref{ex:DQf}_4 &\ref{ex:mm2}_{10}
& 57 & 2 & 2 {\cdot} 2 & \sm{0 & \frac12 \\ \frac12 & 0} & -36 \\
\ref{ex:pi4t7}       & \frac{\pi}{4} & \ref{ex:Dcon}_2 & \ref{ex:con}_2
& 45 & 2 & 7 & -\frac17 & -36 \\
\ref{ex:pi4rk3b}     & \frac{\pi}{4} & \ref{ex:p1rsmooth}_3& \ref{ex:p1cube}
& 98 & 6 & & & -33 \\
\ref{ex:pi4mixed}    & \frac{\pi}{4} & \ref{ex:rk1index2}_1 & \ref{ex:mm2}_{10}
& 91 & 4 &  & & -36 \\
\ref{ex:skew34}      & \frac{\pi}{4} & \ref{ex:p1rsmooth}_2 & \ref{ex:mm2}_{27}
& 92 & 4 & & & -33 \\
\ref{ex:big_angle}   & \frac{\pi}{4} & \ref{ex:DQf}_3 & \ref{ex:mm2}_{17}
& 60 & 6 & & & -33 \\
\ref{ex:small_angle} & \frac{\pi}{4} & \ref{ex:inv2blp}_4 & \ref{ex:mm2}_{17}
& 60 & 6 & & & -39 \\
\ref{ex:t3a}         & \frac{\pi}{4} & \ref{ex:rk1index2}_3 & \ref{ex:Dmm2}_6
& 71 & 6 & 3 & \phantom{-} \frac13 & -36 \\
\ref{ex:t3b}         & -\frac{\pi}{4} & \ref{ex:Dmm2}_6 & \ref{ex:spec1}^2_3 
& 71 & 6 & 3 & \phantom{-} \frac13 & 36 \\
\ref{ex:inertia}     & \frac{\pi}{4} & \ref{ex:Dmm2}_6  & \ref{ex:octic}
& 42 & 4\footnotemark & 8 & & -33 \\
\ref{ex:1+1pi6} & \frac{\pi}{6} & \ref{ex:rk1index2}_4  & \ref{ex:rk1index2}_3
&  54 &  6  &  & & -51  \\
\ref{ex:1+1pi6} & \frac{\pi}{6} & \ref{ex:rk1index2}_3  & \ref{ex:rk1index2}_4
&  54 &  2  &  3  & \phantom{-}\frac13 & -51 \\
\ref{ex:1+1pi6} & \frac{\pi}{6} & \ref{ex:p1rsmooth}_1  & \ref{ex:rk1index2}_3
&  76 &  6  &  & & -51   \\
\ref{ex:1+1pi6} & \frac{\pi}{6} & \ref{ex:rk1index2}_3  & \ref{ex:p1rsmooth}_1
&  76 &  24  & 3 & \phantom{-}\frac13 & -51 \\
\ref{ex:1+1pi6} & \frac{\pi}{6} & \ref{ex:rk1index2}_1 & \ref{ex:rk1index2}_3 &
86 & 6 & & & -51  \\
\ref{ex:1+1pi6} & \frac{\pi}{6} & \ref{ex:rk1index2}_3 & \ref{ex:rk1index2}_1 &
86 & 4 & 3 & \phantom{-}\frac13 & -51 \\
\ref{ex:pi6main} & \frac{\pi}{6} & \ref{ex:deg_sextic} & \ref{ex:deg_sextic}
& 109 & 2 & & & -48  \\
\ref{ex:pi6main8} & \frac{\pi}{6} & \ref{ex:deg_sextic} & \ref{ex:deg_sextic}
& 109 & 8 & & & -48   \\
\ref{ex:pi6deg5} & \frac{\pi}{6} & \ref{ex:deg_sextic} & \ref{ex:DQf}_5
& 77 & 2 & & & -48   \\
\ref{ex:pi6deg5flop} & \frac{\pi}{6} & \ref{ex:deg_sextic} & \ref{ex:Dcon}_5
& 77 & 4 &  & & -48  \\
\ref{ex:pi6t7} & \frac{\pi}{6} & \ref{ex:Dcon}_2 & \ref{ex:spec1}^1_6
& 45 & 2 & 7 & -\frac17 & -48 \\
\bottomrule
\end{array}
\]
\vspace{-2mm}
\caption{Examples of 2-connected extra-twisted connected sums} 
\label{table:xtcs}
\end{table}
}

\subsection{Matchings with pure angle \fpif}

We begin by considering \fpif-extra twisted connected sums, using configurations where the polarising lattices are at ``pure angle'' \fpif{} as discussed in
\S\ref{subsec:sufficient}, so that Theorem \ref{thm:matching} can be
applied to produce matchings without using any genericity results beyond
Proposition \ref{prop:generic_fano}.
The topology is also easy to compute using Corollary \ref{cor:pi4pure}.

Matchings among rank 1 blocks are relatively easy to study systematically.
 We have listed
7 involution blocks of rank 1 (Examples \ref{ex:p3cover},
\ref{ex:rk1index2}$_1$, \ldots \ref{ex:rk1index2}$_5$ and
\ref{ex:p1rsmooth}$_1$), and 18 ordinary rank 1 blocks ($17$ in
Example \ref{ex:spec1}, and one in Example \ref{ex:p1rsmooth}).

If the squares of the generators $x_+$ and $x_-$ of the polarising lattices of
the building blocks are $n_+$ and $n_-$ respectively, then as in
\eqref{eq:sqrt} the necessary and sufficient condition for the existence of a
matching is that $2n_+n_-$ be a square. A simple computer script identifies that
the condition is satisfied for 25 of the 119 pairs of blocks, and computes
the topological invariants from the data in Tables \ref{table:species1} and
\ref{table:invblocks} as follows.

\rkonetable

When the condition holds, we can uniquely write $n_+ = 2mq_+^2$ and $n_- =
mq_-^2$, for $q_+$ and $q_-$ coprime, and define the configuration by
\[ W = \begin{pmatrix} 2mq_+^2 & mq_+q_- \\ mq_+q_- & mq_-^2 \end{pmatrix} . \]
We can now apply Corollary \ref{cor:pi4pure} to compute the topological invariants. We find that $\pi_+$ maps $x_-$ to $\frac{q_-}{2q_+}x_+$,
so $N_+ + 2\pi_+N_-$ is generated by $\frac{1}{q_+}x_+$, which has square $2m$.
Therefore $\Tor H^4(M) \cong \cg{m}$.
Meanwhile $\pi_- N_+ \cap N_-$ is generated by $q_+x_-$, so the
greatest divisor of $p(M)$ modulo torsion is
$(\pi_+^*\bar c_+ + \bar c_-)(q_+x_-)
= \frac{q_-}{2}\bar c_+(x_+) + q_+\bar c_-(x_-)$.

In those cases where the order $m$ of $\Tor H^4(M)$ divides the greatest
divisor of $p(M)$ modulo torsion, the above computation does not suffice to
determine $p(M)$ up to isomorphisms of $H^4(M)$. However, in all cases it turns
out that the greatest divisor of $p(M)$ equals the greatest divisor of $p(M)$
modulo torsion; then it is possible to choose the isomorphism $H^4(M) \cong \ZZ^{b_3(M)} \times \cg{m}$ so that the image of $p(M)$ has no $\cg{m}$ component.
When $m = 2$ there is nothing to check, since $p(M)$ is even
a priori for any spin 7-manifold according to Lemma \ref{lem:p}.
In the remaining 4 cases, we find that $\half \bar c_+$ and $\bar c_-$ are both
divisible by $m$, so $p(M)$ is too.

Finally, $b_3(M)$ is simply $22 + b_3^+(Z_+) + b_3(Z_-)$ by \eqref{eq:b3}.
This is even, so cannot coincide with $b_3$ of any 2-connected ordinary TCS.

These topological invariants of the 25 \fpif-matchings of rank 1 blocks are
summarised in Table \ref{table:rk1pi4xtcs}, listing $b_3(M)$, the greatest divisor $d$ of $p(M)$ and the order of $TH^4(M)$. We also list the self-linking of
a generator of $TH^4(M)$ when it is not vacuous (\ie when the order of the
cyclic group $TH^4(M)$ is greater than 2). We have not included the
$\bar\nu$-invariant in the table, since it is the same in all cases:
for a $\frac{\pi}{4}$-matching of rank 1 blocks, the only possibility for the
configuration angles is that $\alpha^-_1 = \cdots = \alpha^-_{19} = 0$,
so Theorem \ref{thm:nubar} gives $\bar\nu = -39$.

\bigskip
We now give 5 examples of pure angle \fpif-matchings of blocks
of rank 2. In each case we define the desired configuration by writing down
a symmetric $4 \times 4$ matrix $W$, where the diagonal $2 \times 2$ blocks
are the polarising lattices $N_+$ and $N_-$ of the two building blocks,
and the off-diagonal blocks are chosen to ensure that
$N_\pm^\frac{\pi}{4} = N_\pm$; this can be verified by checking that
$\pi_+(x).\pi_+(y) = \half x.y$ for any $x, y \in N_+$.
By using bases for $N_+$ and $N_-$ that consist of edges of the respective
ample cones (\ie the bases used in Tables \ref{table:ordblocks} and
\ref{table:invblocks}), verifying hypothesis \eqref{eq:amps}
of Theorem \ref{thm:matching} becomes a simple matter of checking that
some element in the positive quadrant of $N_+$ is mapped to the positive
quadrant of $N_-$ by $\pi_+$ (or vice versa). 
Theorem \ref{thm:matching} then produces a matching with the desired
configuration. The resulting \fpif-twisted connected sum $M$ has
$b_3(M) = 21 + b_3^+(Z_+) + b_3(Z_-)$ by \eqref{eq:b3}, and the main
remaining topological invariants are easily computed using Corollary
\ref{cor:pi4pure}.

For a pure angle \fpif{} configuration of rank 2 blocks, two of the
configuration
angles take the values $\frac{\pi}{2}$ and $-\frac{\pi}{2}$ while the remaining
17 configuration angles are 0. Hence Theorem \ref{thm:nubar} gives
$\bar\nu = -36$.

We collect the data of these and all remaining 2-connected examples in
Table \ref{table:xtcs}. We list for each example the gluing angle,
the blocks used, $b_3(M)$, the greatest divisor $d$ of $p(M)$ in $H^4(M;\Z)$
(which for all  examples except \ref{ex:inertia} is the same as the greatest
divisor modulo torsion),
the order of the torsion subgroup $TH^4(M)$,
a description of the torsion linking form $b$, and $\bar \nu$.
When the torsion $TH^4(M)$ is cyclic we describe the linking form by giving
the self-linking of a generator. The only examples of non-cyclic $TH^4(M)$
are $(\Z/2)^2$, where the possibilities for the linking form are that it is
diagonalisable $\sm{\frac12 & 0 \\ 0 & \frac12}$ or hyperbolic
$\sm{0 & \frac12 \\ \frac12 & 0}$.

\begin{ex}
\label{ex:pi4main}
We match the involution block from Example \ref{ex:deg_sextic} (from one-point
blow-up of degree 1 del Pezzo 3-fold) and the regular block from Example
\ref{ex:mm2}$_3$ (from degree 1 del Pezzo 3-fold blown up in an elliptic curve)
at pure angle $\frac{\pi}{4}$. The polarising lattices are $N_+ = \sm{2 & 2 \\ 2 & 0}$ and $N_- = \sm{4 & 2 \\ 2 & 0}$, and we define the configuration
using the matrix
\[ W = \begin{pmatrix} 2 & 2 & 2 & 1 \\ 2 & 0 & 2 & 0
\\ 2 & 2 & 4 & 2 \\ 1 & 0 & 2 & 0 \end{pmatrix} . \]
Actually, because Example \ref{ex:deg_sextic} is not a semi-Fano block,
Proposition \ref{prop:generic_fano} does not provide the genericity result
needed for Theorem \ref{thm:matching} to produce matchings; the required
genericity result is instead Lemma \ref{lem:deg_hyper}.

The resulting \fpif-twisted connected sum $M$ is 2-connected,
with $b_3(M) = 21 + 44 + 32 = 97$.
Because $N_+$ has 2-elementary discriminant,
it is immediate from Corollary \ref{cor:pi4pure} that $H^4(M)$ is torsion-free.
In the respective bases for $N_\pm^*$, we have $\bar c_+ = \rvec{26}{24}$ and
$\bar c_- = \rvec{20}{12}$, while $\pi_+ :N_- \to N_+$ is represented by
$\sm{1 & 0 \\ 0 & \frac{1}{2}}$.
In the basis for $N_-^*$ we thus get
$\pi_+^*\bar c_+ = \rvec{26}{24}\sm{1 & 0 \\ 0 & \frac{1}{2}} = \rvec{26}{12}$,
and
\[ \pi_+^* \bar c_+ + \bar c_- = \rvec{26+20}{12+12} = \rvec{46}{24}, \]
so $p(M)$ has greatest divisor 2 by Corollary \ref{cor:pi4pure}.

By Theorem \ref{thm:7class}, there is a unique diffeomorphism class
of 2-connected 7-manifolds $M$ with $b_3(M) = 97$, torsion-free $H^4(M)$
and $d =2$. According to \cite[Table 3]{g2m}, there are
two different rectangular twisted connected sums of rank 1 Fano blocks with
these invariants, so yield further torsion-free \gtstr s on the same manifold.
However, the \fpif-twisted connected sum has $\nu = 36$ while the rectangular
twisted connected sums have $\nu = 24$, so the \gtstr s cannot be homotopic.
In particular, the moduli space of holonomy $G_2$ metrics on this manifold
is disconnected.
\end{ex}

\begin{ex}
\label{ex:77pi4}
Match Example \ref{ex:p1p1smooth} (from K3 with non-symplectic involution
that is a branched double cover of $\PP^1 \times \PP^1$) and
Example \ref{ex:mm2}$_{10}$ (from blow-up of complete intersection of two
quadrics in an elliptic curve) using the configuration defined by
\[ W = \begin{pmatrix}
0 & 2 & 4 & 0 \\ 
2 & 0 & 1 & 1\\ 
2 & 2 & 8 & 4 \\ 
2 & 0 & 4 & 0\end{pmatrix}
\]
Now $b_3(M) = 21 + 32 + 24 = 77$. Corollary \ref{cor:pi4pure} gives that
$H^4(M)$ is torsion-free. Also, $\pi_+^* \bar c_+ + \bar c_- =
\rvec{12}{12}\sm{\frac12 & \frac 12 \\ 2 & 0} + \rvec{28}{12}
= \rvec{30+28}{6+12} = \rvec{58}{18}$, whose greatest divisor is 2.

These are the same invariants as Example \ref{ex:pi6deg5}. Moreover,
according to \cite[Table 3]{g2m} there is also a rectangular twisted
connected sum of rank 1 Fano-type blocks (namely Examples
\ref{ex:spec1}$^1_{12}$ and \ref{ex:spec1}$^1_{14}$) with these invariants.
Thus the smooth 2-connected 7-manifold $M$ with $b_3(M) = 77$, torsion-free
$H^4(M)$ and $d = 2$ admits torsion-free \gtstr s with
$\bar \nu(\varphi) = -36, -48$ and $0$, so its moduli space of holonomy $G_2$
metrics has at least 3 components.
\end{ex}

\begin{ex}
\label{ex:diag}
Match Examples \ref{ex:DQf}$_2$ (from double cover of
quadric-fibred degree 2 semi del Pezzo 3-fold) and \ref{ex:mm2}$_{10}$
using the configuration defined by
\[ W = \begin{pmatrix}
4 & 4 & 4 & 2 \\ 
4 & 0 & 4 & 0\\ 
4 & 4 & 8 & 4 \\ 
2 & 0 & 4 & 0
\end{pmatrix}. \]

$b_3(M) = 21 + 12 + 24 = 57$.

To use Corollary \ref{cor:pi4pure} to compute $TH^4(M)$, note that $2\pi_+N_-$
is contained in $N_+$, so $N_+ + 2\pi_+N_- = N_+$. The discriminant
is a diagonal $\Delta \cong (\Z/4)^2$, so $TH^4(M) \cong (\Z/2)^2$
with diagonal linking form.

$\pi_+^* \bar c_+ + \bar c_- = \rvec{28}{12} \sm{1 & 0 \\ 0 & \frac12} + 
\rvec{28}{12} = \rvec{56}{18}$, so Corollary \ref{cor:pi4pure} implies that
the greatest divisor of $p(M)$ modulo torsion is 2. Since there is only
2-torsion, and $p(M)$ is even a priori, $p(M)$ cannot have any interesting
torsion component.
\end{ex}

\begin{ex}
\label{ex:hyperb}
Matching Examples \ref{ex:DQf}$_4$ (from double cover of
quadric-fibred degree 4 semi del Pezzo 3-fold) and \ref{ex:mm2}$_{10}$
using
\[ W = \begin{pmatrix}
8 & 4 & 6 & 4 \\ 
4 & 0 & 2 & 0\\ 
6 & 2 & 8 & 4 \\ 
4 & 0 & 4 & 0
\end{pmatrix}. \]

The calculations are very similar to the previous example. We again find
$b_3(M) = 21 + 12 + 24 = 57$. However,
this time the discriminant form on $\Delta \cong (\Z/4)^2$ is hyperbolic,
so although $TH^4(M) \cong (\Z/2)^2$ again, 
the torsion linking form is hyperbolic is hyperbolic in this example.

$\pi_+^* \bar c_+ + \bar c_- = \rvec{32}{12} \sm{\frac12 & 0 \\ \frac12 & 1} + 
\rvec{28}{12} = \rvec{46}{24}$, so Corollary \ref{cor:pi4pure} implies that
the greatest divisor of $p(M)$ modulo torsion is 2.
Again $p(M)$ cannot have any interesting torsion component.

Thus this example is distinguished from Example \ref{ex:diag} only by
the torsion linking form.
\end{ex}

\begin{ex}
\label{ex:pi4t7}
Match Examples \ref{ex:Dcon}$_2$ (from double cover of conic-fibred
degree 2 del Pezzo 3-fold) and \ref{ex:con}$_2$ (ordinary block from the
conic-fibred degree 2 del Pezzo 3-fold itself) using
\[ W = \begin{pmatrix}
4 & 6 & 6 & 2 \\ 
6 & 2 & 2 & 3\\ 
6 & 2 & 4 & 6 \\ 
2 & 3 & 6 & 2
\end{pmatrix}. \]
$b_3 = 21 + 6 + 18 = 45$. $N_+ + 2\pi_+N_- = N_+$, whose discriminant group
$\Delta \cong \Z/14 \times \Z/2$.
Thus $TH^4(M) \cong \Delta/T_2\Delta \cong \Z/7$, and the image of
$\alpha := \rvec{1}{0} \in \Delta$ is a generator of $TH^4(M)$.
Now $b_\Delta(\alpha, \alpha) =
\rvec{1}{0} \sm{4 & 6 \\ 6 & 2}^{-1}\cvec{1}{0} = -\frac{1}{14}$,
so the image in $TH^4(M)$ has self-linking $-\frac17$.

$\pi_+^*\bar c_+ + \bar c_- =
 \rvec{28}{18}\sm{0 & \frac12 \\ 1 & 0} + \rvec{20}{18} = \rvec{38}{32}$,
so $d = 2$.
\end{ex}

Finally, here is a rank 3 matching.

\begin{ex}
\label{ex:pi4rk3b}
Use involution block from Example \ref{ex:p1rsmooth}$_3$ and ordinary block
from Example \ref{ex:p1cube}. Match using
\[ W = \begin{pmatrix}
2 &  2 &  2 & 1 &  1 & 2 \\
2 &  0 &  2 & 1 &  1 & 0 \\
2 &  2 &  0 & 0 &  2 & 2 \\
1 &  1 &  0 & 0 &  2 & 2 \\
1 &  1 &  2 & 2 &  0 & 2 \\
2 &  0 &  2 & 2 &  2 & 0
\end{pmatrix} \]
$b_3(M) = 20 + 28 + 50 = 98$. Since $N_+$ is 2-elementary,
$H^4(M)$ is torsion-free.

$\pi_+^*\bar c_+ + \bar c_- = \sm{18 & 12 & 12}\sm{0 & 1 & 0 \\ 0 & 0 &1 \\ \frac12 & -\frac12 & 0} + \sm{12 & 12 & 12} = \sm{18 & 24 & 24}$, so $d = 6$.

For any pure \fpif{} matching of rank 3 blocks, exactly 2 each of the
configuration angles $\alpha^-_1, \ldots, \alpha^-_{19}$ are $\frac{\pi}{2}$
and $-\frac{\pi}{2}$ while the other 15 are 0. Thus Theorem \ref{thm:nubar}
gives $\bar \nu = -33$.
\end{ex}

\subsection{Other \fpif-matchings}

We now consider 8 examples of \fpif-extra twisted connected sums where the
configuration does not have pure angle \fpif{} (including one that is not
2-connected). This involves carrying out some
extra work for each example. In addition to checking hypothesis \eqref{eq:amps} in Theorem \ref{thm:matching}, we also need to compute $\Lambda_\pm$
as in \eqref{eq:Lambda}, and verify that the families are $\Lambda_\pm$-generic
(most of the work for the last step has already been carried out in
\S\ref{sec:genericity}).

Moreover, we cannot use Corollary \ref{cor:pi4pure} to compute the topology,
but instead have to apply the more cumbersome Proposition \ref{prop:pi4detail}.
However, we can speed up the required calculation of $\coker \wh W$ a little with the following observation: if $A_+ \in N_+^\frac{\pi}{4}$ and $A_- = \pi_- A_+ \in N_-^\frac{\pi}{4}$, then the image of $\wh W$ is contained in the
kernel of the homomorphism $(A_+,A_-) : N_+^* \oplus N_-^* \to \Z$.

\begin{ex}
\label{ex:pi4mixed}
Match the involution block from 
Example \ref{ex:rk1index2}$_1$ and the regular block from Example \ref{ex:mm2}$_{10}$
at angle $\frac{\pi}{4}$ using the matrix
\[ W = \begin{pmatrix} 2 & 3 & 1 \\ 3 & 8 & 4 \\ 1 & 4 & 0 \end{pmatrix} . \]
$\pi_-$ maps the positive generator $H_+ \in N_+$
to $\quart A_-$, for  $A_- := \cvec{1}{1} \in N_-$.
This is in indeed in the ample cone of
the family of Example \ref{ex:mm2}$_{10}$, so \eqref{eq:amps} holds.

Now $\Lambda_- = N_-$, so for the family of Example \ref{ex:mm2}$_{10}$
we do not need any genericity result beyond Proposition \ref{prop:generic_fano}.
On the other hand, $\Lambda_+$ is generated by $N_+$ and the orthogonal
complement of $A_-$ in $N_-$, so
\[ \Lambda \cong \begin{pmatrix} 2 & 0 \\ 0 & -16 \end{pmatrix} . \]
In particular there are no $(-2)$-classes orthogonal to the degree 2
class $H_+$. Therefore Proposition~\ref{prop:hyper} implies that
the family of blocks from Example \ref{ex:rk1index2}$_1$
is $(\Lambda_+, H_+\bbrp)$-generic, so we can apply Theorem \ref{thm:matching}
to find a matching with this configuration.

The resulting extra-twisted
connected sum $M$ is 2-connected, and
\eqref{eq:b3} gives $b_3(M) = 23 - 1 - 2 + 46 + 24 + 1 = 91$.
Proposition \ref{prop:pi4detail} shows that the torsion is isomorphic
to the cotorsion of the image of the matrix
\[ \wh W = \begin{pmatrix} 1 & 3 & 1 \\ 3 & 8 & 4 \\ 1 & 4 & 0 \end{pmatrix} . \]
Its image is exactly the kernel of $\sm{4 & -1 & -1}$,
so the torsion is in fact trivial.

Since $\bar c_+ = 26$ while
$\bar c_- = \rvec{12}{4}$, Proposition \ref{prop:pi4detail}
further gives the greatest divisor of $p(M)$
in terms of the greatest divisor of $\sm{13 & -12 & -4\!}$
modulo $\im \wh W$; since
$\sm{4 & -1 & -1} \cdot \sm{13 & -12 & -4\!} = 68$, the greatest divisor of
$p(M)$ is $\gcd(24,68) = 4$.

Only one of the configuration angles $\alpha^-_1, \ldots, \alpha^-_{19}$ is
non-zero, and takes the value $\pi$. Hence Theorem \ref{thm:nubar} gives
$\bar \nu = -36$.

According to \cite[Table 4]{exotic}, there are two 
rectangular twisted connected sums from Fanos of rank 1 or 2,
with the same diffeomorphism invariants.
\end{ex}

\begin{ex}
\label{ex:skew34}
Match Examples \ref{ex:p1rsmooth}$_2$ (from K3 with non-symplectic involution
branched over one-point blow-up of $\PP^2$) and \ref{ex:mm2}$_{27}$ (from
$\PP^3$ blown up in a twisted cubic) using
\[ W = \begin{pmatrix}
2 & 2 & 2 & 3 \\ 
2 & 0 & 1 & 1\\ 
2 & 1 & 2 & 5 \\ 
3 & 1 & 5 & 4
\end{pmatrix}. \]
Let $A_+ := \cvec{2}{3} \in N_+$, and $A_- = \cvec{1}{1} \in N_-$.
Then $A_+^2 = 32$ and $A_-^2 = 16$, and $\pi_- A_+ = A_-$ and
$\pi_+ A_- = \half A_+$. Thus $A_\pm \in N_\pm^\frac{\pi}{4}$, so
\eqref{eq:amps} is satisfied.

The orthogonal complements of $A_\pm$ in $N_\pm$ are spanned by
$B_\pm$ for $B_+ := \cvec{2}{-5}$ and $B_- := \cvec{9}{-7}$.
$\Lambda_\pm$ is spanned by $N_\pm$ and $B_\mp$, so
\[
\Lambda_+ = \begin{pmatrix}
2 & 2 & -3 \\
2 & 0 & 2 \\
-3 & 2 & -272
\end{pmatrix}, \quad
\Lambda_- = \begin{pmatrix}
2 & 5 & -1 \\
5 & 4 & 1 \\
-1 & 1 & -32
\end{pmatrix}
\]
Then Proposition \ref{prop:smoothing_generic} and Lemma \ref{lem:twisted}
give the genericity results needed for Theorem \ref{thm:matching}
to yield a matching.

$b_3(M) = 23 - 2 -2 + 32 + 40 + 1 = 92$.
By Proposition \ref{prop:pi4detail}, $TH^4(M)$ is isomorphic to the cotorsion of
\[ \wh W = \begin{pmatrix}
1 & 1 & 2 & 3 \\ 
1 & 0 & 1 & 1\\ 
2 & 1 & 2 & 5 \\ 
3 & 1 & 5 & 4
\end{pmatrix}, \]
which is trivial. Indeed, $\coker \wh W$ is mapped isomorphically to $\Z$ by
$\sm{2 & 3 & -1 & -1}$. This maps
$(\half \bar c_+, -\bar c_-) = \sm{9 & 6 & -18 & -22}$ to $76$, so $d = 4$.

To compute $\bar \nu$ we need to determine the configuration angles.
Note that $\pi_+ B_- = \half B_+$, whose square is $\frac{1}{34}$ of the square
of $B_-$. So $B_\pm$ is in the $\frac{1}{34}$-eigenspace of $\pi_\pm \pi_\mp$.
By \eqref{eq:mult_angle}, two of the configuration angles
are $\pm 2\psi$ where $(\cos \psi)^2 = \frac{1}{34}$, and the other 17
configuration angles are 0.
Because $2\psi$ is in the interval $(\frac{\pi}{2}, \pi)$,
Theorem \ref{thm:nubar} gives $\bar \nu = -33$.

The diffeomorphism classifying invariants coincide with those of the
extra-twisted connected sum of Examples \ref{ex:rk1index2}$^1_2$ and
\ref{ex:spec1}$^1_{16}$ in line 11 of Table \ref{table:rk1pi4xtcs}, but
the $\bar\nu$-invariants differ.
\end{ex}

The next two examples illustrate the dependence of $\bar\nu$ on the
configuration angles.

\begin{ex}
\label{ex:big_angle}
Matching of Examples \ref{ex:DQf}$_3$ (from double cover of a quadric-fibred
degree 3 semi del Pezzo, or equivalently a double cover of a small resolution
of cubic 3-fold containing a plane) and \ref{ex:mm2}$_{17}$ (from the blow-up
of a quadric 3-fold in an elliptic curve of degree 5), using
\[ W = \begin{pmatrix} 6 & 4 & 4 & 5 \\ 4 & 0 & 2 & 2 \\
4 & 2 & 4 & 7 \\ 5 & 2 & 7 & 6 \end{pmatrix} . \]
The ample class $A_+ = \cvec{4}{3} \in N_+$ (of square 192) is mapped by
$\pi_-$ to $A_- = \cvec{2}{2} \in N_-$ (of square~96),
while $\pi_+A_- = \half A_+$.
Therefore $A_\pm \in N_\pm^\frac{\pi}{4}$, so \eqref{eq:amps} is satisfied.

The orthogonal complement of $A_\pm$ in $N_\pm$ is spanned by $B_\pm$
for $B_+ = \cvec{4}{-9}$ and $B_- = \cvec{13}{-11}$, of square $-192$ and
$-600$ respectively.
\[ \Lambda_+ =
\begin{pmatrix} 6 & 4 & 3 \\ 4 & 0 & -4 \\ 3 & -4 & -600 \end{pmatrix}, \quad
 \Lambda_- =
\begin{pmatrix} 4 & 7 & 4 \\ 7 & 6 & -2 \\ 4 & -2 & -192 \end{pmatrix}  \]
Proposition \ref{prop:deg_cubics} and Lemma \ref{lem:deg5generic} provide the
genericity results needed for Theorem \ref{thm:matching} to yield
matchings.

$b_3(M) = 23 -2 -2 + 12 + 28 + 1 = 60$. The cokernel of
\[ \wh W = \begin{pmatrix} 3 & 2 & 4 & 5 \\ 2 & 0 & 2 & 2 \\
4 & 2 & 4 & 7 \\ 5 & 2 & 7 & 6 \end{pmatrix} . \]
is mapped isomorphically to $\Z$ by $\sm{4 & 3 & -2 & -2}$, so $H^4(M)$ is
torsion-free. $(\half \bar c_+, \bar c_-) = \sm{15 & 6 & -22 & -26}$
is mapped to 174, so $d = \gcd(174, 24) = 6$.

$\pi_+ B_- = \quart B_+$, whose square is $-12$. Therefore $B_\pm$ are
$\pi_\mp\pi_\mp$-eigenvectors with eigenvalue $\frac{1}{50}$.
Then the non-zero configuration angles are $\pm 2\psi$ for $(\cos \psi)^2 = \frac{1}{50}$. Because $\psi \in (\frac{\pi}{2}, \pi)$, Theorem \ref{thm:nubar} gives $\bar \nu = -33$.
\end{ex}

\begin{ex}
\label{ex:small_angle}
Match Examples \ref{ex:inv2blp}$_4$ (from double cover of one-point blow-up
of a complete intersection of two quadrics, or equivalently a flop of the small
resolution of a cubic 3-fold containing a plane that was used in the previous
example) and \ref{ex:mm2}$_{17}$, using
\[ W = \begin{pmatrix} 8 & 8 & 4 & 6 \\ 8 & 6 & 5 & 4 \\
4 & 5 & 4 & 7 \\ 6 & 4 & 7 & 6 \end{pmatrix} . \]
The ample class $A_+ = \cvec{3}{2} \in N_+$ (of norm 192) is mapped by $\pi_-$
to $A_- = \cvec{2}{2} \in N_-$ (of norm 96), while $A_-$ is mapped by $\pi_+$
to $\half A_+$. So $A_\pm \in N^{\frac{\pi}{4}}_\pm$.
The orthogonal complements are spanned by $B_+ = \cvec{-9}{10} \in N_+$
and $B_- = \cvec{13}{-11} \in N_-$, of square $-192$ and $-600$ respectively.
\[ \Lambda_+ =
\begin{pmatrix} 8 & 8 & 14 \\ 8 & 6 & -21 \\ 14 & -21 & -600 \end{pmatrix}, \quad
 \Lambda_- =
\begin{pmatrix} 4 & 7 & 14 \\ 7 & 6 & -14 \\ 14 & -14 & -192 \end{pmatrix}  \]
Proposition \ref{prop:deg_cubics} and Lemma \ref{lem:deg5generic} provide the
genericity results needed for Theorem \ref{thm:matching} to yield
matchings.

$b_3(M) = 60$ just as in the previous example. Also, we find again that
$H^4(M)$ is torsion-free, and that $d = 6$, so the classifying diffeomorphism
invariants all agree.

However, $\pi_+ B_- = \frac{7}{4}B_+$, whose square is $-588$. Therefore the non-trivial configuration angles $\pm 2\psi$ are in this case given by
$(\cos \psi)^2 = \frac{49}{50}$. Since $2\psi < \frac{\pi}{2}$, Theorem \ref{thm:nubar} yields $\bar \nu = -39$.
\end{ex}

The next two examples of $\frac{\pi}{4}$-twisted connected sums are related by
an orientation-reversing diffeomorphism.  As the underlying manifold has
$TH^4 = \Z/3$, it does not admit an orientation reversing self-diffeomorphism,
and components of its $G_2$ moduli space can be distinguished by the sign of
$\bar \nu$.

\begin{ex}
\label{ex:t3a}
Match Example \ref{ex:rk1index2}$_3$ (from double cover of cubic hypersurface)
with Example \ref{ex:Dmm2}$_6$ (from double cover of (1,1)-divisor).
The polarising
lattices are $N_+ = (6)$ and $N_- = \sm{2 & 4 \\ 4 & 2}$, and we use
the configuration defined by
\[ W = \begin{pmatrix} 6 & 3 & 3 \\ 3 & 2 & 4 \\ 3 & 4 & 2 \end{pmatrix} . \]
If $H_+$ is the generator of $N_+$ and $A_- := \cvec{1}{1} \in N_-$ then
$\pi_+A_- = H_+$ and $\pi_-H_+ = \half A_-$, so
$N_+ = N_+^\frac{\pi}{4}$ and $A_- \in N_-^\frac{\pi}{4}$.
Thus condition \eqref{eq:amps} holds.

The orthogonal complement of $A_-$ in $N_-$
is generated by $B_- = \cvec{1}{-1}$, and
\[ \Lambda_+ = N_+ \oplus B_-\Z \cong
\begin{pmatrix} 6 & 0 \\ 0 & -12 \end{pmatrix}. \]
The family of blocks from Example \ref{ex:spec1}$_2$ is
$(\Lambda_+, H_+\bbrp)$-generic by Proposition \ref{prop:cubics}, so
Theorem \ref{thm:matching} yields matchings with the given configuration.

$b_3(M) = 23 - 1 -2 +18 + 32 + 1  = 71$.
By Proposition \ref{prop:pi4detail},
$\delta(H^3(T^2 \times \kd))$ is isomorphic to the cokernel of
\[ \wh W =
\begin{pmatrix} 3 & 3 & 3 \\ 3 & 2 & 4 \\ 3 & 4 & 2 \end{pmatrix} . \]
The image of $\wh W$ is an index 3 sublattice of the kernel
of $\sm{2 & -1 & -1} : \Z^3 \to \Z$, so $TH^4(M) \cong \Z/3$.
The cotorsion of $\wh W$ is generated by $\sm{1 \\ 1 \\ 1}$.
Its preimage under $\wh W$ is $\frac{1}{3} \sm{1 \\ 0 \\ 0}$,
so by Proposition \ref{prop:pi4detail} the corresponding generator of $TH^4(M)$
has self linking $\frac{1}{3}$.

The image of $(\half \bar c_+, \bar c_-) = \sm{15 & -18 & -18}$ in $\Z$
is $56$, so the greatest divisor of $p(M)$ modulo torsion is $\gcd(66,24) = 6$.
Since this is not coprime to the order of the torsion subgroup, we also need
to check the divisibility of $p(M)$ itself to determine the isomorphism class
of the pair $(H^4(M), p(M))$. But the image of $\sm{15 & -18 & -18}$ in
$\coker \wh W$ is divisible by 6 too, so we can choose an isomorphism
$H^4(M) \cong \Z^{71} \times \Z/3$ such that the image of $p(M)$ has no $\Z/3$
component.

We find $\bar\nu = -36$ like in Example \ref{ex:pi4mixed}.
\end{ex}

\begin{ex}
\label{ex:t3b}
Match Example \ref{ex:Dmm2}$_6$ (from double cover of (1,1)-divisor) with
Example \ref{ex:spec1}$^2_3$ (from cubic 3-fold in $\PP^4$).
The polarising lattices are the same as in the previous example, except that
the roles of $N_+$ and $N_-$ have been swapped, so we can use
essentially the same $W$ as above to define the configuration.
The justification for existence of matching is then just the same, and
$\bar \nu = -36$ by the same calculation as before.

However, the topological computations are different from the previous example,
even though most of the final values turn out to be the same.
This time $b_3(M)$ is computed by $23 - 1 -2 +14 + 36 + 1  = 71$,
while $TH^4(M)$ etc is controlled by
\[ \wh W =
\begin{pmatrix} 1 & 2 & 3 \\ 2 & 1 & 3 \\ 3 & 3 & 6 \end{pmatrix} . \]
The image of $\wh W$ is an index 3 sublattice of the kernel
of $\sm{1 & 1 & -1} : \Z^3 \to \Z$, so $TH^4(M) \cong \Z/3$.
The cotorsion of $\wh W$ is generated by $\sm{1 \\ 1 \\ 2}$.
Its preimage under $\wh W$ is $\frac{1}{3} \sm{0 & 0 & 1}$,
so by Proposition \ref{prop:pi4detail} the corresponding generator of $TH^4(M)$
has self linking $\frac{2}{3}$.

$(\half \bar c_+, -\bar c_-) = \sm{9 & 9 & -24}$, which is divisible by 6
modulo the image of $\wh W$. Thus $p(M)$ is divisible by 6.
The image in the free part of the cokernel is $9 + 9 + 24 = 42$, so the
greatest divisor of $p(M)$ modulo torsion is 6 too.

Since the torsion-linking form is different from Example \ref{ex:t3a},
there is no orientation-preserving diffeomorphism between these \fpif-twisted
connected sums. However, if we reverse the orientation of one, then the sign of
the torsion linking form changes (as does $\bar \nu$) while the other
invariants stay the same, so there does exist an orientation-reversing
diffeomorphism.
\end{ex}

\begin{rmk}
Recalling from \S\ref{subsec:angles} that changing the sign of the gluing angle
corresponds to reversing orientation, we could rephrase this as:
If we use the configuration in this example to construct a
$(-\frac{\pi}{4})$-twisted connected sum, then that is oriented-diffeomorphic
to the \fpif-twisted connected sum from Example \ref{ex:t3a}. 
However, the $(-\frac{\pi}{4})$-twisted connected sum has $\bar \nu = 36$,
so the two components of the $G_2$ moduli space are distinguished.
To emphasise this point, the entry in Table \ref{table:xtcs} for %
Example \ref{ex:t3b} lists the $(-\frac{\pi}{4})$-twisted connected sum. 
\end{rmk}

\begin{ex}
\label{ex:inertia}
Match Example \ref{ex:Dmm2}$_6$ with \ref{ex:octic} using
\[ W = \begin{pmatrix} 2 & 4 & 4 & 2 \\ 4 & 2 & 4 & 2 \\
4 & 4 & 8 & 8 \\ 2 & 2 & 8 & 0 \end{pmatrix} . \]
If we set $A_+ = \cvec{1}{1} \in N_+$ and $A_- = \cvec{1}{1} \in N_-$, then
$\pi_+A_- = A_+$ and $\pi_- A_+ = \half A_-$.
So $A_\pm \in N_\pm^{\frac{\pi}{4}}$, and condition \eqref{eq:amps} is
satisfied.
The orthogonal complements are generated by $B_+ = \cvec{1}{-1} \in N_+$
and $B_- = \cvec{1}{-2}$ respectively. In fact $B_\pm$ is also orthogonal
to $N_\mp$, and
\[ \Lambda_+ = \begin{pmatrix} 2 & 4 & 0 \\ 4 & 2 & 0 \\ 0 & 0 & -32 \end{pmatrix}, \quad 
 \Lambda_- =
\begin{pmatrix} 8 & 8 & 0 \\ 8 & 0 & 0 \\ 0 & 0 & -12 \end{pmatrix} . \]
Proposition \ref{prop:11div} and Lemma \ref{lem:octic_elliptic} provide
the genericity results needed for Theorem \ref{thm:matching} to produce
matchings with the given configuration.

$b_3(M) = 23-2-2+14+8+1 = 42$.
The cokernel of 
\[ \wh W = \begin{pmatrix} 1 & 2 & 4 & 2 \\ 2 & 1 & 4 & 2 \\
4 & 4 & 8 & 8 \\ 2 & 2 & 8 & 0 \end{pmatrix}  \]
is isomorphic to $\Z \oplus \Z/8$. The first component is multiplication by
$\sm{2 & 2 & -1 & -1}$, while the second component can be taken to be
multiplication by $\sm{0 & 0 & 1 & 2}$. In particular $TH^4(M) = \Z/8$.
We can take $\sm{0 \\ 0 \\ 1 \\ -1}$ as a generator for the cotorsion.
It has $\frac18 \sm{0 \\ 0 \\ -1 \\ 2}$ as a preimage under $\wh W$, so
the self-linking of the corresponding generator of $TH^4(M)$ is $\frac58$.

The image $(\half \bar c_+, -\bar c_-) = \sm{9 & 9 & -28 & -24}$ in the
free part of $\coker \wh W$ is $88$, so the greatest divisor of
$p(M)$ modulo torsion is 8. On the other hand, the image in $\Z/8$ is
$28+24 = 4 \mmod 8$.

The parameter $\max\{ d_o : sp(M) \textrm{ is divisible by } s^2d_o \textrm{ for some } s \in \mathbb{N}\}$ is identified by
Wilkens \cite[Conjecture p.\,548]{wilkens74} as key to computing the inertia
group of a 2-connected 7-manifold. In this example, we have $d_o = 4$,
so Wilkens' conjecture predicts that the inertia group of $M$ is the full
group of homotopy 7-spheres $\Theta_7 \cong \Z/28$; equivalently that
the topological manifold underlying $M$ has a unique class of smooth structure.
However, it turns out that this isomorphism class $(H^4(M), p(M))$ is
an exceptional case where Wilkens' prediction is incorrect. There are in fact
two inequivalent smooth structures on this manifold, see
\cite[Theorem 1.10 \& Example 5.2]{7class}.

Of the 19 configuration angles $\alpha^-_1, \ldots, \alpha^-_{19}$, two take
the value $\pi$ while the other 17 are 0. Thus $\bar \nu = -33$.
\end{ex}

Finally, here is a \fpif-matching using a configuration where there is a
non-trivial intersection between the polarising lattices.
 
\begin{ex}
\label{ex:intersection}
The involution blocks in Example \ref{ex:mm2}$_8$ (from double cover of one-point blow-up of $\PP^3$) have polarising lattice $N_+ = \sm{4 & 4 \\ 4 & 2}$,
while Example \ref{ex:deg7} (from blow-up of $\PP^3$ in an elliptic curve of
degree 7) has $N_- = \sm{4 & 9 \\ 9 & 8}$.
Let $A_+ := \cvec{1}{8} \in N_+$ and $A_- := \cvec{3}{1} \in N_-$. The
respective orthogonal complements are spanned by $B_+ := \cvec{5}{-9} \in N_+$
and $B_- := \cvec{-5}{3}$.
We have $A_+^2 = 196$, $A_-^2 = 98$ and $B_+^2 = B_-^2 = -98$.
We can thus view $N_+$ as the overlattice extending $\sm{196 & 0 \\ 0 & -98}$
by adjoining $\frac{1}{49}(9A_+ + 8B_+)$, and $N_-$ as
extending $\sm{98 & 0 \\ 0 & -98}$ by $\frac{1}{14}(5A_- + 3B_-)$.
Now extending
\[ \begin{pmatrix}
196 & 0 & 98 \\
0 & -98 & 0 \\
98 & 0 & 98
\end{pmatrix} \]
by $\frac{1}{49}\sm{9 \\ 8 \\ 0}$ and $\frac{1}{14}\sm{0 \\ 3 \\ 5}$ defines
an integral lattice $W$ that contains $N_+$ and $N_-$, and can be used to
define a configuration where $A_\pm \in N_\pm^{\frac{\pi}{4}}$. Alternatively,
$W$ can be described as the quotient of the degenerate lattice
\[ \begin{pmatrix}
4 & 4 & 5 & 3 \\
4 & 2 & 2 & 4 \\
5 & 2 & 4 & 9 \\
3 & 4 & 9 & 8
\end{pmatrix} \]
by its kernel. In any case, although this configuration does not have pure
angle \fpif, because $N_\pm$ is spanned by $N_\pm^{\frac{\pi}{4}}$ and
$N_+ \cap N_-$ it is still the case that $N_\pm = \Lambda_\pm$.
Therefore we do not need any genericity results beyond Proposition
\ref{prop:generic_fano} in order to produce matchings with this configuration
from Theorem \ref{thm:matching}.

The resulting \fpif-twisted connected sums have $\pi_2 M \cong H^2(M) \cong N_+ \cap N_- \cong \Z$, so are not 2-connected. From \eqref{eq:b3} we get
$b_3(M) = 23 - 2 - 2 + 1 + 16 + 12 + 1 = 49$. The cokernel of
\[ \wh W = \begin{pmatrix}
2 & 2 & 5 & 3 \\
2 & 1 & 2 & 4 \\
5 & 2 & 4 & 9 \\
3 & 4 & 9 & 8
\end{pmatrix} \]
is mapped isomorphically to $\Z$ by $\sm{1 & 8 & -3 & -1}$, so $H^4(M)$
is torsion-free. The image of
$(\half \bar c_+, -\bar c_-) = \sm{10 & 9 & -22 & -32}$ is 186, so the greatest
divisor of $p(M)$ is $d = \gcd(186,24) = 6$.

All 19 of the configuration angles $\alpha^-_1 = \cdots = \alpha^-_{19} = 0$,
so $\bar \nu = -39$ by Theorem \ref{thm:nubar}.
\end{ex}

\pagebreak

\xtcstable

\subsection{\fpis-matchings}

Finally we give 11 examples of \fpis-matchings (all but one with pure
angle \fpis).

\begin{ex}
\label{ex:1+1pi6}
We can search for \fpif-matchings of rank 1 involution blocks
similarly to how we found the \fpif-matchings of rank 1 blocks in Table
\ref{table:rk1pi4xtcs}. If the generators of the polarising lattices square to
$n_+$ and $n_-$ respectively, then there is a \fpis-configuration if and only
if $3n_+n_-$ is a square integer. 
Among the 7 rank 1 involution blocks in Table \ref{table:invblocks}, there
are 6 such (ordered) pairs.

\footnotetext{In Example \ref{ex:inertia}, the greatest divisor of $p(M)$
\emph{modulo torsion} is 8.}

For instance, we can match the involution blocks from
Examples \ref{ex:rk1index2}$_1$ and \ref{ex:rk1index2}$_3$ at
pure angle~$\frac{\pi}{6}$ using the matrix
\[ W = \begin{pmatrix} 2 & 3 \\3 & 6\end{pmatrix} . \]
Then
\[ b_3(M) = 23 - 1 - 1 + 18 + 46 +1 = 86 . \]
Since $N_+$ is 2-elementary, $H^4(M)$ is torsion-free.
Further we have that $\twothird \pi_- N_+ \cap N_- = N_-$, so
${\pi_+^* \bar c_+ + \half \bar c_-} \in (\twothird \pi_- N_+ \cap N_-)^*
= N_-^* \cong \Z$ corresponds to $26 \cdot \frac{3}{2} + \half 30 = 54$.
Hence the greatest divisor of $p(M)$ is~6.

If we swap the roles of those two blocks, then we instead define the
configuration by 
\[ W = \begin{pmatrix} 6 & 3 \\3 & 2\end{pmatrix} . \]
$\pi_+$ of the generator of $N_-$ is half the generator of $N_+$, so in
particular $N_+ + 2\pi_+ N_- = N_+$. Its discriminant group is
$\Delta = \ZZ/6\ZZ$, so Corollary \ref{cor:pi6pure} gives
$\Tor H^4(M) \cong \Delta/T_2\Delta \cong \ZZ/3\ZZ$, and that
a generator has self-linking $\frac13$.

We still have $\twothird \pi_-N_+ \cap N_- = N_-$. In terms of the generator
for $N_-^*$ we have
$\pi_+^* \bar c_+ + \half \bar c_- =
\third 26 \cdot \frac{3}{2} + \half 30 = 28$, so the greatest divisor of
$p(M)$ is 4.

Similarly we get two examples by matching Example \ref{ex:rk1index2}$_3$ to
Example \ref{ex:rk1index2}$_4$ and another two by matching it to Example
\ref{ex:p1rsmooth}$_1$, with invariants as listed in Table \ref{table:xtcs}.
\end{ex}

\begin{ex}
\label{ex:pi6main}
Match the involution block from Example \ref{ex:deg_sextic} with itself at pure
angle $\vartheta = \frac{\pi}{6}$ using the matrix
\[ W = \begin{pmatrix}
2 & 2 & 2 & 1 \\
2 & 0 & 1 & 2 \\
2 & 1 & 2 & 2 \\
1 & 2 & 2 & 0 \end{pmatrix} . \]
\[ b_3(M) = 23 - 2.2 + 2.44 + 2 = 109 .\]
Since $N_+$ has 2-elementary discriminant, $H^4(M)$ is torsion-free, and
to determine the greatest divisor of $p(M)$ we just have to consider
${\pi_+^* \bar c_+ + \half \bar c_-} \in N_+^*$.
We compute 
\[ \pi_+^* \bar c_+ + \half \bar c_- = 
\rvec{26}{24}\sm{\frac12 & 1 \\ \frac12 & \frac12} + \rvec{13}{12}
= \rvec{38}{50} \]
so the greatest divisor of $p(M)$ is 2.

According to row labelled 86 in \cite[Table 3]{g2m}, there are 3 rectangular
TCS of rank 1 Fanos with the same classifying invariants.
\end{ex}

\begin{ex}
\label{ex:pi6main8}
Example \ref{ex:deg_sextic} with itself at pure angle \fpis{} again, but
this time with configuration

\[ W = \begin{pmatrix}
2 & 2 & 2 & 1 \\
2 & 0 & 3 & 0 \\
2 & 3 & 2 & 2 \\
1 & 0 & 2 & 0 \end{pmatrix} . \]
The topological calculations are the same as in the previous example, except
that $p(M)$ is determined from
\[ \pi_+^* \bar c_+ + \half \bar c_- = 
\rvec{26}{24}\sm{\frac32 &  0 \\ -\frac12 & \frac12 } + \rvec{13}{12}
= \rvec{40}{24} \]
leading to $d = 8$ instead. So different pure angle matchings of the same
pair of blocks can lead to non-diffeomorphic extra-twisted connected sums.
\end{ex}

\begin{ex}
\label{ex:pi6deg5}
Match the involution blocks from Examples \ref{ex:deg_sextic} and
\ref{ex:DQf}$_5$ at pure angle $\thet = \frac{\pi}{6}$ using the configuration
defined by
\[ W = \begin{pmatrix}
2 & 2 & 4 & 2 \\
2 & 0 & 3 & 0 \\
4 & 3 & 10 & 4 \\
2 & 0 & 4 & 0 \end{pmatrix} . \]
$b_3(M) = 21 + 44 + 12 = 77$.
$N_+$ is 2-elementary, so $H^4(M)$ is torsion-free
$\pi_+^* \bar c_+ + \half \bar c_- =
\rvec{26}{24}\sm{\frac32 & 0 \\ \frac12 & 1} + \rvec{17}{6} = \rvec{68}{30}$,
with greatest divisor 2.
\end{ex}

\begin{ex}
\label{ex:pi6deg5flop}
We can match the involution blocks from Examples \ref{ex:deg_sextic} and
\ref{ex:Dcon}$_5$
with a configuration defined by
\[ W = \begin{pmatrix}
2 & 2 & 4 & 2 \\
2 & 0 & 3 & 3 \\
4 & 3 & 10 & 6 \\
2 & 3 & 6 & 2 \end{pmatrix} . \]
In fact, instead of applying Theorem \ref{thm:matching} directly,
we can obtain the matchings with this prescribed configuration from the
matchings in Example \ref{ex:pi6deg5}. This relies on the fact that
Example \ref{ex:Dcon}$_5$ is a flop of Example \ref{ex:DQf}$_5$, and the
lattice $W$ defining the configuration here is isometric to the configuration
lattice from Example \ref{ex:pi6deg5}. Therefore, for any $\frac{\pi}{6}$-matching $\hkr : \kd_+ \to \kd_-$ of blocks $Z_+$ from Example \ref{ex:deg_sextic}
and $Z_-$ from Example \ref{ex:DQf}$_5$ as in Example \ref{ex:pi6deg5},
flopping $Z_-$ yields a building block $\wh Z_-$ in the family of Example
\ref{ex:Dcon}$_5$ with the same anticanonical divisor $\kd_-$, so that
$\hkr$ is a $\frac{\pi}{6}$-matching of $Z_+$ and $\wh Z_-$.
Thus the $\frac{\pi}{6}$-twisted connected sums from this example and Example
\ref{ex:pi6deg5} can be regarded as being related by a ``$G_2$ conifold
transition'' of the kind discussed in \cite[\S 8]{g2m}.

Flopping does not change the cohomology groups, so just like
in the previous example we find that $b_3(M) = 21 + 44 + 12 = 77$,
and $H^4(M)$ is torsion-free. On the other hand
$\pi_+^* \bar c_+ + \half \bar c_- =
\rvec{26}{24}\sm{\frac32 & \frac32 \\ \frac12 & -\frac12} + \rvec{17}{9}
= \rvec{68}{36}$,
so the greatest divisor of $p(M)$ is 4 in this example.
\end{ex}

Finally we consider a matching that is not at pure angle \fpis.

\begin{ex}
\label{ex:pi6t7}
Match Examples \ref{ex:Dcon}$_2$ and \ref{ex:spec1}$^1_6$ using
\[ W =
\begin{pmatrix}
4 & 6 & 5\\
6 & 2 & 4 \\
5 & 4 & 6
\end{pmatrix} \]
Letting $A_+ = \cvec{1}{1} \in N_+$ and $H_-$ be the generator of $N_-$, we find
$\pi_-A_+ = \frac{3}{2} H_-$ and $\pi_+ H_- = \half A_+$, so
$A_+ \in N_+^{\frac{\pi}{6}}$ and $N_- = N_-^\frac{\pi}{6}$.
Thus \eqref{eq:amps} holds.
$\Lambda_-$ is spanned by $N_-$ and $B_+ := \cvec{4}{-5}$, so
\[ \Lambda_- \cong \begin{pmatrix} 6 & 0 \\ 0 & -126 \end{pmatrix} . \]
The family of blocks from Example \ref{ex:spec1}$^1_6$ is
$(\Lambda_-, H_-\bbrp)$-generic by Proposition \ref{prop:cubics}, so
Theorem \ref{thm:matching} yields a matching with the prescribed
configuration.

$b_3(M) = 23 -1 -2 + 6 + 18 +1 = 45$. The image of
\[ \wh W =
\begin{pmatrix}
2 & 3 & 5\\
3 & 1 & 4 \\
5 & 4 & 9
\end{pmatrix} \]
is an index 7 sublattice of the kernel of $\sm{1 & 1 & -1} : \Z^3 \to \Z$, so
$TH^4(M) \cong \Z/7$. The image of
$(\half \bar c_+, -\half \bar c_-) =  (14, 9, -15)$ in $\Z$ is 38,
so the greatest divisor of $p(M)$ modulo torsion $d = \gcd(38,24) = 2$.
As this is coprime to the order of the torsion, $p(M)$ can have no interesting
torsion component.

The data we have computed so far is enough to show that this \fpis-twisted
connected sum is diffeomorphic to Example \ref{ex:pi4t7}, but to determine the
orientedness of the diffeomorphism we also need to determine the
torsion-linking form.
The cotorsion of $\wh W$ is generated by $\sm{1 \\ 1 \\ 2}$. That has
$\frac17\sm{1 \\ 0 \\ 1}$ as a preimage under $\wh W$, so the corresponding
generator of $TH^4(M)$ has torsion self-linking~$\frac{3}{7}$.
As 3 is not a quadratic residue mod 7, another choice of generator has
self-linking $\frac{-1}{7}$.
Thus the diffeomorphism between this \fpis-twisted connected sum and
the one from Example \ref{ex:pi4t7} is orientation-preserving.
\end{ex}

\pagebreak

\bibliographystyle{amsinitial}
\bibliography{g2geom}

\end{document}